\documentclass[11pt]{article}
\usepackage{xcolor}
\usepackage[centering,marginparwidth=0.5cm]{geometry}  

\usepackage[most]{tcolorbox}
            
\usepackage{setspace}
\usepackage[title]{appendix}
\usepackage[round]{natbib}
\usepackage{subcaption}
\usepackage{amssymb}
\usepackage{amsmath}
\usepackage{bbm}
\usepackage{mathtools}
\usepackage{multicol}
\usepackage{graphicx}
\usepackage{array}
\usepackage{bm}
\usepackage{soul}
\usepackage{dirtytalk}
\usepackage{makecell}
\usepackage{caption}
\usepackage[noend]{algpseudocode}
\usepackage{bm}

\usepackage{accents}

\def\LPGA{{\texttt{LPGA}}}



\usepackage{algorithm,algpseudocode}

\makeatletter
\newcommand\fs@boxedtopcap{\def\@fs@cfont{\bfseries}\let\@fs@capt\floatc@plain
	\def\@fs@pre{\setbox\@currbox\vbox{\hbadness10000
			\moveleft3.4pt\vbox{\advance\hsize by6.8pt
				\hrule \hbox to\hsize{\vrule\kern3pt
					\vbox{\kern3pt\box\@currbox\kern3pt}\kern3pt\vrule}\hrule}}}%
	\def\@fs@mid{\kern2pt}%
	\def\@fs@post{}\let\@fs@iftopcapt\iftrue}
\makeatother

\floatstyle{ruled}
\restylefloat{algorithm}

\usepackage{threeparttable}

\usepackage{enumitem}
\setlist[enumerate]{noitemsep, topsep=1pt}
\setlist[itemize]{noitemsep, topsep=1pt}

\usepackage{extarrows}
\usepackage{diagbox}

\usepackage{anysize}
\usepackage{latexsym}
\usepackage{amsthm}
\usepackage{epsfig}

\usepackage[T1]{fontenc}
\usepackage{mathrsfs}
\usepackage{float}
\usepackage{setspace}
\usepackage{color}
\definecolor{lawngreen}{RGB}{0,250,154}

\usepackage{titletoc}

\usepackage{color}
\usepackage{dsfont}
\usepackage{booktabs}
\definecolor{darkblue}{rgb}{0.0,0.0,0.5}
\definecolor{winered}{rgb}{0.5,0,0}
\definecolor{deeppink}{RGB}{255,20,147}
\usepackage[bookmarks,colorlinks,breaklinks]{hyperref}  %
\hypersetup{linkcolor=winered,citecolor=winered,filecolor=winered,urlcolor=winered} %
 
\allowdisplaybreaks[1]

\graphicspath{{./Figures/}}

\usepackage{etoolbox}
\newcommand{\zerodisplayskips}{%
  \setlength{\abovedisplayskip}{5pt}%
  \setlength{\belowdisplayskip}{5pt}%
  \setlength{\abovedisplayshortskip}{5pt}%
  \setlength{\belowdisplayshortskip}{5pt}}
\appto{\normalsize}{\zerodisplayskips}
\appto{\small}{\zerodisplayskips}
\appto{\footnotesize}{\zerodisplayskips}

\usepackage{titlesec}

\titlespacing*{\subsection}
{0pt}        %
{5pt}        %
{1pt}        %

\titlespacing*{\section}
{0pt}        %
{8pt}       %
{3pt}        %

\titlespacing*{\subsubsection}
{0pt}        %
{3pt}        %
{1pt}        %

\titlespacing*{\paragraph}
{0pt}        %
{3pt}        %
{1em}        %

\usepackage{lipsum}
\makeatletter
\patchcmd{\@addmarginpar}{\ifodd\c@page}{\ifodd\c@page\@tempcnta\m@ne}{}{}
\makeatother

\usepackage{thmtools}

\usepackage{authblk}

\theoremstyle{definition}
\newtheorem{theorem}{Theorem}
\newtheorem{definition}{Definition}
\newtheorem{lemma}{Lemma}
\newtheorem{claim}{Claim}

\newtheorem{corollary}{Corollary}

\newtheorem*{informalqptas}{Theorem~\ref{thm:qptas} (Informal)}
\newtheorem*{informalcf}{Theorems~\ref{theorem:greedy} and~\ref{thm:LPGA_approx_ratio} (Informal)}

\newtheorem{proposition}{Proposition}

\def\dd{\downdownarrows}
\def\truncate{\textrm{tr}}

\newcommand{\eps}{\epsilon}

\def\Ebb{\mathbb{E}}

\def\Dcal{\mathcal{D}}

\def\Acal{\mathcal{A}}

\def\Scal{\mathcal{S}}

\def\Rcal{\mathcal{R}}

\def\Bcal{\mathcal{B}}

\def\STT{\texttt{S}}

\def\greedy{\texttt{G}}

\def\Fcal{\mathcal{F}}

\def\Ccal{\mathcal{C}}

\newcommand{\Pp}{\mathbb P}
\newcommand{\E}{\mathbb E}
\newcommand{\OPT}{\mathrm{OPT}}
\newcommand{\ALG}{\mathrm{ALG}}
\newcommand{\LP}{\mathrm{LP}}
\newcommand{\Zp}{\mathbb Z_+}

\def\phcomments{1}
\newcounter{note}[section]
\renewcommand{\thenote}{\thesection.\arabic{note}}
\newcommand{\noteR}[2]{\refstepcounter{note}\marginpar{\tiny\bf \textcolor{red}{#1~\thenote}}
$\ll${\sf \textcolor{red}{#1~\thenote:}} 
{\color{red}{#2}}$\gg$}

\newcommand{\noteP}[2]{\refstepcounter{note}\marginpar{\tiny\bf \textcolor{deeppink}{#1~\thenote}}
$\ll${\sf \textcolor{deeppink}{#1~\thenote:}} 
{\color{deeppink}{#2}}$\gg$}

\newcommand{\noteO}[2]{\refstepcounter{note}\marginpar{\tiny\bf \textcolor{orange}{#1~\thenote}}
$\ll${\sf \textcolor{orange}{#1~\thenote:}} 
{\color{orange}{#2}}$\gg$}

\newcommand{\noteS}[2]{\refstepcounter{note}\marginpar{\tiny\bf \textcolor{blue}{#1~\thenote}}
$\ll${\sf \textcolor{orange}{#1~\thenote:}} 
{\color{blue}{#2}}$\gg$}

\ifnum\phcomments=1
\newcommand{\chris}[1]{\noteP{Chris}{#1}}
\newcommand{\zhenyu}[1]{\noteR{Zhenyu}{#1}}
\newcommand{\joline}[1]{\noteO{JU}{#1}}
\newcommand{\hs}[1]{\noteS{HS}{#1}}
\else
\newcommand{\chris}[1]{}
\newcommand{\zhenyu}[1]{}
\newcommand{\joline}[1]{}
\newcommand{\hs}[1]{}
\fi

\usepackage[nameinlink]{cleveref}
\crefname{assumption}{Assumption}{Assumptions}
\crefname{lemma}{Lemma}{Lemmas}
\crefname{theorem}{Theorem}{Theorems}
\crefname{corollary}{Corollary}{Corollaries}
\crefname{proposition}{Proposition}{Propositions}
\crefname{claim}{Claim}{Claims}
\crefname{procedure}{Procedure}{Procedures}
\crefname{algorithm}{Algorithm}{Algorithms}
\crefname{figure}{Figure}{Figures}
\crefname{remark}{Remark}{Remarks}
\crefname{section}{Section}{Sections}
\crefname{appendix}{Appendix}{Appendices}
\crefname{procedure}{Procedure}{Procedures}
\crefname{example}{Example}{Examples}
\crefname{table}{Table}{Tables}
\crefname{equation}{}{}
\crefname{enumi}{}{}
\crefname{notation}{Notation}{Notations}
\AtBeginEnvironment{appendices}{\crefalias{section}{appendix}}
\AtBeginEnvironment{appendices}{\crefalias{subsection}{appendix}}



\begin{document}

\title{\Large Dynamic Service Recommendation with Congestion-Dependent Joining: Near-Optimal and Constant-Factor Approximation Algorithms}

\author{Yi-Chun Akchen$^\diamond$, Sena Asl{\i} Bozkurt$^\diamond$, Chen-An Lin$^\star$}
\affil{\footnotesize $^\diamond$School of Management, University College London\\$^\star$Mitch Daniels School of Business, Purdue University\\
{\footnotesize \texttt{yi-chun.akchen@ucl.ac.uk}, \texttt{sena.bozkurt.25@ucl.ac.uk}, \texttt{lin1800@purdue.edu}}
}
\date{}
\maketitle

\vspace{-2.0cm}

\begin{abstract}
Modern service platforms often show customers real-time congestion information, such as anticipated waiting times, before they decide whether to use a service. This creates an intertemporal tradeoff in service recommendation: directing a customer to a service may generate immediate value, but the resulting congestion can make that service less attractive to future customers. We study this tradeoff in a finite-horizon problem in which a platform dynamically recommends one of multiple services, each represented by a queue, and customers' joining probabilities decrease with congestion.

We show that optimal recommendations generally depend on the joint congestion state and cannot be reduced to indices that depend only on each service's own state. Hence, exact dynamic programming suffers from a curse of dimensionality: its state space grows exponentially in the number of services. We then develop approximation algorithms that retain progressively less of this joint state. First, we construct a quasi-polynomial-time approximation scheme that achieves a $(1-\epsilon)$-approximation for any $\epsilon>0$. It rounds the model primitives and truncates the relevant queue lengths, which allows it to compress the joint state into a histogram over service types. We develop coupling arguments that bound the revenue loss from rounding and truncation, even though their effects propagate through the evolution of congestion. Second, under general monotone joining probabilities and service capacities, we study policies that retain only \mbox{per-service} information. We show that a greedy policy, which uses only each service's current queue length, is a $1/2$-approximation. This factor is tight, because the policy ignores the congestion externality that a customer imposes on later customers. We then propose an LP-guided algorithm that plans each service's queue-length distribution over the horizon. It accounts for this externality and achieves a $(1-1/e)$-approximation. Finally, we show that monotone joining, whereby customers are less likely to join a more congested service, is what makes the problem tractable: with nonempty initial queues, both constant-factor guarantees continue to hold under monotonicity, whereas without it the problem is NP-hard to approximate within any constant factor.

\end{abstract}

\section{Introduction}
\label{sec:intro}

Waiting is an inherent part of many service experiences. Increasingly, modern service platforms provide real-time congestion information, such as anticipated waiting times, to customers \emph{before} they decide whether to use a service. Food-delivery platforms display estimated delivery times for restaurants \citep{xie2024strategic}; healthcare systems provide information about waiting times for hospitals and treatment providers \citep{dong2019impact}; and ride-hailing platforms display estimated pickup times to passengers before they request a ride \citep{yu2022delay}. As customers arrive over time and congestion evolves, such information can directly affect customers' joining behavior: a service that is attractive when immediately available may become substantially less appealing when the anticipated waiting time is long \citep{guo2007analysis,akcsin2017impact}.

Such congestion-dependent joining behavior is an important consideration for service platforms when deciding which services to present, recommend, or promote. For example, Uber Eats describes incorporating operational congestion into its restaurant-recommendation decisions, including ``down-ranking restaurants with longer pick-up times,'' which can also lead to longer delivery times \citep{liu2018food}. Importantly, customers ultimately retain the decision of whether to join a recommended service after observing its anticipated waiting time \citep{economou2021impact,snitkovsky2026foresee}. Each customer's decision to join, in turn, affects the congestion experienced by subsequent customers, creating a \emph{congestion externality}. Thus, recommendation decisions and customer joining behavior jointly shape how congestion evolves over time.

This dynamic interaction between congestion and operational decisions, such as which service to recommend, distinguishes service recommendation from traditional product recommendation, inventory planning, and resource allocation. In these settings, product features that determine customer preferences, such as quality and brand, are typically treated as \emph{exogenous}: decisions made for one customer do not change how subsequent customers value the same product. In service recommendation, by contrast, waiting time is an \emph{endogenous} service attribute: recommending a service can increase its congestion and make it less attractive to future customers, while withholding recommendations allows congestion to dissipate and can restore its attractiveness. This interaction raises a natural dynamic service recommendation problem: \emph{how should a platform dynamically recommend services when customers' willingness to join depends on the congestion generated by past recommendations?}

In this paper, we introduce a stylized model of dynamic service recommendation that isolates this interaction. A platform operates multiple services whose congestion levels evolve stochastically over time. Customers arrive randomly, and the platform may recommend a service to each arriving customer. Upon receiving a recommendation, the customer decides whether to join the recommended service, with a joining probability that decreases with the service's current congestion. Customers who join increase congestion, while service completions reduce it over time. The platform dynamically observes congestion across services and decides which service to recommend based on each service's underlying attractiveness, revenue potential, and evolving congestion. Its objective is to maximize the expected revenue generated by customers who ultimately complete service within the planning horizon.

Unlike classical routing and load-balancing problems, in which the decision maker typically controls where an arriving customer is assigned \citep{winston1977optimality,ahn2013flexible}, the platform here controls only the \emph{recommendation}. Congestion consequently affects the system through two channels: it reflects the number of existing customers awaiting service while simultaneously influencing new customers' willingness to join. Platform recommendations, customer joining behavior, and queueing dynamics are therefore intrinsically intertwined.

\subsection{Model Description}
\label{subsec:model}

\paragraph{Notation.} Throughout the paper, we use $\Zp := \{0,1,2,\ldots \}$ to denote the set of nonnegative integers. We write $[m]:=\{1,\ldots,m\}$ and $[m]_0:=[m]\cup\{0\}$ for any nonnegative integer $m$.

\paragraph{Platform and services.}
We consider a platform that operates \(n\) parallel services, indexed by \(j\in[n] \), over a finite horizon of \(T\) periods, indexed by \(t\in[T]\). Let \(Q_j(t)\in\mathbb Z_+\) denote the queue length of service \(j\) at the beginning of period \(t\). We assume that all services are initially empty, i.e., \(Q_j(1)=0\) for each \(j\in[n]\). Customers are served first-in-first-out (FIFO), so a customer who joins service $j$ with $q$ customers ahead waits for these $q$ customers to complete service.

Each service serves at most one customer per period, with a service completion occurring with probability \(\xi_j\) whenever the service is nonempty. Specifically, for each service~$j$ and period~$t$, let
\begin{equation}
	\label{eq:service_bernoulli}
	B_{j,t}\sim\operatorname{Bernoulli}(\xi_j),
	\qquad \xi_j\in(0,1],
\end{equation}
indicate whether a service-completion opportunity occurs. We assume that \(\{B_{j,t}\}_{j \in [n],t \in [T]}\) are mutually independent. If service $j$ has a nonempty queue, service for the customer at the head of the queue is completed in period $t$ whenever $B_{j,t}=1$. Once a customer reaches the head of the queue, the remaining service time is therefore geometric with mean \(1/\xi_j\).

\paragraph{Customer arrivals, recommendations, and joining.}
We divide the planning horizon into sufficiently short periods so that at most one customer arrives between two consecutive decision epochs. We allow for non-stationary customer arrivals: at the beginning of each period $t\in[T]$, the platform observes whether such a customer has arrived, which occurs with probability \mbox{$\lambda_t\in(0,1]$}, independently across periods and independently of the service-completion opportunities. If a customer arrives, the platform may recommend one service; it may also choose not to make a recommendation. 

Suppose that service \(j\) is recommended when its queue length is \(q\). Motivated by the observation that customers become less willing to join as their anticipated waiting time increases \citep{guo2007analysis,akcsin2017impact}, we model the customer's joining probability using the following exponential specification \citep{whitt1999improving,inoue2023estimating}
\begin{equation}
	a_j(q)
	:=
	p_j\gamma^{-q/\xi_j},
	\qquad p_j\in[0,1],\quad \gamma>1.
	\label{eq:joining-probability}
\end{equation}
Here, \(p_j\) represents the baseline attractiveness of service \(j\) in the absence of congestion and may capture factors such as service quality, ratings, price, or brand. The factor \(\gamma^{-q/\xi_j}\) captures the reduction in willingness to join due to congestion. Since each incumbent customer requires an expected service time of \(1/\xi_j\), the quantity \(q/\xi_j\) is the expected time until a newly joining customer reaches the head of the queue. Thus, \(\gamma^{-q/\xi_j}\) can be interpreted as an exponential discount for expected waiting time, with \(1/\gamma\in(0,1)\) representing the per-unit-time discount factor on the customer's willingness to join. If the discount is instead applied to the expected time until service \emph{completion}, then $a_j(q)=p_j\gamma^{-(q+1)/\xi_j}=(p_j\gamma^{-1/\xi_j})\,\gamma^{-q/\xi_j}$, which has the same form with baseline attractiveness $p_j\gamma^{-1/\xi_j}$.

An important feature of~\eqref{eq:joining-probability} is that customers become less likely to join as congestion increases:
\begin{equation}
	a_j(q+1)\le a_j(q),
	\qquad q\in\mathbb Z_+.
	\label{eq:monotone-a}
\end{equation}

\paragraph{Within-period dynamics.}
Events within each period occur in the following order: a customer (if any) first arrives, the platform makes a recommendation, the customer decides whether to join, and service completions may then occur. A customer who does not join leaves the system immediately. Let \(A_{j,t}\in\{0,1\}\) indicate whether a customer joins service $j$ in time period $t$. For notational convenience, define $Q_j^+(t):=Q_j(t)+A_{j,t}
$ as the post-arrival queue length. By the end of the period, the queue then evolves according to
\begin{equation}
	Q_j(t+1)
	=
	\bigl(Q_j^+(t)-B_{j,t}\bigr)_+
	=
	\bigl(Q_j(t)+ A_{j,t} -B_{j,t}\bigr)_+.
	\label{eq:true-dynamics}
\end{equation}
where $(x)_+:=\max\{x,0\}$ is the positive-part operator. Thus, a customer who joins an empty queue may be served immediately in the same period.

The platform observes the queue lengths and all past arrival and joining outcomes when making its recommendation decisions. It need not directly observe the latent service opportunity \(B_{j,t}\). When \(Q_j^+(t)>0\), its realization is revealed by the subsequent queue transition; when \(Q_j^+(t)=0\), its realization has no effect on the system state.

\paragraph{Policies and objective.}
The platform maximizes its net revenue, defined as the revenues collected from customers who join minus the refunds issued to customers who remain unserved at the end of the horizon. Specifically, a customer who joins service \(j\) generates an initial revenue \(r_j\). However, if the customer remains unserved at the end of the horizon, the platform provides a full refund of \(r_j\). This payment structure prevents the platform from benefiting from recommending services that attract customers but fail to serve them within the relevant horizon.

A recommendation policy \(\pi\) specifies, in each period \(t\), a recommendation decision based on the information available to the platform up to that period. We allow \(\pi\) to be randomized and history dependent. Accordingly, under policy \(\pi\), the platform's expected net revenue is
\begin{equation}
	\Rcal(\pi)
	:=
	\E_\pi\left[
	\sum_{t=1}^{T}\sum_{j=1}^n r_j A_{j,t}
	-
	\sum_{j=1}^n r_j Q_j(T+1)
	\right].
	\label{eq:original_objective}
\end{equation}
Here, the expectation is taken over all system randomness and any randomization of $\pi$, and $Q_j(T+1)$ is the number of customers remaining at service $j$ at the end of the horizon. Both $A_{j,t}$ and $Q_j(T+1)$ depend on $\pi$ through its recommendations. The first term in~\eqref{eq:original_objective} is the revenue collected from customers who join, and the second is the refund issued to customers who remain unserved at the end of the horizon. Because all queues start empty, every customer counted in $Q_j(T+1)$ joined during the horizon, so each refund cancels a payment collected earlier. Hence, \eqref{eq:original_objective} equals the expected revenue from customers who join and complete service by the end of the horizon, as if customers paid only upon service completion.

\paragraph{The dynamic service recommendation problem.}
Let \(\Pi\) denote the class of all recommendation policies. The \emph{dynamic service recommendation} problem (\ref{problem:DSR-abstract}) is
\begin{equation}
	\label{problem:DSR-abstract}
	\tag{\texttt{DSR}}
	\OPT
	:=
	\sup_{\pi\in\Pi} \Rcal(\pi).
\end{equation}

\subsection{Contributions}
\label{subsec:contributions}

In Section~\ref{sec:preliminaries}, we first cast~\ref{problem:DSR-abstract} as a Markov decision process (MDP) and characterize its optimal policy through the Bellman equation. A central challenge is the high dimensionality of the underlying state space. In particular, Proposition~\ref{prop:non-indexability} shows that, in general, an optimal policy cannot be represented by indices that depend only on the local state of each service. Instead, the optimal recommendation may depend on the queue lengths of all services jointly, that is, on the vector
\[
\bm{Q}(t)=\bigl(Q_1(t),\ldots,Q_n(t)\bigr).
\]
Since \(Q_j(t)\in [t-1]_0 \), the number of possible joint queue states can grow as \(O(T^n)\), making a direct implementation of the Bellman recursion computationally prohibitive when the number of services is large or the horizon is long. To address this curse of dimensionality, we develop approximation algorithms that differ in how much of the joint congestion state they retain, and hence in their approximation quality, computational efficiency, and modeling generality. 

A common route in queueing theory is to study such systems in an asymptotic regime, such as a fluid or heavy-traffic limit, and to show that a policy is asymptotically optimal. We take a different route that does not rely on a limiting regime: each of our algorithms comes with a non-asymptotic guarantee relative to the optimal expected net revenue, which holds for any number of services and any finite horizon.

\paragraph{A quasi-polynomial-time approximation scheme (Section~\ref{sec:QPTAS}).}
Our first contribution is a quasi-polynomial-time approximation scheme (QPTAS) for~\ref{problem:DSR-abstract}. To simplify the runtime expression, we use the standard soft-$O$ notation \(\widetilde O_\epsilon(f(n,T))\) to suppress polynomial dependence on \(1/\epsilon\) and polylogarithmic factors in \(n\), \(T\), and \(1/\epsilon\); that is, $
\widetilde O_\epsilon(f(n,T))
=
O\!\left(
\operatorname{poly}(1/\epsilon)\,
f(n,T)\,
\log^k(nT/\epsilon)
\right)$ for some constant \(k\ge0\). Informally, our main result is as follows.

\begin{informalqptas}
	\label{thm:qptas-informal}
	For any accuracy level $\epsilon \in \left( 0,\frac{1}{10} \right)$, there is a deterministic
	\(O\!\left(
	T\cdot n^{\widetilde O_\epsilon(1)}
	\right)\)
	-time algorithm that computes a policy whose expected net revenue is at least \((1-\epsilon)  \OPT\).
\end{informalqptas}

In contrast to the $O(T^n)$ state-space dependence of the direct Bellman recursion, this running time is quasi-polynomial in both $n$ and $T$ for any fixed $\epsilon$. The QPTAS is built around two main ideas. First, we show that the model parameters can be rounded and the range of relevant queue lengths can be truncated while incurring an arbitrarily small loss in net revenue. Establishing these approximations is nontrivial because perturbing a model parameter changes not only the revenue associated with a given state, but also the stochastic evolution of the queues, so its effect can compound over the horizon. We address this difficulty through probabilistic coupling arguments that relate policies and queue trajectories across the original and modified systems.

Second, these reductions allow us to substantially compress the state space by grouping services with identical rounded parameters into the same \emph{type}. Rather than tracking the queue length of every individual service, the compressed MDP records a histogram of queue lengths for each type, counting the number of services at each relevant congestion level. We can then solve the compressed MDP exactly by backward induction in quasi-polynomial time and transfer the resulting policy back to the original system with performance guarantee.

\paragraph{Polynomial-time constant-factor approximations (Section~\ref{sec:LP}).}
Our second set of contributions concerns polynomial-time policies under more general primitives: arbitrary joining probabilities satisfying the monotonicity condition~\eqref{eq:monotone-a} and general discrete service-capacity distributions. We first analyze a natural greedy policy, which recommends the service with the highest expected revenue from the current customer, adjusted for the probabilities that the customer joins and completes service by the end of the horizon. We then develop an LP-guided algorithm ({\LPGA}). It solves an occupancy-measure linear program that plans the queue-length distribution of each service separately, thereby avoiding the joint state space, and uses the LP solution to decide which services are eligible for recommendation to each arriving customer.

\begin{informalcf}
	The greedy policy earns at least $\frac{1}{2}\OPT$, and this factor is tight. {\LPGA} runs in polynomial time and earns at least $\left(1-\frac{1}{e}\right)\OPT$.
\end{informalcf}

The gap between the two guarantees reflects what the greedy policy ignores. It accounts for the congestion that a customer faces, but not the congestion externality that the customer imposes on later customers. In the instances that make its guarantee tight, it keeps recommending a valuable service while that service is busy. This keeps the service unattractive to future customers. Many of them thus decline the recommendation, although they could have been served elsewhere. The LP accounts for this externality, and {\LPGA} applies the greedy rule only among the services that the LP activates. The main difficulty in the analysis is that each recommendation changes future joining probabilities, so the actual system departs from the LP's plan. We address this with a virtual system that follows the LP's plan exactly. We show that the virtual queues are never shorter than the actual ones, so by monotone joining, customers are more likely to join in the actual system. A correlation-gap inequality \citep{yan2011mechanism,chekuri2011submodular,feldman2016online} bounds the remaining loss.

\paragraph{How much of the congestion state is needed.}
Proposition~\ref{prop:non-indexability} shows that optimal decisions can depend on the joint congestion state across services, and the Bellman recursion tracks all of it. Our results show how much of this state is needed for each level of performance. The QPTAS keeps the joint state but compresses it. After an arbitrarily small loss in revenue, services with approximately the same parameters become interchangeable, so individual service identities can be replaced by the number of services of each type at each of a bounded range of congestion levels. Thus, while the high dimensionality of the exact MDP is structural, much of it becomes dispensable when near-optimal rather than exact decisions are sought.

The polynomial-time policies go further and give up the joint state altogether, working service by service. {\LPGA} retains each service's marginal queue-length distribution over the horizon. This suffices to plan for future congestion: the LP accounts for how current recommendations shift each service's future queue lengths. It also relaxes the interaction between services to a constraint that holds only in expectation, so the LP may plan to recommend several services to the same customer. Choosing among them costs at most a factor of $1-1/e$. The greedy policy retains only each service's current queue length. By Proposition~\ref{prop:non-indexability}, no such local-state index policy is optimal in general, and the greedy policy's tight factor of $1/2$ quantifies what it can lose: without a view of future congestion, it cannot account for the congestion externality of its recommendations. 

Across the three algorithms, retaining less of the congestion state yields weaker guarantees. In exchange, the greedy policy and {\LPGA} run in polynomial time and apply to substantially more general primitives.

\paragraph{The role of customers' congestion-sensitive joining (Section~\ref{subsec:impatience}).}
A key feature of our model is that the joining probability $a_j(q)$ is assumed to be nonincreasing in the queue length $q$, capturing customers' congestion aversion (Condition~\eqref{eq:monotone-a}). This behavioral assumption is widely adopted in the queueing literature \citep{hassin2016rational,economou2021impact} and plays a fundamental role in establishing all of our approximation guarantees. In the QPTAS, it provides the stochastic stability needed to establish near-optimality under successive approximations; for the greedy policy, it ensures that a shorter queue is always at least as valuable as a longer one; and in {\LPGA}, it makes customers more likely to join the actual queues than the virtual ones.

More fundamentally, monotonicity marks a sharp tractability boundary for congestion-dependent joining behavior. To formalize this insight, in Section~\ref{subsec:impatience} we consider a natural extension of the model that allows nonempty initial queues. Under monotone joining behavior, the generalized problem remains approximable in polynomial time: both the greedy policy and {\LPGA} retain their guarantees. In sharp contrast, in the same model, once monotonicity is removed, we show that, for any fixed $\epsilon>0$, the resulting \ref{problem:DSR-abstract} problem is NP-hard to approximate within a factor of $O(n^{1-\epsilon})$. Thus, monotonicity alone separates the tractable and intractable regimes, showing that customers' congestion aversion is not merely a behavioral assumption but a source of tractability.

\subsection{Related Work}
\label{subsec:literature}

\paragraph{Service systems with congestion-dependent choice.}
Our work is closely related to research on assortment optimization in service systems where customer choices depend on congestion. \citet{guo2025dynamic} study dynamic assortment optimization in a \emph{single}-queue service system where customer choices depend on real-time congestion, characterizing the optimal control through a Hamilton--Jacobi--Bellman equation and deriving structural properties of the optimal assortment. Our model instead features \(n\) parallel queues that evolve separately, which results in a high-dimensional state space and motivates our focus on approximation algorithms. \citet{wang2024anticipated} and \citet{liu2024assortment} study \emph{static} assortment and pricing decisions in service systems, incorporating anticipated congestion and customer decision times, respectively, whereas we study a \emph{dynamic} problem in which the platform chooses which service to recommend as the congestion of multiple queues changes over time. More broadly, our work relates to models of observable queues, where customers decide whether to join or balk based on congestion information and anticipated waiting times \citep{naor1969regulation,hassin2016rational,economou2021impact,lin2023wait,snitkovsky2026foresee}, and to routing and load balancing across parallel queues, including join-the-shortest-queue and power-of-two-choice policies \citep{winston1977optimality,vvedenskaya1996queueing,jhunjhunwala2026join}. In routing models, the controller assigns each arriving customer to a queue. In ours, the platform recommends a service and the customer then decides whether to join, so the recommendation affects the arrival process without directly controlling it. Methodologically, this literature typically characterizes optimal or asymptotically optimal policies, often for a single queue or in a fluid or heavy-traffic limit. Our focus is complementary: we study a finite-horizon system with many separately evolving queues, in which exact optimization is intractable, and develop approximation algorithms with non-asymptotic guarantees.

\paragraph{Approximation schemes through stability.}
Our QPTAS is closely related to \citet{segev2025near}, who develops a QPTAS for maximizing the expected reward from adaptively serving stochastically departing customers, building on the model introduced by \citet{cygan2013catch}; \citet{xu2024sequential} consider a related sequential selection problem with stochastic expirations and evaluation times. The QPTAS of \citet{segev2025near} establishes stability of the adaptive optimum under successive approximations of the problem primitives, using probabilistic coupling arguments to control the resulting loss, and \citet{aouad2023stability} exploit a related notion of stochastic stability for choice-based inventory planning under the MNL model.

Our QPTAS follows the framework of successive approximations of \citet{segev2025near}, but its main challenges differ. In \citet{segev2025near}, the dynamics are driven by exogenous departure probabilities, and the central difficulty is that rounding errors in these probabilities compound over time, which requires realigning the original and modified systems. In our model, the dynamics depend on the state: a discrepancy in queue lengths between the two systems changes their subsequent joining probabilities. Our couplings keep the two systems ordered by congestion, which prevents such discrepancies from compounding. Monotone joining then turns this ordering into an advantage: a customer is more likely to join a shorter queue, which allows us to compare the revenues of the two systems. The most challenging step of our analysis is rounding the joining probabilities, which have no counterpart in \citet{segev2025near}.

\paragraph{LP-based and myopic policies for dynamic stochastic systems.}
Our constant-factor results relate to online assortment and allocation with reusable resources, where allocated resources become temporarily unavailable and return after random usage durations \citep{rusmevichientong2020dynamic,feng2024near,huang2024basic,goyal2025asymptotically,hu2025constant}. This literature similarly confronts high-dimensional stochastic resource states and develops LP-based and myopic approximation policies \citep{gong2022online}. In these models, past allocation decisions affect future resource availability, while customer demand is governed by exogenous choice primitives. In our setting, services remain available as their queues grow, but past recommendations generate congestion that endogenously changes future customers' willingness to join: the system state affects future demand rather than resource availability. Our LP-based approach is also methodologically related to dynamic and stationary matching, which develops LP-based approximation algorithms for stochastic matching systems \citep{aouad2022dynamic,pollner2024improved,macrury2025random}. In particular, contention-resolution methods translate fractional matching probabilities into feasible online decisions \citep{pollner2024improved,macrury2025random}, and recent work develops adaptive approximation algorithms that explicitly leverage queue-length information \citep{amanihamedani2024improved,amanihamedani2025adaptive}. In our setting, queue lengths not only describe the system state but also affect future customers' joining probabilities, so the actual system departs from the LP's plan. {\LPGA} addresses this feedback through a virtual system that follows the plan and is coupled to the actual system through monotone joining.

\section{Preliminaries}
\label{sec:preliminaries}

In this section, we formulate \ref{problem:DSR-abstract} as a finite-horizon MDP and show that its optimal decisions cannot, in general, be reduced to indices that depend only on each service's own state.

\subsection{Dynamic Programming Formulation}
\label{subsec:DP}

Recall that $\bm{Q}(t)=(Q_1(t),\ldots,Q_n(t))\in\Zp^n$ denotes the vector of queue lengths of all services at the beginning of period $t$. Because arrivals and service opportunities are independent across periods and their distributions depend on the past only through the current queue lengths and the current recommendation, $(t,\bm{Q}(t))$ is a sufficient state for \ref{problem:DSR-abstract}.

\paragraph{The Bellman equation.} Let \(V(t,\bm{q})\) denote the maximum expected net revenue from periods \(t,\ldots,T\), conditional on \( \bm Q(t)= \bm q\). The terminal value is $
V(T+1, \bm q) = -\sum_{j=1}^n r_jq_j$. For convenience, we use $\bm e_j$ to denote the $j$-th unit vector for $j\in[n]$. We also let $j=0$ represent the action of making no recommendation, and define $\bm e_0:= \bm 0$, $a_0(q):=0$ for any $q \in \Zp$, and $r_0:=0$. The Bellman equation, for $t\in[T]$, is then
\begin{align*}
	V(t,\bm q)
	= 
	\max_{j\in[n]_0}
	\Bigg\{
	\E \big[
	V & \left(t+1,(\bm q-\bm B_t)_+\right)
	\big] 
	+ \\ &
	\lambda_t \cdot a_j(q_j) \cdot
	\bigg(
	r_j
	+
	\E \big[
	V \left(t+1,(\bm q+ \bm e_j- \bm B_t)_+\right)  
	- V \left(t+1,(\bm q- \bm B_t)_+\right)
	\big]
	\bigg)
	\Bigg\},
\end{align*}
where $\bm B_t=(B_{1,t},\ldots,B_{n,t})$, the positive part $(\cdot)_+$ is taken componentwise, and the expectations are with respect to the independent service opportunities $B_{j,t}\sim\operatorname{Bernoulli}(\xi_j)$. The optimal objective value $\OPT$ of \ref{problem:DSR-abstract} is simply $V(1,\bm 0)$.

\paragraph{Markov policies and state-space complexity.} Because the horizon and action space are finite, an action attaining the maximum in the Bellman equation exists at every state, so an optimal deterministic Markov policy exists, even though $\OPT$ is defined over all randomized and history-dependent policies. However, the Bellman recursion runs over the joint state of all services. Starting from empty queues, $Q_j(t)\in\{0,\ldots,t-1\}$ for each $j$, so the recursion faces up to $t^n=O(T^n)$ joint states in period $t$. A natural question is whether this high-dimensional state can be circumvented by assigning each service an index based only on its own congestion. We next show that, in general, it cannot.

\subsection{Non-Indexability of Optimal Recommendations}
\label{subsec:nonindexability}

We first consider the class of local-state index policies \citep{whittle1988restless,nino2002dynamic}. Such a policy assigns each service \(j\) a score
\begin{equation}
	\mathcal{I}_{j,t}\bigl(Q_j(t)\bigr)
	\label{eq:index-def}
\end{equation}
in period \(t\), where \(\mathcal{I}_{j,t}(\cdot)\) may depend arbitrarily on time and on all primitive parameters of the problem instance, but on the current system state only through the local queue length \(Q_j(t)\). The policy recommends a service with the largest index, breaking ties by a fixed priority order over services.
The following proposition shows that an index depending only on the local queue length can rule out optimality, even in a small instance.

\begin{proposition}
	\label{prop:non-indexability}
	There exists an instance of the model in Section~\ref{subsec:model} with three services in which no local-state index policy is optimal.
\end{proposition}

The proof, provided in Appendix~\ref{subsec:proof-non-indexability}, constructs a three-service example. In two states that are reachable under an optimal policy, services 1 and 2 have the same queue lengths. Yet the optimal choice between them differs, because service~3 is more congested in one state than in the other. The mechanism is an \emph{intertemporal substitution effect}. Service~2 offers a high revenue. It is therefore valuable both for the current customer and for future customers. Service~1 offers a lower revenue but is currently empty. When service~3 is lightly congested, the platform can still recommend it to future customers, so it can use the high-revenue service~2 now. When service~3 is heavily congested, it is of little use for future customers. Keeping service~2 available for them then becomes more valuable, and the platform recommends service~1 now. Thus, the congestion of service~3 changes the opportunity cost of recommending service~2 now. A local-state index cannot capture this cross-service dependence.

Proposition~\ref{prop:non-indexability} shows that the high dimensionality of the Bellman equation, which comes from the joint state, is structural. The remainder of the paper asks how much of the joint state can be given up. The QPTAS in Section~\ref{sec:QPTAS} rounds the model primitives so that many services become identical. It then replaces the joint state with a histogram that counts how many services of each type have each queue length, losing an arbitrarily small fraction of the optimal net revenue. Although local-state index policies are not optimal in general, they are simple to implement and interpret, so we revisit them in Section~\ref{subsec:greedy}. We consider a natural index policy, the greedy policy, which maximizes the current customer's expected contribution to the net revenue. We show that it earns at least half of the optimal net revenue. We then propose an LP-guided algorithm that improves this guarantee to $1-1/e$ by solving an LP that plans each service's queue-length distribution over the horizon.

\section{Quasi-Polynomial-Time Approximation Scheme}
\label{sec:QPTAS}

In this section, we develop a quasi-polynomial-time approximation scheme (QPTAS) for~\ref{problem:DSR-abstract}. We first state mild regularity conditions used throughout this section and then present the main theorem and the key steps underlying the approximation scheme.

\subsection{Regularity Conditions}
\label{subsec:QPTAS-assumptions}

We impose mild regularity conditions throughout this section. First, we treat the congestion-sensitivity parameter $\gamma>1$ as a fixed constant independent of $n$ and $T$. This is natural because $\gamma$ describes customers' sensitivity to waiting, a behavioral characteristic that should not depend on how many services the platform manages or how long it operates. Second, there exist fixed constants \(c_\lambda,c_p,c_\xi>0\) such that
\begin{equation}
	\lambda_{\max}:=\max_{t\in[T]}\lambda_t\ge T^{-c_\lambda}, \qquad
	p_j\ge T^{-c_p},
	\qquad\text{and}\qquad
	\xi_j\ge T^{-c_\xi},
	\qquad j\in[n].
	\label{eq:qptas-regularity}
\end{equation}
These conditions rule out degenerate regimes in which the relevant probabilities are exponentially small relative to the planning horizon. If \(\lambda_{\max}\) were exponentially small in \(T\), then the expected total number of arrivals over the entire horizon, \(\sum_{t\in[T]}\lambda_t\), would also be exponentially small in \(T\). Similarly, if \(p_j\) were exponentially small in \(T\), then even recommending service \(j\) in every period would yield only an exponentially small expected number of joining customers over the horizon. Finally, if \(\xi_j\) were exponentially small in \(T\), then the expected service time for a customer, \(1/\xi_j\), would be exponentially larger than the horizon. Thus, if the condition on $p_j$ or on $\xi_j$ fails, service $j$ is degenerate: over the entire horizon, it is expected to attract almost no customers or to serve almost none.

For ease of exposition, we present the algorithm and analysis for $c_\lambda=c_p=c_\xi=1$, so that $\lambda_{\max}\ge 1/T$ and $p_j,\xi_j\ge 1/T$ for each $j\in[n]$. This choice is made only to simplify the presentation. The arguments extend directly to any fixed $c_\lambda,c_p,c_\xi>0$, with these constants affecting only the constants in the exponent of the quasi-polynomial running time.

\subsection{Main Result}

We now present our main result, a QPTAS for~\ref{problem:DSR-abstract}.
\begin{theorem}
	\label{thm:qptas}
	Under the assumptions in Section~\ref{subsec:QPTAS-assumptions}, for every $n,T \in \mathbb{N}$ and  \(\epsilon\in(0,1/10)\), there is a deterministic
	\(
	O\left(
	T \cdot n^{O(\log^5 (T/\epsilon) / \epsilon^3)}
	\right)
	\)
	-time algorithm that computes a policy whose expected net revenue is at least \((1-\epsilon) \cdot \OPT\).
\end{theorem}

For any fixed \(\epsilon>0\), the running time in Theorem~\ref{thm:qptas} is quasi-polynomial in both \(n\) and \(T\), replacing the exponential dependence on \(n\) of the original \(O(T^n)\) state space with $O(T  n^{\widetilde{O}_\epsilon(1)})$. The QPTAS achieves this by showing that, up to an arbitrarily small loss, both the heterogeneity across services and the range of relevant queue lengths can be compressed, which enables a histogram-based representation of the joint queue state. The construction proceeds in two steps.

\paragraph{Step 1: Instance simplification.}
Each service $j$ is characterized by a parameter tuple $(r_j,p_j,\xi_j)$; services with the same tuple form a \emph{service type}, so the original instance may contain as many as $n$ distinct types. We round each parameter down to a geometric grid, so that only $K$ types remain, where $K$ is polylogarithmic in $T/\epsilon$, and we truncate the relevant queue lengths to a range logarithmic in $T/\epsilon$. The main difficulty is that these transformations change not only the immediate revenues but also the joining and service probabilities. Their effects therefore propagate through the stochastic evolution of congestion and the future decisions of a policy, and can compound over the horizon. As a result, the same policy may perform very differently in the original and modified systems, which makes their revenues difficult to compare. We address this challenge with coupling arguments that keep one system always at least as congested as the other. This ordering recurs as we successively round the model primitives and truncate the queue lengths. Each of these modifications, however, poses a difficulty of its own in comparing the revenues of the two systems. We also show that any policy for the modified instance can be transferred back to the original instance without further loss. 

At the end of Step~1, we obtain a modified instance with only $K$ distinct service types and $O(\log(T/\epsilon))$ relevant queue lengths for each type.

\paragraph{Step 2: State-space compression and dynamic programming.}
After Step~1, services of the same type are identical except for their current queue lengths. We exploit this symmetry to replace the queue vector $\bm Q(t)$ with a histogram that records, for each type and each relevant queue length, the number of services at that queue length. The resulting compressed MDP has quasi-polynomially many states and can be solved exactly by backward induction within the runtime of Theorem~\ref{thm:qptas}. Combined with the policy transfer of Step~1, this yields a policy for the original system with expected net revenue at least $(1-O(\epsilon))\cdot\OPT$.

\subsection{Step 1: Rounding and Truncation}

Throughout this subsection, we fix $\epsilon \in (0,\frac{1}{10})$. We construct the modified instance in four steps: we round the revenues (Step~1a) and the service probabilities (Step~1b), truncate the policy class (Step~1c), and round the joining probabilities (Step~1d). Each step loses only a small fraction of the optimal net revenue. Step~1e combines the four steps into the modified problem and shows that any policy for it can be transferred back to the original instance.

\subsubsection{Step 1a: Rounding Revenues}

We begin with the service revenues $r_1,\ldots,r_n$, because rounding them is much simpler than rounding the other model parameters. Under a fixed policy, changing the revenues does not affect the stochastic evolution of the system. It only changes the revenue associated with each realized outcome. We show that the revenues can be rounded so that only $O(\epsilon^{-1}\log(T/\epsilon))$ distinct values remain, at a small loss in the optimal expected net revenue.

\paragraph{Rounded revenues.}
Let $r_{\max}:=\max_{j\in[n]} r_j$, and let $U_{\mathrm{LB}} := \lambda_{\max} \cdot \max_{j \in [n]}  r_j p_j \xi_j$, which is a lower bound on $\OPT$ (Appendix~\ref{subsec:appendix-lower-bound}). Define the geometric grid
\[
\texttt{Grid}_r
:=
\{0\}
\cup
\left\lbrace
r_{\max}(1+\epsilon)^{-m}
:
m=0,1,\ldots,
\left\lceil
\log_{1+\epsilon}
\left(
\frac{T r_{\max}}
{\epsilon U_{\mathrm{LB}}}
\right)
\right\rceil
\right\rbrace.
\]
For each service $j\in[n]$, let $r^{\downarrow}_j:=\max\{z\in\texttt{Grid}_r : z\le r_j\}$. Thus, every revenue above $\epsilon U_{\mathrm{LB}}/T$ is rounded down geometrically to the nearest grid point, and every revenue below this threshold is rounded to zero. Since $\lambda_{\max}\ge 1/T$ and $p_j,\xi_j\ge 1/T$ (Section~\ref{subsec:QPTAS-assumptions}), we have $U_{\mathrm{LB}} \ge r_{\max}/T^3$, so the grid has $O(\log_{1+\epsilon}(T^4/\epsilon)) = O(\epsilon^{-1}\log(T/\epsilon))$ points.

\paragraph{Effect of revenue rounding.}
Let $\Rcal_{(r)}(\pi)$ and $\Rcal_{(r^{\downarrow})}(\pi)$ denote the expected net revenues of a policy $\pi$ under the original and rounded revenue parameters, respectively, and let $\pi^*$ be an optimal policy for the original instance.

\begin{lemma}
	\label{lemma:QPTAS_rounding-rewards}
	For every policy $\pi \in \Pi$, we have $
	(1-\epsilon) \cdot \Rcal_{(r)}(\pi)
	-
	\epsilon \cdot \Rcal_{(r)}(\pi^*)
	\le
	\Rcal_{(r^{\downarrow})}(\pi)
	\le
	\Rcal_{(r)}(\pi)$.
\end{lemma}
The proof follows a rounding argument similar to that of Lemma~2.3 in \citet{segev2025near} and is included in Appendix~\ref{subsec:QPTAS_proof-rounding-r} for completeness. Lemma~\ref{lemma:QPTAS_rounding-rewards} allows us to work with the rounded revenues at only a small additional loss. To see this, let $\tilde{\pi}^*$ be an optimal policy for the rounded instance, and suppose that $\tilde{\pi}$ is an $\alpha$-approximation for that instance. Then
\begin{align*}
	\Rcal_{(r)}(\tilde\pi)
	\ge
	\Rcal_{(r^{\downarrow})}(\tilde\pi)
	\ge
	\alpha\,\Rcal_{(r^{\downarrow})}(\tilde\pi^*)
	\ge
	\alpha\,\Rcal_{(r^{\downarrow})}(\pi^*)
	\ge
	\alpha(1-2\epsilon) \cdot \Rcal_{(r)}(\pi^*),
\end{align*}
where the first and last inequalities follow from Lemma~\ref{lemma:QPTAS_rounding-rewards}, the second from the approximation guarantee of $\tilde\pi$, and the third from the optimality of $\tilde\pi^*$. Hence, an $\alpha$-approximation for the rounded instance yields an $\alpha(1-2\epsilon)$-approximation for the original instance. For the remainder of the analysis, we therefore replace the original revenues by the rounded revenues, with the understanding that this step incurs an additional multiplicative factor of $1-2\epsilon$.

\subsubsection{Step 1b: Rounding the Service Rates}
\label{subsubsec:QPTAS-Step-rounding-xi}

We next turn to the service-completion process. Unlike the revenue rounding in Step~1a, modifying the service probabilities changes the stochastic evolution of congestion. Even under the same policy, the original and rounded systems may follow different queue-length trajectories, which in turn lead to different future recommendations. Controlling the effect of this rounding is therefore substantially more delicate.

A further complication is that $\xi_j$ plays two roles in our model: it governs the service-completion process and also enters the joining probability via $a_j(q)=p_j\gamma^{-q/\xi_j}$ in~\eqref{eq:joining-probability}. To separate these roles temporarily, we write $\xi_j^S$ and $\xi_j^A$ for the parameters of the service-completion and joining processes, respectively, both initially equal to $\xi_j$. In this step, we round only $\xi_j^S$ and keep $\xi_j^A=\xi_j$, so the joining probabilities $a_j(q) = p_j \gamma^{-q/\xi^A_j}$ are unchanged. We round $\xi_j^A$ in Step~1d.

\paragraph{Rounded service parameters.}
Recall from Section~\ref{subsec:QPTAS-assumptions} that $\xi_j^S\in[1/T,1]$ for every $j\in[n]$. Let $\epsilon_\xi \in (0,1)$ be a rounding parameter whose value is specified at the end of Step~1, and define the geometric grid
\[
\texttt{Grid}_\xi
:=
\left\lbrace
\tfrac{1}{T}(1+\epsilon_\xi)^m
:
m=0,1,\ldots,
\left\lceil
\log_{1+\epsilon_\xi}T
\right\rceil
\right\rbrace.
\]
For each service $j\in[n]$, define its rounded-down service parameter by $\xi_j^{S\downarrow}:=\max\{z\in\texttt{Grid}_\xi:z\le\xi_j^S\}$. By construction, $\xi_j^S/(1+\epsilon_\xi)\le\xi_j^{S\downarrow}\le\xi_j^S$ and $|\texttt{Grid}_\xi|=O(\epsilon_\xi^{-1}\log T)$. We refer to the system in which $\{\xi_j^S\}_{j\in[n]}$ is replaced by $\{\xi_j^{S\downarrow}\}_{j\in[n]}$ as the \emph{slow system}, reflecting its slower service. The joining probabilities still use the original parameters $\xi_j^A=\xi_j$, so the term ``slow system'' refers only to the service-completion dynamics.

\paragraph{Effect of service-rate rounding.}
With a slight abuse of notation, let $\Rcal_{(\xi^S)}(\pi)$ and $\Rcal_{(\xi^{S\downarrow})}(\pi)$ denote the expected net revenues of a policy $\pi$ in the unrounded and slow systems, respectively, both evaluated under the rounded revenues $\{r_j^\downarrow\}_{j\in[n]}$. Define
\[
\OPT_{(\xi^S)}
:=
\sup_{\pi\in\Pi} \Rcal_{(\xi^S)}(\pi),
\qquad
\OPT_{(\xi^{S\downarrow})}
:=
\sup_{\pi\in\Pi} \Rcal_{(\xi^{S\downarrow})}(\pi).
\]
The following lemma bounds the loss in the optimal objective value caused by rounding the service-completion parameters. Its proof is in Appendix~\ref{subsec:QPTAS-proof-of-rounding-xi-part1}.

\begin{lemma}
	\label{lemma:QPTAS_rounding-xi-OPT-comparison}
	$\OPT_{(\xi^{S\downarrow})}
	\ge
	(1 - \epsilon_\xi) \cdot
	\OPT_{(\xi^S)}$.
\end{lemma}

Unlike Lemma~\ref{lemma:QPTAS_rounding-rewards}, this lemma cannot be proved by evaluating the same policy in both systems: as discussed above, even under the same policy, the two systems may follow very different queue-length trajectories. Because joining depends on congestion, these differences alter future joining outcomes and can amplify the divergence. We therefore use a coupling argument. Let $\pi^*_{(\xi^S)}$ be an optimal policy for the unrounded system. We construct a policy for the slow system that maintains a virtual copy of the unrounded system and recommends, in the slow system, the service that $\pi^*_{(\xi^S)}$ recommends in the virtual copy. We then couple the two systems by generating the outcomes of the virtual copy from those observed in the slow system and additional randomness. Under this coupling, each queue in the slow system is never shorter than its counterpart in the virtual copy. In particular, whenever a customer is available for service in the virtual copy, a customer is also available in the corresponding queue of the slow system. The remaining loss comes only from the slower service, which costs at most a factor of $1-\epsilon_\xi$.

\subsubsection{Step 1c: Truncation of the Policy Class}

After rounding the revenues $\{r_j\}_{j\in[n]}$ and the service parameters $\{\xi_j^S\}_{j\in[n]}$, the natural next step is to round the parameters that appear in the joining probabilities $a_j(\cdot)$. A direct rounding argument, however, encounters a difficulty. Because $a_j(q)=p_j\gamma^{-q/\xi_j^A}$, the joining probability becomes exponentially small in $T$ when the queue length $q$ is of order $T$. Accurately approximating $a_j(q)$ over this entire range would therefore require an excessively fine discretization.

To address this difficulty, we observe that once the joining probability of a service is sufficiently small, recommending that service contributes only negligibly to the optimal expected net revenue. This motivates restricting attention to policies that stop recommending a service once its queue is sufficiently long. Under this restriction, the joining probabilities encountered when a service is recommended are bounded away from excessively small values, which makes the subsequent rounding step tractable.

\paragraph{Truncated policy class.}
Define $\Lambda(z):=z(1+\epsilon_\xi)\log_\gamma\bigl(T/(z\epsilon)\bigr)$. For each service $j\in[n]$, let $\Lambda_j:=\Lambda(\xi_j^{S\downarrow})$ be a queue-length threshold, defined in terms of the rounded service rate $\xi_j^{S\downarrow}$.

\begin{definition}
	A policy belongs to $\Pi^{\mathrm{tr}}$ if it never recommends service $j$ at any state satisfying $Q_j(t)\ge \Lambda_j$.
\end{definition}

With a slight abuse of notation, throughout this step we let $\Rcal(\pi)$ denote the expected net revenue of policy $\pi$ for the instance obtained after rounding $\{r_j\}_{j\in[n]}$ and $\{\xi_j^S\}_{j\in[n]}$ in Steps~1a and~1b. The following lemma shows that restricting attention to $\Pi^{\mathrm{tr}}$ incurs only a small loss. 

\begin{lemma}
	\label{lem:truncation}
	$\sup_{\pi\in\Pi^{\mathrm{tr}}}\Rcal(\pi)
	\ge
	(1-\epsilon) \cdot
	\sup_{\pi\in\Pi}\Rcal(\pi)$.
\end{lemma}
The proof, provided in Appendix~\ref{subsec:QPTAS_proof-policy-truncation}, shows that each time a recommendation is suppressed by the truncation rule, the resulting expected loss is at most an \(\epsilon/T\) fraction of the optimal expected net revenue. Since at most one recommendation can be suppressed in each period, these losses accumulate to at most an \(\epsilon\) fraction over the entire horizon.

\paragraph{State-space compression.}
Beyond its approximation guarantee, truncation provides a second benefit: it places an effective cap on every queue length and thereby substantially reduces the relevant state space. This immediately implies the following corollary.

\begin{corollary}
	\label{collorary:QPTAS-policy-hard-cap}
	Under every policy in $\Pi^{\mathrm{tr}}$, we have $Q_j(t)\le \lceil\Lambda_j\rceil$
	for all $j\in[n]$ and $t\in[T+1]$.
\end{corollary}

Consequently, the number of queue-length states that need to be tracked for each service $j$ decreases from $O(T)$ to
\[
O(\Lambda_j)
=
O\left(
\xi_j^{S\downarrow}
\log_\gamma\left(
\frac{T}{\xi_j^{S\downarrow}\epsilon}
\right)
\right)
=
O\left(
\log\left(\frac{T}{\epsilon}\right)
\right),
\]
where the second equality follows from $\gamma=\Theta(1)$ and $1/T\le\xi_j^{S\downarrow}\le 1$. This reduction is what allows a sufficiently coarse grid for the joining probabilities in Step~1d, and it is instrumental in constructing the compressed MDP in Step~2.

\subsubsection{Step 1d: Rounding the Joining Probabilities}

The final rounding step concerns the joining probabilities. Recall that the joining probability for service $j$ is $a_j(q)=p_j\gamma^{-q/\xi_j^A}$. By Corollary~\ref{collorary:QPTAS-policy-hard-cap}, we may restrict attention to queue lengths $q\le\lceil\Lambda_j\rceil$, over which the joining probabilities are bounded away from exponentially small values. This restriction allows us to round the parameters of the joining functions so that the resulting joining probabilities remain within a uniform multiplicative factor of their original values.

Our analysis proceeds in two stages. We first show that a uniform multiplicative approximation of the joining probabilities yields a corresponding multiplicative approximation of the optimal objective value. We then round the parameters $p_j$ and $\xi_j^A$ so that the resulting joining functions satisfy this uniform condition.

\paragraph{Effect of rounding the joining probabilities.}
In this step, we restrict attention to policies in the truncated class $\Pi^{\mathrm{tr}}$. Consider rounded-down joining functions
$\{a_j^\downarrow(\cdot)\}_{j\in[n]}$ satisfying
\begin{equation}
	\label{eq:QPTAS_a-rounding-condition}
	\beta a_j(q)
	\le
	a_j^\downarrow(q)
	\le
	a_j(q),
	\qquad
	\forall j\in[n],\quad
	0\le q < \Lambda_j,
\end{equation}
for some $\beta\in(0,1]$. Thus, throughout the truncated state space, rounding decreases each joining probability by at most a factor $\beta$. All other primitives, including the rounded revenues $\{r^\downarrow_j\}_{j\in[n]}$ and the rounded service parameters $\{\xi^{S\downarrow}_j\}_{j\in[n]}$, are held fixed. For a policy $\pi$, let $\Rcal_{(a)}(\pi)$ and $\Rcal_{(a^\downarrow)}(\pi)$ denote its expected net revenue in the unrounded and rounded systems, with $\{a_j(\cdot)\}_{j\in[n]}$ and $\{a_j^\downarrow(\cdot)\}_{j\in[n]}$, respectively. In particular, $\Rcal_{(a)}(\pi)$ is the objective $\Rcal(\pi)$ of Step~1c. Define
\[
\OPT_{(a)}^{\mathrm{tr}}
:=
\sup_{\pi\in\Pi^{\mathrm{tr}}} \Rcal_{(a)}(\pi)
\qquad\text{and}\qquad
\OPT_{(a^\downarrow)}^{\mathrm{tr}}
:=
\sup_{\pi\in\Pi^{\mathrm{tr}}} \Rcal_{(a^\downarrow)}(\pi).
\]
The following lemma shows that the multiplicative approximation in~\eqref{eq:QPTAS_a-rounding-condition} translates into a multiplicative approximation of the optimal objective value.

\begin{lemma}
	\label{lemma:QPTAS_rounding-a-OPT-comparison}
	$\OPT_{(a^\downarrow)}^{\mathrm{tr}}
	\ge
	\beta \cdot \OPT_{(a)}^{\mathrm{tr}}$.
\end{lemma}

The proof of Lemma~\ref{lemma:QPTAS_rounding-a-OPT-comparison}, provided in Appendix~\ref{subsec:QPTAS-proof_rounding-a-OPT-comparison}, requires a more delicate coupling than that used for rounding the service probabilities. Rounding down a service probability has a direct connection to the objective: net revenue is generated by customers who are ultimately served, and the service probability directly governs service completions. By contrast, the joining probability governs whether a customer enters the queue, not whether that customer ultimately generates net revenue. Consequently, even if every relevant joining probability is within a factor $\beta$ of its original value, this pointwise comparison does not immediately imply a comparison of the objective values. Establishing it requires a careful comparison of the customers who join and are ultimately served in the unrounded and rounded systems.

To this end, we introduce an auxiliary system, used only in the analysis, that bridges the unrounded and rounded systems: it retains each customer who joins the unrounded system independently with probability $\beta$. We couple the three systems, namely the unrounded, rounded, and auxiliary systems, so that they follow the same recommendations and share their randomness. We then compare the revenues of the rounded and unrounded systems through the auxiliary system in two steps. First, by Condition~\eqref{eq:QPTAS_a-rounding-condition} and monotone joining, every customer who joins the auxiliary system also joins the rounded system, so the auxiliary system completes no more customers. Second, removing customers can only speed up the service of the remaining ones, so the auxiliary system completes, in expectation, at least a $\beta$ fraction of the customers completed in the unrounded system. See Appendix~\ref{subsec:QPTAS-proof_rounding-a-OPT-comparison} for details.

\paragraph{Rounding $p_j$ and $\xi^A_j$.}
Recall from Section~\ref{subsec:QPTAS-assumptions} that $p_j\in[1/T,1]$. Define the geometric grid
\[
\texttt{Grid}_p
:=
\left\{
\tfrac{1}{T}(1+\epsilon)^m
:
m=0,1,\ldots,
\left\lceil
\log_{1+\epsilon}T
\right\rceil
\right\}.
\]
For each service $j\in[n]$, define $p_j^\downarrow:=\max\{z\in\texttt{Grid}_p:z\le p_j\}$. By construction, $p_j/(1+\epsilon)\le p_j^\downarrow\le p_j$ and $|\texttt{Grid}_p|=O(\epsilon^{-1}\log T)$. We round $\xi_j^A$ in the same way as the service parameter $\xi_j^S$ in Step~1b: using the same grid $\texttt{Grid}_\xi$, we define $\xi_j^{A\downarrow}:=\max\{z\in\texttt{Grid}_\xi:z\le\xi_j^A\}$, so that $\xi_j^A/(1+\epsilon_\xi)\le\xi_j^{A\downarrow}\le\xi_j^A$.

It remains to choose $\epsilon_\xi$ so that Condition~\eqref{eq:QPTAS_a-rounding-condition} holds over all queue lengths retained after truncation. The proof of the following claim is in Appendix~\ref{subsec:QPTAS-choosing-eps-proof}.

\begin{claim}\label{claim:QPTAS-selecting-epsilons}
	Choose $\epsilon_\xi = \epsilon / \left( \ln(T/\epsilon) \right)$. Then Condition~\eqref{eq:QPTAS_a-rounding-condition} holds with $\beta=1-5\epsilon$.
\end{claim}

Recall that $\xi_j^A$ and $\xi_j^S$ are two copies of the same original parameter $\xi_j$. Because they are rounded down on the same grid $\texttt{Grid}_\xi$, their rounded values coincide. Hence, from now on, we write $\xi_j^\downarrow:=\xi_j^{A\downarrow}=\xi_j^{S\downarrow}$ for each $j \in [n]$.

\subsubsection{Step 1e: The Modified Problem and Policy Transfer}
\label{subsubsec:QPTAS-step1-summary}

We conclude Step~1 by combining the implications of Steps~1a--1d. Consider the following rounded and truncated version of~\ref{problem:DSR-abstract}, referred to as the modified DSR problem (\ref{problem:reduced-MDP}):
\begin{align}
	\label{problem:reduced-MDP}
	\tag{\texttt{DSR-M}}
	\overline{\OPT}
	:=
	\sup_{\pi\in\Pi^{\truncate}} \bigg[
	\overline{\Rcal}(\pi) \bigg].
\end{align}
Here, $\overline{\Rcal}$ denotes the objective value under the rounded parameters $\{r_j^\downarrow\}_{j\in[n]}$, $\{p_j^\downarrow\}_{j\in[n]}$, and $\{\xi_j^\downarrow\}_{j\in[n]}$, where the latter two determine the joining probabilities and the service dynamics. The truncated policy class $\Pi^{\truncate}$ is likewise defined with respect to the rounded parameters $\{\xi_j^\downarrow\}_{j\in[n]}$: service $j$ cannot be recommended whenever its current queue length is at least \mbox{$\Lambda_j=\Lambda(\xi_j^{S\downarrow})=\Lambda(\xi_j^\downarrow)$}, with $\Lambda(\cdot)$ as in Step~1c and $\epsilon_\xi=\epsilon/\ln(T/\epsilon)$ as in Claim~\ref{claim:QPTAS-selecting-epsilons}. Thus, both the dynamics and the truncation thresholds of \texttt{DSR-M} are completely specified by the rounded parameters. Lemmas~\ref{lemma:QPTAS_rounding-rewards},
\ref{lemma:QPTAS_rounding-xi-OPT-comparison},
\ref{lem:truncation}, and
\ref{lemma:QPTAS_rounding-a-OPT-comparison}, together with Claim~\ref{claim:QPTAS-selecting-epsilons}, imply that $\overline{\OPT}$ and the optimal objective value $\OPT$ of the original~\ref{problem:DSR-abstract} satisfy
\begin{equation}
	\label{eq:QPTAS-nearly-optimality-of-fully-rounded-instance}
	\overline{\OPT}
	\ge
	(1-2\epsilon)(1-\epsilon_\xi)(1-\epsilon)(1-5\epsilon) \cdot \OPT
	\ge
	(1-9\epsilon) \cdot \OPT,
\end{equation}
where the second inequality uses $\epsilon_\xi\le\epsilon$.

In Step~2, we will show that \texttt{DSR-M} can be solved in quasi-polynomial time. In particular, we will obtain an optimal deterministic Markov policy $\bar\pi^*$ for \texttt{DSR-M}. The following lemma shows that any such policy can be transferred back to the original problem without additional loss in objective value. Its proof is in Appendix~\ref{subsec:QPTAS-proof_policy-transfer}.

\begin{lemma}
	\label{lemma:QPTAS_policy-transfer}
	Let $\bar\pi$ be any deterministic Markov policy feasible for \texttt{DSR-M}. Then we can construct a policy $\pi$ for the original instance such that $\Rcal(\pi)
	\ge
	\overline{\Rcal}(\bar\pi)$, where $\Rcal(\cdot)$ denotes the expected net revenue under the original parameters $\{r_j\}_{j\in[n]}$, $\{p_j\}_{j\in[n]}$, and $\{\xi_j\}_{j\in[n]}$.
\end{lemma}

The proof constructs $\pi$ explicitly: the policy $\pi$ maintains a virtual copy of the \texttt{DSR-M} system operated under $\bar\pi$, and the analysis compares both systems with an intermediate system in which only the service probabilities are rounded. Applying Lemma~\ref{lemma:QPTAS_policy-transfer} to $\bar\pi^*$ and using~\eqref{eq:QPTAS-nearly-optimality-of-fully-rounded-instance} yields
\[
\Rcal(\pi)
\ge
\overline{\Rcal}(\bar\pi^*)
=
\overline{\OPT}
\ge
(1-9\epsilon) \cdot \OPT.
\]
Thus, once an optimal policy for \texttt{DSR-M} is computed in quasi-polynomial time, we obtain a $(1-9\epsilon)$-approximate policy for the original problem without any further loss.

\subsection{Step 2: Compressed MDP}

After Step 1, each service is characterized by rounded parameters drawn from finite geometric grids, and its queue length is restricted by the truncation threshold. We now exploit these properties to compress the state space of \texttt{DSR-M}. The key observation is that services sharing the same rounded parameters, and hence belonging to the same type, are statistically identical. Consequently, rather than tracking the queue length of every individual service, it suffices to track, for each service type, the number of services currently at each possible queue length. We formulate the resulting compressed MDP and show that it can be solved exactly in quasi-polynomial time. We first make a few definitions.

\paragraph{Service types.} After Step 1, service $j$ has parameters $(r^\downarrow_j,p^\downarrow_j,\xi^\downarrow_j)$. Define two services $i$ and $j$ to have the same type if $(r^\downarrow_i,p^\downarrow_i,\xi^\downarrow_i) = (r^\downarrow_j,p^\downarrow_j,\xi^\downarrow_j)$. Therefore, we define a service type $k \in [K]$, which is fully characterized by $(r^\downarrow_k,p^\downarrow_k,\xi^\downarrow_k)$. For notational simplicity in this Step 2, we drop the superscript $\downarrow$, with the understanding that parameters $\{r_j\}_{j\in[n]}$,
$\{p_j\}_{j\in[n]}$, and $
\{\xi_j\}_{j\in[n]}$ have been rounded. Totally, the number of service types is bounded as
\begin{align}
	K \leq \,\, & | \texttt{Grid}_r  | \cdot | \texttt{Grid}_p  | \cdot | \texttt{Grid}_\xi  |  \notag \\ = \,\, & O\left(
	\frac{1}{\epsilon}
	\log\left(\frac{T}{\epsilon}\right)
	\right) \cdot O \left(  \frac{\log T}{\epsilon}  \right) \cdot O \left(  \frac{ \log T }{\epsilon / \ln(T/\epsilon)}  \right) =  O\left(   \frac{ \log^4 \left( {T}/{ \epsilon } \right)  }{\epsilon^3}   \right). \label{eq:number_of_types}
\end{align}

\paragraph{Queue length.} All services of type $k$ share the truncation threshold $\Lambda_k:=\Lambda(\xi_k)$ of Step~1c, and we define
\begin{equation}
	H_k:=\lceil\Lambda_k\rceil=O\left(\log\left(\frac{T}{\epsilon}\right)\right).
	\label{eq:type-queue-length}
\end{equation}
By Corollary~\ref{collorary:QPTAS-policy-hard-cap}, queue lengths of type-$k$ services stay in $[H_k]_0$ under the truncated policy class.

\subsubsection{State and Action}

\paragraph{Histogram state.} For each service type \(k\) and queue length \(q\in [H_k]_0\), we use $N_{k,q}(t)$ to denote the number of type-$k$ services with queue length $q$ at the beginning of time period $t$. The histogram-based compressed state is
\[
\bm N(t) := \bigl(N_{k,q}(t)\bigr)_{k \in [K],q \in [H_k]_0}.
\]
This aggregation loses no information relevant to future decisions. Conditional on their current queue lengths, services belonging to the same rounded type have identical joining and service primitives. Consequently, they are exchangeable, and the distribution of the next state \(\bm N(t+1)\) depends on the current system only through \(\bm N(t)\) and the chosen action.

\paragraph{Action.}
In the compressed MDP, an action either makes no recommendation, denoted by $0$, or selects a type-state pair $(k,q)$ with $N_{k,q}(t)>0$ and recommends one of the services of type $k$ whose current queue length is $q$. Thus, given a compressed state $\bm N(t) = \bm N$, the action set is
\begin{equation}
	\mathcal A(\bm N)
	:=
	\{0\}
	\cup
	\left\{
	(k,q):
	N_{k,q}>0,\;
	q<\Lambda_k
	\right\}.
	\label{eq:compressed-action-set}
\end{equation}
The condition $N_{k,q}>0$ ensures that there exists at least one service of type $k$ whose current queue length is $q$, while $q < \Lambda_k$ enforces the truncated recommendation rule.

\subsubsection{One-Period Transition Law}
\label{subsubsec:compressed-transition}

Fix a period $t$, a state $\bm N(t)=\bm N$, and a nonnull action $a=(k^*,q^*)\in\mathcal A(\bm N)$. The one-period transition law consists of two parts: the transition of the recommended service and the transitions of all other services. We call the latter services \emph{background services}.

\paragraph{Transition of the recommended service.} The recommended service moves from queue length $q^*$ to $(q^*+A^*-B^*)_+$, where $A^*\sim\operatorname{Bernoulli}(\lambda_t a_{k^*}(q^*))$ is its joining indicator and $B^*\sim\operatorname{Bernoulli}(\xi_{k^*})$ is its service opportunity, independently of each other.

\paragraph{Transition of the background services.} Let $\eta_{k,q}:=N_{k,q}-\mathbb I\{(k,q)=(k^*,q^*)\}$ be the number of background services of type $k$ with queue length $q$. No customer joins a background service, so each background service with $q\ge1$ moves to queue length $q-1$ if it receives a service opportunity, which happens independently with probability $\xi_k$; a background service with $q=0$ stays at $0$. Since all these services have the same type and queue length, it suffices to count how many of them move. Specifically, let $Z_{k,q}$ denote the number of background services of type-state pair $(k,q)$ that receive a service opportunity and hence move from queue length $q$ to $q-1$. Then
\[
Z_{k,q}\sim\operatorname{Binomial}\bigl(\eta_{k,q},\,\xi_k\bigr),
\qquad k\in[K],\ q\in[H_k],
\]
and the variables $Z_{k,q}$ are independent of each other and of $A^*$ and $B^*$.

\paragraph{Combined transition.} The next state $\bm N(t+1)$ is a deterministic function of $\bm Z=(Z_{k,q})$, $A^*$, and $B^*$. Hence the one-period transition kernel $\Pp_t(\bm N'\mid\bm N,a)$ is an explicit sum of products of binomial and Bernoulli probabilities, which can be evaluated in quasi-polynomial time (Appendix~\ref{subsec:appendix-compressed-MDP-runtime}). For the null action $a=0$, every service is a background service, so $\eta_{k,q}=N_{k,q}$ for all $k\in[K]$ and $q\in[H_k]_0$. The binomial structure is what makes the kernel efficient to compute: it depends only on how many services of each type-state pair receive a service opportunity, not on which ones.

\subsubsection{Bellman Recursion for \ref{problem:reduced-MDP}}

Let $\overline V(t, \bm N)$ denote the optimal net revenue from periods $t,\ldots,T$, conditional on the compressed state $\bm N$ at the beginning of period $t$. The terminal value is $\overline{V}(T+1, \bm N)=-\sum_{k=1}^K r_k\sum_{q=0}^{H_k}qN_{k,q}$, which accounts for refunds to customers remaining at the end of the horizon. For a nonnull action $a=(k,q)$, let $A^*$ be the joining indicator of the recommended service and define the immediate revenue as $R( \bm N,a):=r_k A^*$, while $R( \bm N,0):=0$. The Bellman recursion is
\begin{equation}
	\overline V(t,\bm N)
	=
	\max_{a\in\mathcal A( \bm N)}
	\Ebb\left[
	R(\bm N,a)+ \overline V(t+1,\bm N(t+1))
	\,\middle|\,
	\bm N(t) = \bm N,a
	\right],
	\qquad t\in[T],
	\label{eq:bellman}
\end{equation}
where $\bm N(t+1)$ follows the transition law of Section~\ref{subsubsec:compressed-transition}. Since the compressed state and action spaces are finite, backward induction in~\eqref{eq:bellman} yields an optimal deterministic Markov policy for the compressed MDP. The backward induction can be carried out in quasi-polynomial time. We summarize the resulting runtime below and defer the state-space, action-space, and transition-complexity calculations to Appendix~\ref{subsec:appendix-compressed-MDP-runtime}. 

\begin{proposition}
	\label{proposition:compressed-MDP-runtime}
	The Bellman equation~\eqref{eq:bellman} of the compressed MDP can be solved exactly in runtime $
	O\!\left(T\cdot n^{O(\log^5(T/\epsilon)/\epsilon^3)}\right)$ by backward induction.
\end{proposition}

\paragraph{Lifting the compressed policy.}
It remains to map the optimal compressed policy back to the original labeled state representation of \texttt{DSR-M}. Given a queue-length vector \(\bm Q(t)\) of \texttt{DSR-M}, we construct the corresponding compressed state \(\bm N(t)\) and apply the compressed policy. If the resulting action is \((k,q)\), we recommend the smallest-index service of type \(k\) whose current queue length is \(q\); such a service exists whenever \((k,q)\) is feasible, since \(N_{k,q}(t)>0\). Because all services represented by the same type-state pair \((k,q)\) have identical rounded parameters and the same current queue length, the immediate reward, the feasible actions, and the distribution of the next compressed state depend on $\bm Q(t)$ only through $\bm N(t)$ and the chosen pair $(k,q)$; in particular, they do not depend on which of these services is recommended. Hence, by backward induction on $t$, the optimal value of \texttt{DSR-M} from period $t$ onward at state $\bm Q$ equals $\overline V(t,\bm N)$, where $\bm N$ is the histogram of $\bm Q$. The lifted policy attains this value, so it is optimal for \texttt{DSR-M}. Together with Proposition~\ref{proposition:compressed-MDP-runtime} and Section~\ref{subsubsec:QPTAS-step1-summary}, this yields a policy for the original instance with expected net revenue at least $(1-9\epsilon)\cdot\OPT$. Running the algorithm with $\epsilon/9$ in place of $\epsilon$ proves Theorem~\ref{thm:qptas}; this changes only the constant in the exponent of the running time.

\section{Constant-Factor Approximation Algorithms}
\label{sec:LP}

In this section, we study \ref{problem:DSR-abstract} under substantially more general primitives: arbitrary monotone joining probabilities and general discrete service-capacity distributions. Unlike the QPTAS, the policies we develop run in polynomial time. We first show that a simple greedy index policy earns at least half of the optimal expected net revenue, and that this factor is tight. We then develop an LP-guided policy that improves the guarantee to $1-1/e$. Finally, we show that both guarantees rest on monotone joining: under monotonicity, both extend to nonempty initial queues, whereas without it the problem is NP-hard to approximate within any constant factor.

\subsection{Model}
\label{subsec:general-model}

We generalize the dynamic service recommendation model in Section~\ref{subsec:model}. First, we no longer require the joining probability $a_j(\cdot)$ to take the exponential form in~\eqref{eq:joining-probability}. Instead, $a_j(\cdot)$ may be any function satisfying the monotonicity condition~\eqref{eq:monotone-a}; that is, customers become less likely to join as the queue length increases. Second, we allow the latent service capacity $B_{j,t}$ to follow an arbitrary nonnegative discrete distribution: for each $j\in[n]$,
\begin{equation}
	\Pp(B_{j,t}=b)=\xi_{j,b},
	\qquad b\in\Zp,
	\label{eq:service-law}
\end{equation}
where $\xi_{j,b} \geq 0$ and $\sum_{b} \xi_{j,b}=1$. The collection $\{B_{j,t} \}_{j \in[n],t\in[T]}$ is mutually independent and independent of all customer-arrival, customer-choice, and platform randomness. We continue to refer to the resulting problem as~\ref{problem:DSR-abstract}, with the understanding that the joining probabilities and service-capacity distributions are generalized as above.

We introduce two quantities associated with the service-capacity distributions that will be useful throughout this section. First, for $q,q'\in\Zp$, define the service \emph{transition kernel}
\begin{equation}
	K_j(q,q')
	:=
	\Pp\!\left( (q-B_{j,t})_+=q'\right)
	  =
	\begin{cases}
		\sum_{b \geq q} \xi_{j,b},  & \text{if } q'=0,\\[1mm]
		\xi_{j,q-q'}, & \text{if } 1\le q' \le q,\\[1mm]
		0, & \text{if } q'>q.
	\end{cases}
\end{equation}
Thus, $K_j(q,q')$ is the probability that a queue of length $q$ transitions to length $q'$ after one period of service. Second, we define the \emph{completion factor}
 \begin{equation}
 	\kappa_{j,q,t}
 	:=
 	\Pp\!\left(
 	\sum_{\tau=t}^{T} B_{j,\tau}\ge q+1
 	\right).
 	\label{eq:kappa}
 \end{equation}
 Under FIFO, if a customer joins service $j$ in period $t$ when the pre-arrival queue length is $q$, then $\kappa_{j,q,t}$ is exactly the probability that this customer completes service by the end of period $T$. It follows immediately that $
 	\kappa_{j,q+1,t}
 	\le
 	\kappa_{j,q,t}$. 
 The completion factors can be computed by the backward
 recursion $\kappa_{j,q,t}
 =
 \sum_{b=0}^{q}\xi_{j,b}\kappa_{j,q-b,t+1}
 +
 \sum_{b\ge q+1}\xi_{j,b}$, with boundary conditions
 $\kappa_{j,q,T+1}=0$ for all $q\ge0$. Computing all completion factors through this recursion takes $O(nT^3)$ time.
 
 \paragraph{Objective function.} The expected net revenue of a policy $\pi$ is still given by~\eqref{eq:original_objective}, where the expectation is now taken under the generalized model primitives. With the completion factors, the objective~\eqref{eq:original_objective} can be written period by period. This alternative form underlies both algorithms in this section. To state it, let $\bar a_{j,t}(q):=\lambda_t\, a_j(q)$ denote the probability that, in period~$t$, a customer arrives and joins service $j$ when $j$ is recommended at queue length $q$.

 \begin{lemma}
 	\label{lemma:objective-2}
 	For $j\in[n]$, $t\in[T]$, and $q\in\Zp$, define the index $I_{j,t}(q) := r_j\, \bar a_{j,t}(q)\, \kappa_{j,q,t}$, and let $I_{0,t}(\cdot):=0$. For any policy $\pi$, its expected net revenue~\eqref{eq:original_objective} can be written as
 	\[
 	\Rcal(\pi) =  \E_\pi\left[
 	\sum_{t=1}^{T}  I_{J^\pi_t,t}\big(Q_{J^\pi_t}(t)\big)
 	\right],
 	\]
 	where $J^\pi_t \in [n]_0$ is the service that $\pi$ recommends in period $t$ if a customer arrives, and $Q_0(t):=0$.
 \end{lemma}

 The index $I_{j,t}(q)$ is the expected revenue that recommending service $j$ at queue length $q$ eventually generates: a customer arrives and joins with probability $\bar a_{j,t}(q)$, then completes service by the end of the horizon with probability $\kappa_{j,q,t}$, and pays $r_j$ only in that case. The refunds in~\eqref{eq:original_objective} are thus accounted for by the completion factors. The proof of Lemma~\ref{lemma:objective-2} is in Appendix~\ref{appendix-subsec:objective-reformulation}.

 \subsection{Greedy Policy}
 \label{subsec:greedy}

By Proposition~\ref{prop:non-indexability}, no local-state index policy is optimal. Such policies nonetheless remain attractive in practice, because they rely only on each service's own queue length. This raises two questions: which index to use, and how much revenue the resulting policy may lose. Lemma~\ref{lemma:objective-2} suggests a natural choice: the index $I_{j,t}(\cdot)$ defined in the lemma. By the lemma, the expected net revenue is the sum, over periods, of the index of the recommended service at its current queue length. Periods interact only through the queue lengths, because the current recommendation changes the queue lengths at which future indices are evaluated. The \emph{greedy policy}, denoted~$\texttt{G}$, ignores this interaction and maximizes the current period's index. In period~$t$, it recommends
\begin{align}
\label{eq:greedy-policy}
\tag{\texttt{Greedy}}
J_t^{\texttt{G}} \in \arg\max_{j \in [n]} \; I_{j,t}(Q_j(t)),
\end{align}
breaking ties by a fixed priority order over services. That is, the greedy policy recommends
the service with the highest completion-adjusted expected revenue from the current customer.
It is the local-state index policy of Section~\ref{subsec:nonindexability} with index
$\mathcal{I}_{j,t}=I_{j,t}$. The following theorem establishes its approximation ratio.

\begin{theorem}
	\label{theorem:greedy}
	For every instance of \ref{problem:DSR-abstract}, the greedy policy satisfies $\Rcal(\texttt{G})  \geq \OPT / 2$.
\end{theorem}
The proof, given in Appendix~\ref{appendix-subsec:proof-of-greedy}, is a charging argument. We
compare two stochastic systems: one operated by the greedy policy and one by an arbitrary policy. We couple the two systems so that they share the same service capacities throughout the horizon, and compare their revenues period by period. In each period, we bound the expected revenue of the arbitrary policy by one of two terms: the index of the greedy policy's recommendation in the same period, or the revenue from a customer already waiting in the greedy policy's queue. Each of these two terms is bounded by the revenue of the greedy policy, which yields the factor of~$2$. This charging structure follows the analysis of the myopic policy for online assortment with reusable resources in \citet{gong2022online}. The difference is that the state of a service here affects how likely a customer is to join, whereas in their model it affects whether a product is available. In our setting, the bound relies on the FIFO service order and on monotone joining. We further show that this factor of $2$ cannot be improved.
\begin{proposition}
	\label{proposition:greedy-tightness}
	There exists a sequence of \ref{problem:DSR-abstract} instances such that ${ \Rcal(\texttt{G})  }/{ \OPT} \rightarrow 1/2$.
\end{proposition}

In the instances behind Proposition~\ref{proposition:greedy-tightness}, detailed in Appendix~\ref{appendix-subsec:greedy-tightness}, the greedy policy ignores the \emph{congestion externality} that a customer imposes on later customers. Each instance has two services. Service~1 is valuable but slow. Service~2 is less valuable but never congested. The greedy policy keeps recommending Service~1 even when it is busy, because a customer who joins is still likely to complete service before the end of the horizon and pay the high revenue of Service~1. A better policy lets Service~1 drain its line until it becomes attractive to later customers again, and recommends it only then. Service~1 still uses all of its capacity, but each recommendation to it is now more likely to lead to a joining customer. The recommendations saved in this way earn revenue at Service~2. In other words, the greedy policy ignores the congestion externality: keeping Service~1 busy makes it less attractive to later customers.

Proposition~\ref{proposition:greedy-tightness} suggests that a better policy must account for how current recommendations affect congestion in later periods. Tracking this exactly requires the full vector of queue lengths, which leads to the curse of dimensionality. In the next subsection, we avoid it by planning each service's queue-length distribution globally over the horizon with a linear program.

\subsection{Marginal Occupancy LP}
\label{subsec:MOLP}

By Lemma~\ref{lemma:objective-2}, the expected net revenue of any policy $\pi$ can be written as
\begin{equation}
	\Rcal(\pi)
	=
	\sum_{t=1}^{T}\sum_{j=1}^{n}\sum_{q=0}^{t-1}
	I_{j,t}(q)\cdot
	\Pp\!\left(Q_j(t)=q,\; J^\pi_t=j\right).
	\label{eq:objective-marginal}
\end{equation}
Thus, a policy's revenue depends on it only through, for each service separately, the probability of recommending that service at each queue length in each period. The LP below takes these probabilities as decision variables and imposes conditions that every policy satisfies.

\paragraph{Variables.}
We call $J^\pi_t$ in Lemma~\ref{lemma:objective-2}, the service that $\pi$ would recommend if a customer arrives in period~$t$, the \emph{contingent recommendation}. Choosing it before observing whether a customer arrives is equivalent to choosing it upon arrival, but is convenient for defining the LP variables. For $j\in[n]$, $t\in[T]$, and $q\in[t-1]_0$, the variable $x_{j,q,t}$ represents the \emph{unconditional} probability that service $j$ has queue length $q$ at the beginning of period $t$ and is the contingent recommendation. Variable $y_{j,q,t}$ represents the marginal probability that service $j$ has queue length $q$. Thus,
\[
x_{j,q,t}:=\Pp\!\left(Q_j(t)=q,\; J^\pi_t=j\right),
\qquad
y_{j,q,t}:=\Pp\!\left(Q_j(t)=q\right).
\]
As we will see, the values of $y_{j,q,t}$ are uniquely determined by the variables $x_{j,q,t}$ through the constraints and the initial condition.

\paragraph{The linear program.} We consider the following linear program
\begin{align}
		Z^{\LP} \, := \, \max_{\bm x \geq \bm 0, \bm y \geq \bm 0} \quad
		& \sum_{t=1}^{T} \sum_{j=1}^n \sum_{q=0}^{t-1}
		I_{j,t}(q)\, x_{j,q,t} \tag{\texttt{MOLP}}
		\label{eq:MOLP}
		\\ 
		\text{s.t. }\quad
		&(1)~~ x_{j,q,t}\le y_{j,q,t},
		&&\forall j  ,\  t ,\ q \in [t-1]_0, \nonumber \\
		&(2)~~ \sum_{j=1}^n \sum_{q=0}^{t-1} x_{j,q,t} \leq 1, &&\forall t, \nonumber \\
		&(3)~~ y_{j,q',t+1} = \sum_{q=0}^{t-1} \bigg[
		\bigl(y_{j,q,t} - \bar{a}_{j,t}(q) \cdot x_{j,q,t}\bigr) \cdot K_j(q,q')
		&& \nonumber \\
		& 
		\qquad \qquad \qquad \qquad \qquad   + \bar{a}_{j,t}(q)\cdot x_{j,q,t} \cdot K_j(q+1,q')
		\bigg],
		&&\forall j , \ t ,\ q' \in [t]_0, \nonumber \\
		&(4)~~ y_{j,0,1}=1, && \forall j, \nonumber \\
		& (5)~~ y_{j,q,1}=0, &&\forall j , \ q \geq 1.
\end{align}
We refer to this linear program as the \emph{marginal occupancy} LP,
\ref{eq:MOLP}, since it tracks the time-varying marginal occupancy
(i.e., marginal queue-state) distribution of each queue separately,
rather than the joint distribution of the queue-length vector. By~\eqref{eq:objective-marginal}, its objective is exactly the per-period form of Lemma~\ref{lemma:objective-2}, evaluated on the marginal probabilities.

\paragraph{Constraints.}
Constraint (1) states that the probability that service $j$ is the contingent recommendation at queue length $q$ cannot exceed the probability that $j$ has queue length $q$. Constraints (4) and (5) state that every queue is empty at the beginning of period~$1$. Constraint (3) is the flow equation of the probability mass. Fix a queue length $q'$ at the beginning of period $t+1$. From queue length $q$ in period $t$, no customer joins with probability $y_{j,q,t} - \bar{a}_{j,t}(q) x_{j,q,t}$, in which case the queue length before service is $q$ and moves to $q'$ with probability $K_j(q,q')$. A customer joins with probability $\bar{a}_{j,t}(q)x_{j,q,t}$, in which case the queue length before service is $q+1$ and moves to $q'$ with probability $K_j(q+1,q')$. Therefore, for every $q' \in [t]_0$,
\begin{equation}
	y_{j,q',t+1}
	=
	\sum_{q=0}^{t-1}
	\bigg[
	\underbrace{\bigl(y_{j,q,t} -  \bar{a}_{j,t}(q) \cdot x_{j,q,t} \bigr) \cdot K_j(q,q')}_{\text{no joining in period $t$}}
	+
	\underbrace{ \bar{a}_{j,t}(q) \cdot x_{j,q,t} \cdot K_j(q+1,q')}_{\text{a joining in period $t$}}
	\bigg].
	\label{eq:flow-kernel}
\end{equation}
Constraint (3) is what the greedy policy lacks. Recommending service $j$ in the current period shifts its probability mass toward longer queues in later periods, where its index is lower. The LP therefore weighs the current index against the congestion the recommendation creates, which is the congestion externality discussed in Section~\ref{subsec:greedy}.

Constraint (2) requires only that, in expectation, at most one service be recommended in period~$t$. However, a policy must recommend at most one service in every state. This is what makes \ref{eq:MOLP} a relaxation.

\paragraph{Upper bound.}  We next establish that \ref{eq:MOLP} provides an upper bound on the optimal expected net revenue of~\ref{problem:DSR-abstract}. The key observation is that any policy can be mapped to a feasible solution of \ref{eq:MOLP} while preserving its expected net revenue. We provide the proof in Appendix~\ref{subsec:LP-upperbound-proof}.

\begin{proposition}
	\label{prop:LP_upperbound}
	Every policy induces a feasible solution to~\ref{eq:MOLP} with the same expected net revenue. Therefore, $Z^{\LP}\geq \OPT$, where $\OPT$ is the optimal objective value of \ref{problem:DSR-abstract}.
\end{proposition}

However, an optimal solution of \ref{eq:MOLP} is not itself a policy. Because Constraint~(2) holds only in expectation, the LP may plan to recommend more than one service to the same customer. The next subsection develops a policy that chooses among them and earns at least $(1-1/e)\,Z^{\LP}$.

\subsection{LP-Guided Algorithm (LPGA)}
\label{subsec:LPGA}

\paragraph{Motivation.}
By Proposition~\ref{prop:LP_upperbound}, $Z^{\LP}$ is an optimistic benchmark. To see what it would take to earn it, let $(\bm x^*,\bm y^*)$ be an optimal solution of~\ref{eq:MOLP} and define
\begin{equation}
	\theta_{j,q,t}
	:=
	\begin{cases}
		x^*_{j,q,t}/y^*_{j,q,t}, & y^*_{j,q,t}>0,\\
		0, & y^*_{j,q,t}=0,
	\end{cases}
	\label{eq:theta}
\end{equation}
so that $\theta_{j,q,t}\in[0,1]$ by Constraint~(1). Consider the following hypothetical system, in which the $n$ queues evolve independently. In each period $t$, a queue $j$ of length $q$ is \emph{activated} with probability $\theta_{j,q,t}$. If it is activated, it receives its own customer, who arrives and joins with probability $\bar a_{j,t}(q)$. The queue is then served with capacity $B_{j,t}$, as before. We write $\widehat Q_j(t)$ for the queue lengths of this system. The flow equation of each queue is exactly Constraint~(3) with $x_{j,q,t}=\theta_{j,q,t}\,y_{j,q,t}$, so $\Pp(\widehat Q_j(t)=q)=y^*_{j,q,t}$ (Lemma~\ref{lem:virtual-marginals}). Since an activated customer at queue length $q$ joins with probability $\bar a_{j,t}(q)$ and then completes service with probability $\kappa_{j,q,t}$, the expected revenue of this system is exactly $Z^{\LP}$:
\[
\sum_{j,t}
	\E\!\left[\theta_{j,\widehat Q_j(t),t}\, I_{j,t}\big(\widehat Q_j(t)\big)\right]
	=
	\sum_{j,t} \sum_{q=0}^{t-1} \Pp\!\left( \widehat Q_j(t) = q \right) \theta_{j,q,t}\, I_{j,t}(q)
	=
	\sum_{j,t}\sum_{q=0}^{t-1} y^*_{j,q,t}\, \theta_{j,q,t}\, I_{j,t}(q)
	=
	Z^{\LP}.
\]
However, this hypothetical system \emph{cannot} be implemented. More than one queue may be activated in the same period, each receiving its own customer, whereas in the actual system at most one customer arrives in each period and receives a single recommendation. This suggests a policy: maintain the hypothetical system, which follows the plan of \ref{eq:MOLP}, as a guide, and recommend to the arriving customer only one of the services it activates.

\paragraph{The algorithm.}
{\LPGA}, given in Algorithm~\ref{alg:lpga}, implements this idea by maintaining a copy of the hypothetical system alongside the actual one. We call this copy the \emph{virtual system}, and its queue lengths $\widehat{\bm Q}(t)$ the \emph{virtual queues}. Each period consists of three steps.
\begin{itemize}
	\item \emph{Activation} (Lines~3--4). As in the hypothetical system, each service $j$ is activated independently with probability $\theta_{j,\widehat Q_j(t),t}$, evaluated at its virtual queue length.
	\item \emph{Selection} (Lines~5--6). Among the activated services, {\LPGA} recommends the one with the largest index $I_{j,t}(Q_j(t))$ at its actual queue length. This is the greedy rule of Section~\ref{subsec:greedy}, restricted to the activated set $S_t$.
	\item \emph{Update} (Lines~7--19). {\LPGA} observes the transition of the actual system to period $t+1$ (Line~7). It then updates the virtual system, reusing the observed outcomes from the actual system where possible and sampling the others. For the selected service, the virtual customer joins only if the actual customer joins, and is then retained with probability $\bar a_{j,t}(\widehat Q_j(t))/\bar a_{j,t}(Q_j(t))$ (Line~10). Every other activated service has no actual customer, so a virtual customer is created and joins with probability $\bar a_{j,t}(\widehat Q_j(t))$ (Line~12), as in the hypothetical system. The virtual queues are served with the same capacities as the actual queues. These capacities are recovered from the observed transition when the actual queue is nonempty, and otherwise resampled from their conditional distribution (Lines~15--18):
	\begin{equation}
		\Pp\!\left(
		\widehat B_{j,t}=b
		\,\middle|\,
		Q_j^+(t),Q_j(t+1)=0
		\right)
		=
		\frac{\xi_{j,b}}{\sum_{b'\ge Q_j^+(t)}\xi_{j,b'}},
		\qquad b\ge Q_j^+(t).
		\label{eq:bhat-conditional}
	\end{equation}
\end{itemize}

\paragraph{Discussion.} Activation uses the virtual queue length, not the actual one, for the following reason. Because {\LPGA} recommends only one of the activated services, the actual queues are generally shorter than the virtual ones, and their distributions differ from $\bm y^*$. Activating with probability $\theta_{j,Q_j(t),t}$ would therefore not reproduce the LP's contributions. The virtual system, in contrast, has exactly the law of the hypothetical system: its queues are independent, with marginals $\bm y^*$. The construction also keeps each virtual queue never shorter than its actual counterpart, \mbox{$Q_j(t)\le\widehat Q_j(t)$} (Lemma~\ref{lem:path-dom}). By monotonicity, $\bar a_{j,t}(Q_j(t))\ge\bar a_{j,t}(\widehat Q_j(t))$, so the retention probability in Line~10 is well defined.

\begin{algorithm}[!t]
	\caption{LP-guided algorithm ({\LPGA})}\label{alg:lpga}
	\begin{algorithmic}[1]
		\State Initialize $Q_j(1) = \widehat{Q}_j(1) = 0$ for all $j \in [n]$.
		\For{$t = 1$ to $T$}
		\State Sample $\zeta_{j,t} \sim\operatorname{Bernoulli}({\theta}_{j,\widehat{Q}_j(t),t})$ independently for each $j \in [n]$. \Comment{{\footnotesize \texttt{Activate}}}
		\State Set $S_t:=\{j\in[n]:\zeta_{j,t}=1\}$.
		\State Set $J_t\in %
		{\arg\max} \left\lbrace I_{j,t}\big(Q_j(t)\big) \, : \, j\in S_t  \cup \{ 0 \} \right\rbrace$. \Comment{{\footnotesize \texttt{Select \& Recommend}}}
		\State Recommend $J_t$ if a customer arrives in period $t$.
		\State Observe $\bm Q(t+1)$ and record the actual joining: \Comment{{\footnotesize \texttt{Observe \& Update}}}
		\[
		\forall j \in [n], \quad A_{j,t} := \mathbb I\{\text{a customer joins service \(j\) in period \(t\)}\}.
		\]
		\For{$j = 1$ to $n$} %
		\If{$j = J_t$} 
		\State $\widehat A_{j,t} = \begin{cases}
			A_{j,t},
			&\text{with probability } \bar{a}_{j,t}(\widehat{Q}_j(t)) / \bar{a}_{j,t}({Q}_j(t)) \text{ if } \bar{a}_{j,t}({Q}_j(t))>0,\\[1mm]
			0,
			&\text{otherwise}.
		\end{cases}$
		\ElsIf {$j \in S_t \backslash \{ J_t \}$} 
		\State Sample $\widehat A_{j,t}
		\sim
		\operatorname{Bernoulli}\!\left(\bar a_{j,t}(\widehat{Q}_j(t))\right)$  independently. 
		\Else
		\State Set $\widehat{A}_{j,t} = 0$.
		\EndIf
		\If{$Q_j(t+1) > 0$}
		\State Set $\widehat{B}_{j,t} = Q_j(t) + A_{j,t} - Q_j(t+1)$.
		\Else
		\State Sample $\widehat{B}_{j,t}$ independently according to the conditional distribution~\eqref{eq:bhat-conditional}.
		\EndIf
		\State  Set $\widehat Q_j(t+1)
		=
		\bigl(
		\widehat Q_j(t)
		+\widehat A_{j,t}
		-\widehat B_{j,t}
		\bigr)_+$.
		\EndFor
		\EndFor
	\end{algorithmic}
\end{algorithm}

\paragraph{Performance guarantee.}
\label{subsec:LP-main-theorem}
We solve~\ref{eq:MOLP} once, before the horizon begins. It has $O(nT^2)$ variables and constraints, so {\LPGA} runs in polynomial time.

\begin{theorem}
	\label{thm:LPGA_approx_ratio}
	{\LPGA} is well defined, and its expected net revenue satisfies $\Rcal(\LPGA)\ge(1-1/e)\,\OPT$.
\end{theorem}

The proof, given in Appendix~\ref{subsec:proof_of_LP_performance}, compares {\LPGA} with $Z^{\LP}$ rather than with $\OPT$ directly, and then applies Proposition~\ref{prop:LP_upperbound}. The main difficulty is that the actual system departs from the LP's plan. As discussed above, the actual queues are generally shorter than the virtual ones. Since joining depends on congestion, the actual system faces different joining probabilities than those the LP planned for. We resolve this difficulty by comparing the actual system with the virtual system, using three properties: (i) the virtual queue lengths have exactly the planned marginals $\bm y^*$; (ii) the virtual queues are independent across services; and (iii) since joining probabilities and completion factors are nonincreasing in the queue length, each activated service has an index in the actual system no smaller than in the virtual one. By~(i), the virtual system earns $Z^{\LP}$, and by~(iii), {\LPGA} earns at least the virtual value of the service it recommends. Hence, relative to $Z^{\LP}$, {\LPGA} loses revenue only when several services are activated in the same period but only one can be recommended. To bound this loss, we use a correlation-gap inequality \citep{yan2011mechanism,chekuri2011submodular,feldman2016online}. By~(i) and~(ii), it shows that recommending the best activated service retains at least a $(1-1/e)$ fraction of the expected total value of all activated services.

\paragraph{Remarks.}
{\LPGA} combines a global plan with a local rule. Its selection step is the greedy rule of Section~\ref{subsec:greedy}. The LP enters only through activation, which determines the services among which the greedy rule chooses. The activation probabilities come from a solution that satisfies the flow constraints~(3), so they account for how the current recommendation shifts future queue-length distributions, which the greedy index ignores. This raises the worst-case guarantee from $1/2$ (Theorem~\ref{theorem:greedy} and Proposition~\ref{proposition:greedy-tightness}) to $1-1/e$. Unlike the greedy policy, {\LPGA} is not a local-state index policy: it is randomized, and its activation of service $j$ depends on the virtual queue, whose evolution depends on past recommendations.

\subsection{The Role of Monotone Joining}
\label{subsec:impatience}
The approximation guarantees of the greedy policy and {\LPGA} (Theorems~\ref{theorem:greedy} and~\ref{thm:LPGA_approx_ratio}) both rely on the monotonicity condition~\eqref{eq:monotone-a}, which captures customers' congestion-sensitive behavior: a service becomes less attractive as its queue grows. Both analyses compare queues of different lengths, and monotonicity makes the shorter queue more valuable. For the greedy policy, the comparison is with the queues of an arbitrary policy. For {\LPGA}, it is between each actual queue and its virtual counterpart, which is never shorter. Monotonicity also underlies the QPTAS of Section~\ref{sec:QPTAS}, where it is used in proving the approximation bounds in Steps~1c--1e.

We now show that monotonicity is more than a convenience for the analysis. With it, the problem remains approximable under a natural extension of the model; without it, the problem is NP-hard to approximate within any constant factor.

\paragraph{Nonempty initial queues.}
We extend the model of Section~\ref{subsec:general-model} by allowing the initial queues $Q_j(1)$ to be nonempty, with $\max_{j\in[n]}Q_j(1)$ polynomially bounded in $T$. This setting arises naturally when the planning horizon begins while the system is already in operation, for example, when the platform re-plans during the day with customers already waiting at some services. We retain the same objective: revenue \(r_j\) is earned from a customer who joins service \(j\) during the planning horizon and completes service by the end of period \(T\). Customers initially present in the system do not directly contribute to the objective, but affect subsequent joining decisions through the congestion they create.

Importantly, under monotone joining, the greedy policy remains a $1/2$-approximation and {\LPGA} remains a polynomial-time $(1-1/e)$-approximation with nonempty initial queues. Both algorithms apply to this setting with little or no modification. The greedy policy requires no modification, since its index $I_{j,t}(Q_j(t))$ is evaluated at the current queue length, which includes any initial customers still waiting. {\LPGA} requires only that the LP occupancy variables $\bm y$ and the virtual queues be initialized at the given initial queue lengths rather than at zero. In Appendix~\ref{appendix-subsec:initial-queues}, we show how the proofs of Theorems~\ref{theorem:greedy} and~\ref{thm:LPGA_approx_ratio} extend to this setting.

\paragraph{Without monotonicity.} In contrast, in the same model with nonempty initial queues, the problem becomes NP-hard to approximate within any constant factor once joining probabilities are allowed to be nonmonotone.

\begin{theorem}
	\label{thm:nonmonotone-hardness}
	For any fixed \(\epsilon>0\), \ref{problem:DSR-abstract} with nonempty initial queues is NP-hard to approximate within a factor of \(O(n^{1-\epsilon})\) when the joining probability functions may be nonmonotone.
\end{theorem}

The proof in Appendix~\ref{subsec:proof-MIS-reduction} is based on a reduction from the maximum independent set problem, which is NP-hard to approximate within a factor of \(O(n^{1-\epsilon})\), where \(n\) denotes the number of vertices \citep{haastad1999clique,zuckerman2007linear}. We construct a \ref{problem:DSR-abstract} instance from any given graph and show that an approximation for this instance can be translated into an approximation for its maximum independent set, thereby transferring the inapproximability result. In the constructed instance, the initial queues act as a clock: each queue length identifies the current period, so a nonmonotone joining function can admit customers only in prescribed periods.

\paragraph{Implication.}
The constant-factor approximation guarantees (Theorems~\ref{theorem:greedy} and~\ref{thm:LPGA_approx_ratio}) and the inapproximability result (Theorem~\ref{thm:nonmonotone-hardness}) concern the same model, the former under monotone joining and the latter without it. These results together identify monotone joining as the source of tractability in dynamic service recommendation. Although congestion-dependent joining links current recommendations to future joining through endogenous congestion, monotone joining gives this link enough structure for polynomial-time constant-factor approximations.

\section{Conclusion}
\label{sec:conclusion}

We study dynamic service recommendation when customers' willingness to join depends on the congestion of the recommended service. This feedback between recommendation, joining behavior, and endogenous congestion creates a high-dimensional dynamic optimization problem. We approach this problem through approximation algorithms and study three that differ in how much of the joint congestion state they retain. Our QPTAS retains a compressed joint state and obtains a near-optimal policy. Under more general model primitives, {\LPGA} retains each service's queue-length distribution over the horizon and earns at least a $(1-1/e)$ fraction of the optimal revenue, while the greedy policy retains only each service's current queue length and earns at least half.

Our results also highlight the role of monotone joining, under which customers become less willing to join as queues grow. This structure enables our approximation guarantees, and without it, the problem with nonempty initial queues becomes hard to approximate within $O(n^{1-\epsilon})$. More broadly, endogenous customer responses to congestion need not only complicate dynamic service systems. When these responses are monotone, they can also make the problem tractable.

\section*{Acknowledgements}
We are grateful to Jake Feldman for valuable discussions
on our QPTAS and to Rouba Ibrahim for helpful discussions
on connections to the queueing literature.

\makeatletter
\newcommand*\mysize{%
	\@setfontsize\mysize{10.0}{11.0}%
}
\makeatother

\renewcommand{\bibfont}{\mysize}
{\setlength{\bibsep}{3pt}
	\bibliographystyle{plainnat}
	\bibliography{dynamic_recommend.bib}
}

\newpage

\pagenumbering{arabic}
\renewcommand*{\thepage}{A.\arabic{page}}
\renewcommand{\thelemma}{A.\arabic{lemma}}
\renewcommand{\thesection}{A.\arabic{section}}
\renewcommand{\theproposition}{A.\arabic{proposition}}
\renewcommand{\thecorollary}{A.\arabic{corollary}}
\renewcommand{\theequation}{A.\arabic{equation}}
\renewcommand{\theremark}{A.\arabic{remark}}
\renewcommand{\theexample}{A.\arabic{example}}
\setcounter{lemma}{0}
\setcounter{section}{0}
\setcounter{proposition}{0}
\setcounter{equation}{0}
\setcounter{corollary}{0}
\setcounter{example}{0}

\begin{appendices}

\section{Omitted Proofs from Section~\ref{sec:preliminaries}}

\subsection{Proof of Proposition~\ref{prop:non-indexability}}
\label{subsec:proof-non-indexability}

Consider the model of Section~\ref{subsec:model} with three services, horizon $T=6$, and
$\lambda_t=1$ for all $t$. Let $\gamma=1.1$, $p_j=\xi_j=1/2$ for $j\in[3]$, and
$(r_1,r_2,r_3)=(1,5,4)$. We compare two joint states at the beginning of period~$5$:
\[
\bm q^G=(0,1,1)
\qquad\text{and}\qquad
\bm q^B=(0,1,2).
\]
In both states, service~1 is empty and service~2 has one customer. Only the queue length of
service~3 differs. The table below reports, for each action in period~$5$, the expected net revenue from
periods $5$ and $6$ when the platform acts optimally in period~$6$. These values
follow from the Bellman equation of Section~\ref{subsec:DP}. In both states, the optimal
action is unique. At $\bm q^G$, it is to recommend service~2. At $\bm q^B$, it is to recommend
service~1. Thus, the optimal choice between services 1 and 2 reverses, although neither
service's own queue length has changed.

\begin{center}
	\begin{tabular}{lcccc}
		\toprule
		& No recommendation & Recommend 1 & Recommend 2 & Recommend 3 \\
		\midrule
		$\bm q^G=(0,1,1)$ & $-1.3125$ & $-0.9531$ & $\mathbf{-0.9251}$ & $-0.9768$ \\
		$\bm q^B=(0,1,2)$ & $-4.5000$ & $\mathbf{-4.1563}$ & $-4.1901$ & $-4.5000$ \\
		\bottomrule
	\end{tabular}
\end{center}

Both states, $\bm q^G$ and $\bm q^B$, are reachable under every optimal policy. Solving the Bellman equation from empty
queues shows that, in periods $1$ to $4$, the optimal action is unique at every state reachable
under an optimal policy. Hence, all optimal policies induce the same distribution of $\bm Q(5)$.
Under this distribution, both $\bm q^G$ and $\bm q^B$ have positive probability. An optimal
policy must therefore take the optimal action in both states.

Suppose that a local-state index policy is optimal. Then it takes the optimal action in both
states. At
$\bm q^G$, it recommends service~2, so $\mathcal{I}_{2,5}(1)\ge \mathcal{I}_{1,5}(0)$. At $\bm q^B$, it
recommends service~1, so $\mathcal{I}_{1,5}(0)\ge \mathcal{I}_{2,5}(1)$. Hence, services 1 and 2 tie for the
largest index in both states. A fixed priority order breaks this tie the same way in both
states. Hence, the policy recommends the same service in both states, a contradiction.
Hence, no local-state index policy is optimal. This proves
Proposition~\ref{prop:non-indexability}.

\section{Omitted Proofs from Section~\ref{sec:QPTAS}}
\label{sec:appendix-QPTAS}

\subsection{Proof of Lemma~\ref{lemma:QPTAS_rounding-rewards}}
\label{subsec:QPTAS_proof-rounding-r}

\paragraph{The lower bound $U_{\mathrm{LB}}$.}
\label{subsec:appendix-lower-bound}
Let $t'$ and $j'$ satisfy $\lambda_{t'}=\lambda_{\max}$ and $r_{j'}p_{j'}\xi_{j'}=\max_{j\in[n]}r_jp_j\xi_j$. Consider the policy that recommends $j'$ in period $t'$ and makes no other recommendation. Service $j'$ is empty in period $t'$, so with probability $\lambda_{\max}p_{j'}\xi_{j'}$ a customer arrives, joins, and completes service within that period. Hence $\OPT\ge U_{\mathrm{LB}}$. The argument uses only the form of the primitives, so the same bound, with the corresponding parameters, holds for every instance considered in Step~1.

\paragraph{Proof of the lemma.}
Fix $\pi\in\Pi$, and let $D_{j,t}\in\{0,1\}$ indicate that service $j$ completes service of a customer in period~$t$. Because all queues start empty, \eqref{eq:original_objective} gives $\Rcal_{(r)}(\pi)=\E[\sum_{t\in[T]}\sum_{j\in[n]}r_jD_{j,t}]$. The revenue parameters do not affect the law of $\{D_{j,t}\}$ under $\pi$, so $\Rcal_{(r^\downarrow)}(\pi)=\E[\sum_{t\in[T]}\sum_{j\in[n]}r^\downarrow_jD_{j,t}]$. The upper bound in the lemma follows from $r^\downarrow_j\le r_j$. For the lower bound, the construction of $\texttt{Grid}_r$ gives $r^\downarrow_j\ge r_j/(1+\epsilon)-\epsilon U_{\mathrm{LB}}/T$ for every $j$ (see Step~1a). %
Since at most one customer arrives per period, $\sum_{t\in[T]}\sum_{j\in[n]}D_{j,t}\le T$ on every sample path. Therefore,
\[
\sum_{t\in[T]}\sum_{j\in[n]} r^\downarrow_j D_{j,t}
\;\ge\;
\frac{1}{1+\epsilon}\sum_{t\in[T]}\sum_{j\in[n]} r_jD_{j,t}-\epsilon\, U_{\mathrm{LB}}
\;\ge\;
(1-\epsilon)\sum_{t\in[T]}\sum_{j\in[n]} r_jD_{j,t}-\epsilon\,\OPT,
\]
where the last step uses $U_{\mathrm{LB}}\le\OPT$. Taking expectations completes the proof.

\subsection{A Coupling Device}
\label{subsec:QPTAS-coupling-device}

The proofs of Lemmas~\ref{lemma:QPTAS_rounding-xi-OPT-comparison}--\ref{lemma:QPTAS_policy-transfer} compare two systems that differ in their primitives or in the policies allowed. In each proof, we construct a policy that operates one system, called the \emph{actual system}. The policy maintains a copy of the other system internally, called the \emph{virtual system}. We couple the two systems through shared uniform random variables. The following elementary fact lets the two systems share uniform random variables even though the constructed policy observes outcomes only in the actual system. We use this lemma first in the proof of Lemma~\ref{lemma:QPTAS_rounding-xi-OPT-comparison}, where the construction is explained in detail. The later proofs follow the same pattern.

\begin{lemma}
	\label{lemma:QPTAS-uniform-reconstruction}
	Let $p\in[0,1]$ and $X\sim\operatorname{Bernoulli}(p)$. Conditional on $X$, draw $V$ uniformly from $[0,p]$ if $X=1$ and from $(p,1]$ if $X=0$. Then $V\sim\operatorname{Uniform}([0,1])$ and $X=\mathbb{I}\{V\le p\}$.
\end{lemma}
\begin{proof}
	For $v\le p$, $\Pp(V\le v)=p\cdot v/p=v$; for $v>p$, $\Pp(V\le v)=p+(1-p)\cdot(v-p)/(1-p)=v$. The second statement holds by construction.
\end{proof}

Before proceeding to the proofs of Lemmas~\ref{lemma:QPTAS_rounding-xi-OPT-comparison}--\ref{lemma:QPTAS_policy-transfer}, we introduce notation used throughout them. As in Section~\ref{subsec:general-model}, we write $\bar a_{j,t}(q):=\lambda_t\,a_j(q)$ for the probability that, in period~$t$, a customer arrives and joins service $j$ when $j$ is recommended at queue length $q$. Under rounded joining functions $a^\downarrow_j(\cdot)$, we write $\bar a^\downarrow_{j,t}(q):=\lambda_t\,a^\downarrow_j(q)$. 

\subsection{Proof of Lemma~\ref{lemma:QPTAS_rounding-xi-OPT-comparison}}
\label{subsec:QPTAS-proof-of-rounding-xi-part1}

Throughout this proof, the unrounded system has service probabilities $\{\xi^S_j\}_{j\in[n]}$, while the slow system has $\{\xi^{S\downarrow}_j\}_{j\in[n]}$. All other model primitives of the two systems are the same, including the rounded revenues $\{r^\downarrow_j\}_{j\in[n]}$. We use the notation of Appendix~\ref{subsec:QPTAS-coupling-device}, in particular $\bar a_{j,t}(q)=\lambda_t\,a_j(q)$, and its coupling device. We prove the following statement, from which Lemma~\ref{lemma:QPTAS_rounding-xi-OPT-comparison} follows.

\begin{claim}
	\label{claim:QPTAS-slow-policy}
	For every deterministic Markov policy $\pi$ for the unrounded system, there exists a policy $\pi^{\STT}\in\Pi$ for the slow system such that $\Rcal_{(\xi^{S\downarrow})}(\pi^{\STT})\ge\Rcal_{(\xi^S)}(\pi)/(1+\epsilon_\xi)$.
\end{claim}

Let $\pi^*_{(\xi^S)}$ be an optimal deterministic Markov policy for the unrounded system, which exists because the horizon and the state space are finite. Applying Claim~\ref{claim:QPTAS-slow-policy} with $\pi=\pi^*_{(\xi^S)}$ gives
\[
\OPT_{(\xi^{S\downarrow})}\ge\Rcal_{(\xi^{S\downarrow})}(\pi^{\STT})\ge\frac{\Rcal_{(\xi^S)}(\pi^*_{(\xi^S)})}{1+\epsilon_\xi}=\frac{\OPT_{(\xi^S)}}{1+\epsilon_\xi}\ge(1-\epsilon_\xi)\cdot\OPT_{(\xi^S)}.
\]
It remains to prove the claim. Fix a deterministic Markov policy $\pi$ for the unrounded system.

\paragraph{Idea of the construction.} The policy $\pi^{\STT}$ we construct for the slow system mimics how $\pi$ operates in the unrounded system. It maintains a virtual copy of the unrounded system, called the \emph{virtual system}, applies $\pi$ to the state of the virtual system, and makes the same recommendation in the slow system. The joining and service outcomes of the virtual system are generated from the same uniform random variables that drive the slow system (Lemma~\ref{lemma:QPTAS-uniform-reconstruction}).

We show that the slow system stays at least as congested as the virtual system throughout the horizon (Claim~\ref{claim:QPTAS-slow-ordering}). Hence, whenever the virtual system completes a service, the slow system has a customer to serve, and it completes a service with probability $\xi^{S\downarrow}_j\ge\xi^S_j/(1+\epsilon_\xi)$. Since revenue is earned upon service completion, the slow system thus earns at least a $1/(1+\epsilon_\xi)$ fraction of the revenue of the virtual system.

\paragraph{Policy construction.} We now construct the policy $\pi^{\STT}$, together with the virtual system that it maintains internally. Let $\bm Q^\STT(t)=(Q^\STT_j(t))_{j\in[n]}$ denote the queue lengths of the slow system at the beginning of period $t$, and let $\bm Q(t)=(Q_j(t))_{j\in[n]}$ denote the state of the virtual system, which $\pi^{\STT}$ updates as described below. Both systems start with empty queues, $\bm Q^\STT(1)=\bm Q(1)=\bm 0$. In period $t$, the policy $\pi^{\STT}$ proceeds as follows.

\begin{center}
	\begin{tcolorbox}[breakable, title={The policy $\pi^{\STT}$ in period $t$}]
		\begin{itemize}[leftmargin=*]
			\item \textbf{Input.} The state $\bm Q^\STT(t)$ of the slow system, which is observed, and the state $\bm Q(t)$ of the virtual system, which $\pi^{\STT}$ maintains internally.

			\item \textbf{Recommendation.} Let $j':=\pi(t,\bm Q(t))$ be the recommendation of $\pi$ in the virtual system. Recommend $j'$ in the slow system; if $j'=0$, make no recommendation.

			\item \textbf{Observe the slow system.}
			\begin{itemize}
				\item Let $A^\STT_{j,t}:=\mathbb{I}\{\text{the arriving customer joins service $j$ in period $t$}\}$, $\forall j\in[n]$.
				\item Let $D^\STT_{j,t}:=\mathbb{I}\{\text{service $j$ completes a customer in period $t$}\}$, $\forall j\in[n]$.
				\item Observe the resulting state $\bm Q^\STT(t+1)$.
			\end{itemize}

			\item \textbf{Virtual joining.} Set $A_{j,t}:=0$, $\forall j\in[n]\setminus\{j'\}$. If $j'\neq0$, reconstruct $V_{j',t}$ from $A^\STT_{j',t}$ as in Lemma~\ref{lemma:QPTAS-uniform-reconstruction} with $p=\bar a_{j',t}(Q^\STT_{j'}(t))$, and define the virtual joining indicator $A_{j',t}:=\mathbb{I}\{V_{j',t}\le\bar a_{j',t}(Q_{j'}(t))\}$.

			\item \textbf{Virtual service.} For each $j\in[n]$, let $Q^{+\STT}_j(t):=Q^\STT_j(t)+A^\STT_{j,t}$. We construct a uniform random variable $U_{j,t}$ shared with the slow system:
			\begin{itemize}
				\item If $Q^{+\STT}_j(t)\ge1$, reconstruct $U_{j,t}$ from $D^\STT_{j,t}$ as in Lemma~\ref{lemma:QPTAS-uniform-reconstruction} with $p=\xi^{S\downarrow}_j$.
				\item If $Q^{+\STT}_j(t)=0$, draw $U_{j,t}\sim\operatorname{Uniform}([0,1])$ independently.
			\end{itemize}
			Define the virtual service opportunity $B_{j,t}:=\mathbb{I}\{U_{j,t}\le\xi^S_j\}$.

			\item \textbf{Virtual update.} Set $Q_j^+(t):=Q_j(t)+A_{j,t}$ and $Q_j(t+1):=\bigl(Q_j^+(t)-B_{j,t}\bigr)_+$, $\forall j\in[n]$.
		\end{itemize}
	\end{tcolorbox}
\end{center}

\paragraph{Validity and coupling.} The policy $\pi^{\STT}$ uses only the queue lengths and joining outcomes of the slow system, which the platform observes, together with its own randomization; the virtual system that it maintains internally is computed from these observations and this randomization. Hence, $\pi^{\STT}\in\Pi$. 

We next verify that the virtual system has the same law as the unrounded system operated under~$\pi$, and that its outcomes are generated from the same uniform random variables as those of the slow system. Let the history up to the start of period~$t$ consist of all joining and service outcomes of the slow system and all random variables generated by $\pi^{\STT}$ before period~$t$. This history determines the states $\bm Q^\STT(t)$ and $\bm Q(t)$, and, since $\pi$ is deterministic, also the recommendation $j'=\pi(t,\bm Q(t))$. Conditional on this history, the joining indicator $A^\STT_{j',t}$ (if $j'\neq0$) and the service opportunities $B^\STT_{1,t},\ldots,B^\STT_{n,t}$ of the slow system are independent Bernoulli random variables with success probabilities $\bar a_{j',t}(Q^\STT_{j'}(t))$ and $\xi^{S\downarrow}_1,\ldots,\xi^{S\downarrow}_n$, and $D^\STT_{j,t}=B^\STT_{j,t}$ whenever $Q^{+\STT}_j(t)\ge1$. By Lemma~\ref{lemma:QPTAS-uniform-reconstruction}, and since these inputs and the reconstruction draws are independent, the variables $V_{j',t},U_{1,t},\ldots,U_{n,t}$ are, conditional on this history, i.i.d.\ uniform on $[0,1]$. This has two consequences.
\begin{enumerate}[label=(\roman*), leftmargin=2.5em]
	\item By construction,
	\[
	A_{j',t}=\mathbb{I}\bigl\{V_{j',t}\le\bar a_{j',t}(Q_{j'}(t))\bigr\},
	\qquad
	B_{j,t}=\mathbb{I}\bigl\{U_{j,t}\le\xi^S_j\bigr\},
	\]
	so, conditional on the history, $A_{j',t},B_{1,t},\ldots,B_{n,t}$ are independent Bernoulli random variables with success probabilities $\bar a_{j',t}(Q_{j'}(t))$ and $\xi^S_1,\ldots,\xi^S_n$. Hence, the virtual process $\{\bm Q(t)\}$ has the law of the unrounded system operated under $\pi$.
	\item By Lemma~\ref{lemma:QPTAS-uniform-reconstruction}, the outcomes of the slow system are generated by the same uniform random variables as their virtual counterparts:
	\[
	A^\STT_{j',t}=\mathbb{I}\bigl\{V_{j',t}\le\bar a_{j',t}(Q^\STT_{j'}(t))\bigr\},
	\qquad
	B^\STT_{j,t}=\mathbb{I}\bigl\{U_{j,t}\le\xi^{S\downarrow}_j\bigr\}.
	\]
	For $j$ with $Q^{+\STT}_j(t)=0$, the service opportunity $B^\STT_{j,t}$ is not observed and has no effect on the slow system, so we may take it to be $\mathbb{I}\{U_{j,t}\le\xi^{S\downarrow}_j\}$.
\end{enumerate}
The completion indicators are $D^\STT_{j,t}=\mathbb{I}\{Q^{+\STT}_j(t)\ge1\}\cdot B^\STT_{j,t}$ and $D_{j,t}=\mathbb{I}\{Q^+_j(t)\ge1\}\cdot B_{j,t}$.

\paragraph{Ordering of the two systems.} We next establish the key ordering of the two systems.

\begin{claim}
	\label{claim:QPTAS-slow-ordering}
	For all $j\in[n]$, we have $Q_j(t)\le Q^\STT_j(t)$ for $t\in[T+1]$, and $Q^+_j(t)\le Q^{+\STT}_j(t)$ for $t\in[T]$.
\end{claim}
\begin{proof}
	We proceed by induction on $t$. The claim holds at $t=1$ because both systems start empty. Suppose that $Q_j(t)\le Q^\STT_j(t)$ for all $j\in[n]$. We first show that this ordering is preserved by the joining step, that is,
	\begin{equation}
		\label{eq:QPTAS-slow-ordering-join}
		Q^+_j(t)=Q_j(t)+A_{j,t}\le Q^\STT_j(t)+A^\STT_{j,t}=Q^{+\STT}_j(t),
		\qquad j\in[n].
	\end{equation}
	Let $j'=\pi(t,\bm Q(t))$ be the service recommended in period $t$, which is the same in both systems. For $j\neq j'$, both joining indicators are zero, so~\eqref{eq:QPTAS-slow-ordering-join} follows. Consider therefore $j=j'$.
	\begin{itemize}[leftmargin=*]
		\item If $Q_j(t)=Q^\STT_j(t)$, the two systems face the same joining probability, $\bar a_{j,t}(Q_j(t))=\bar a_{j,t}(Q^\STT_j(t))$. Since both joining indicators are generated by the same uniform random variable $V_{j,t}$, by consequences~(i) and~(ii), we have $A_{j,t}=A^\STT_{j,t}$, and the post-joining queue lengths are equal.
		\item If $Q_j(t)<Q^\STT_j(t)$, then $Q^\STT_j(t)\ge Q_j(t)+1$, since queue lengths are integer valued. Because both joining indicators belong to $\{0,1\}$, we have $
		Q^\STT_j(t)+A^\STT_{j,t}\ge Q_j(t)+1\ge Q_j(t)+A_{j,t}$.
	\end{itemize}
	This proves~\eqref{eq:QPTAS-slow-ordering-join}. We next show that the ordering is preserved by the service step. Since $\xi^{S\downarrow}_j\le\xi^S_j$ and both service opportunities are generated by the same uniform random variable $U_{j,t}$, we have $B^\STT_{j,t}\le B_{j,t}$. Combining this with~\eqref{eq:QPTAS-slow-ordering-join} completes the induction:
	\[
	Q^\STT_j(t+1)
	=
	\bigl(Q^{+\STT}_j(t)-B^\STT_{j,t}\bigr)_+
	\ge
	\bigl(Q^+_j(t)-B^\STT_{j,t}\bigr)_+
	\ge
	\bigl(Q^+_j(t)-B_{j,t}\bigr)_+
	=
	Q_j(t+1). \qedhere
	\]
\end{proof}

\paragraph{Revenue comparison.} Since both systems start empty, the expected net revenue of each system is the expected revenue from completed services:
\begin{equation*}
	\Rcal_{(\xi^{S\downarrow})}(\pi^{\STT})=\sum_{t\in[T]}\sum_{j\in[n]}r^\downarrow_j\,\E\bigl[D^\STT_{j,t}\bigr],
	\qquad
	\Rcal_{(\xi^S)}(\pi)=\sum_{t\in[T]}\sum_{j\in[n]}r^\downarrow_j\,\E\bigl[D_{j,t}\bigr],
\end{equation*}
where the second identity holds because the virtual system has the law of the unrounded system operated under $\pi$, by consequence~(i) above. By Claim~\ref{claim:QPTAS-slow-ordering}, $Q^+_j(t)\ge1$ implies $Q^{+ \STT}_j(t)\ge1$. Hence,
\[D^\STT_{j,t} = \mathbb{I}\{Q^{+ \STT}_j(t)\ge1\}\cdot B^\STT_{j,t} \ge\mathbb{I}\{Q^+_j(t)\ge1\}\cdot B^\STT_{j,t}. 
\]
The post-joining queue length $Q^+_j(t)$ is determined by the history up to the start of period~$t$ and by $V_{j',t}$. Conditional on this history, $U_{j,t}$ is uniform on $[0,1]$ and independent of $V_{j',t}$, so $U_{j,t}$ is independent of $Q^+_j(t)$. Hence,
\begin{equation*}
	\E\bigl[D^\STT_{j,t}\bigr]
	\ge
	\E \left[ B^\STT_{j,t} \cdot \mathbb{I}\{Q^+_j(t)\ge1\}\right] 
	=
	\xi^{S\downarrow}_j\cdot\Pp\bigl(Q^+_j(t)\ge1\bigr)
	\ge
	\frac{\xi^S_j}{1+\epsilon_\xi}\cdot\Pp\bigl(Q^+_j(t)\ge1\bigr)
	=
	\frac{\E[D_{j,t}]}{1+\epsilon_\xi}.
\end{equation*}
Multiplying by $r^\downarrow_j$ and summing over $j$ and $t$ gives $\Rcal_{(\xi^{S\downarrow})}(\pi^{\STT})\ge \frac{\Rcal_{(\xi^S)}(\pi)}{1+\epsilon_\xi}$, which proves Claim~\ref{claim:QPTAS-slow-policy}.

\subsection{Proof of Lemma~\ref{lem:truncation}}
\label{subsec:QPTAS_proof-policy-truncation}

Throughout this proof, $\Rcal(\pi)$ denotes the expected net revenue of a policy $\pi$ for the instance obtained after Steps~1a and~1b. We use the notation of Appendix~\ref{subsec:QPTAS-coupling-device}, in particular \mbox{$\bar a_{j,t}(q)=\lambda_t\,a_j(q)$}, and its coupling device. We prove the following statement, from which Lemma~\ref{lem:truncation} follows.

\begin{claim}
	\label{claim:QPTAS-truncation-policy}
	For every deterministic Markov policy $\pi$, there exists a policy $\pi^{\truncate}\in\Pi^{\truncate}$ such that $\Rcal(\pi^{\truncate})\ge\Rcal(\pi)-\epsilon\cdot\sup_{\pi'\in\Pi}\Rcal(\pi')$.
\end{claim}

Let $\pi^*$ be an optimal deterministic Markov policy with respect to $\Rcal(\cdot)$. It exists because the horizon and the state space are finite. Applying Claim~\ref{claim:QPTAS-truncation-policy} with $\pi=\pi^*$ gives
\[
\sup_{\pi\in\Pi^{\truncate}}\Rcal(\pi)\ge\Rcal(\pi^{\truncate})\ge\Rcal(\pi^*)-\epsilon\cdot\sup_{\pi'\in\Pi}\Rcal(\pi')=(1-\epsilon)\cdot\sup_{\pi\in\Pi}\Rcal(\pi).
\]
It remains to prove the claim. Fix a deterministic Markov policy $\pi$.

\paragraph{Idea of the construction.} The policy $\pi$ may recommend a service whose queue has reached its threshold, so it need not belong to $\Pi^{\truncate}$. As in the proof of Lemma~\ref{lemma:QPTAS_rounding-xi-OPT-comparison}, the policy $\pi^{\truncate}$ we construct mimics how $\pi$ operates without truncation. We call the system on which $\pi^{\truncate}$ acts, and in which it observes joining and service outcomes, the \emph{actual system}. The policy $\pi^{\truncate}$ maintains a virtual copy of the same system operated under $\pi$, called the \emph{virtual system}. It applies $\pi$ to the state of the virtual system and makes the same recommendation in the actual system, unless the queue of the recommended service in the actual system has reached its threshold. In that case, it makes no recommendation. The joining and service outcomes of the virtual system are generated from the same uniform random variables that drive the actual system (Lemma~\ref{lemma:QPTAS-uniform-reconstruction}).

We show that the actual system stays at most as congested as the virtual system throughout the horizon (Claim~\ref{claim:QPTAS-truncation-ordering}). Building on this ordering, the actual system loses revenue only through suppressed recommendations. When a recommendation of service $j$ is suppressed, its queue has length at least $\Lambda_j$ in the actual system, and hence also in the virtual system. At such a queue length, a customer joins with probability at most $\lambda_t p_j\xi^{S\downarrow}_j\epsilon/T$ by the choice of $\Lambda_j$. Since at most one recommendation is suppressed in each period, the total loss is at most an $\epsilon$ fraction of the optimal objective value.

\paragraph{Policy construction.} We now construct the policy $\pi^{\truncate}$, together with the virtual system that it maintains internally. Let $\bm Q^{\truncate}(t)=(Q^{\truncate}_j(t))_{j\in[n]}$ denote the queue lengths of the actual system at the beginning of period $t$, and let $\bm Q(t)=(Q_j(t))_{j\in[n]}$ denote the state of the virtual system, which $\pi^{\truncate}$ updates as described below. Both systems start with empty queues, $\bm Q^{\truncate}(1)=\bm Q(1)=\bm 0$. In period $t$, the policy $\pi^{\truncate}$ proceeds as follows.

\begin{center}
	\begin{tcolorbox}[breakable, title={The policy $\pi^{\truncate}$ in period $t$}]
		\begin{itemize}[leftmargin=*]
			\item \textbf{Input.} The state $\bm Q^{\truncate}(t)$ of the actual system, which is observed, and the state $\bm Q(t)$ of the virtual system, which $\pi^{\truncate}$ maintains internally.

			\item \textbf{Recommendation.} Let $j':=\pi(t,\bm Q(t))$ be the recommendation of $\pi$ in the virtual system. If $j'=0$, make no recommendation. If $j'\neq0$ and $Q^{\truncate}_{j'}(t)<\Lambda_{j'}$, recommend $j'$ in the actual system. If $j'\neq0$ and $Q^{\truncate}_{j'}(t)\ge\Lambda_{j'}$, make no recommendation. In this last case, we say that the recommendation of $j'$ is \emph{suppressed}.

			\item \textbf{Observe the actual system.}
			\begin{itemize}
				\item Let $A^{\truncate}_{j,t}:=\mathbb{I}\{\text{the arriving customer joins service $j$ in period $t$}\}$, $\forall j\in[n]$.
				\item Let $D^{\truncate}_{j,t}:=\mathbb{I}\{\text{service $j$ completes a customer in period $t$}\}$, $\forall j\in[n]$.
				\item Observe the resulting state $\bm Q^{\truncate}(t+1)$.
			\end{itemize}

			\item \textbf{Virtual joining.} Set $A_{j,t}:=0$, $\forall j\in[n]\setminus\{j'\}$. If $j'\neq0$, construct a uniform random variable $V_{j',t}$:
			\begin{itemize}
				\item If $j'$ is recommended in the actual system, reconstruct $V_{j',t}$ from $A^{\truncate}_{j',t}$ as in Lemma~\ref{lemma:QPTAS-uniform-reconstruction} with $p=\bar a_{j',t}(Q^{\truncate}_{j'}(t))$.
				\item If the recommendation of $j'$ is suppressed, draw $V_{j',t}\sim\operatorname{Uniform}([0,1])$ independently.
			\end{itemize}
			Define the virtual joining indicator $A_{j',t}:=\mathbb{I}\{V_{j',t}\le\bar a_{j',t}(Q_{j'}(t))\}$.

			\item \textbf{Virtual service.} For each $j\in[n]$, let $Q^{+\truncate}_j(t):=Q^{\truncate}_j(t)+A^{\truncate}_{j,t}$. We construct a uniform random variable $U_{j,t}$ shared with the actual system:
			\begin{itemize}
				\item If $Q^{+\truncate}_j(t)\ge1$, reconstruct $U_{j,t}$ from $D^{\truncate}_{j,t}$ as in Lemma~\ref{lemma:QPTAS-uniform-reconstruction} with $p=\xi^{S\downarrow}_j$.
				\item If $Q^{+\truncate}_j(t)=0$, draw $U_{j,t}\sim\operatorname{Uniform}([0,1])$ independently.
			\end{itemize}
			Define the virtual service opportunity $B_{j,t}:=\mathbb{I}\{U_{j,t}\le\xi^{S\downarrow}_j\}$.

			\item \textbf{Virtual update.} Set $Q_j^+(t):=Q_j(t)+A_{j,t}$ and $Q_j(t+1):=\bigl(Q_j^+(t)-B_{j,t}\bigr)_+$, $\forall j\in[n]$.
		\end{itemize}
	\end{tcolorbox}
\end{center}

\paragraph{Validity and coupling.} The policy $\pi^{\truncate}$ uses only the queue lengths and joining outcomes of the actual system, which the platform observes, together with its own randomization; the virtual system that it maintains internally is computed from these observations and this randomization. Hence, $\pi^{\truncate}\in\Pi$. Moreover, $\pi^{\truncate}$ recommends a service $j$ only if $Q^{\truncate}_j(t)<\Lambda_j$, so $\pi^{\truncate}\in\Pi^{\truncate}$.

Let the history up to the start of period~$t$ consist of all joining and service outcomes of the actual system and all random variables generated by $\pi^{\truncate}$ before period~$t$. This history determines the states $\bm Q^{\truncate}(t)$ and $\bm Q(t)$, and, since $\pi$ is deterministic, also the recommendation $j'=\pi(t,\bm Q(t))$ and whether this recommendation is suppressed. Conditional on this history, the service opportunities $B^{\truncate}_{1,t},\ldots,B^{\truncate}_{n,t}$ of the actual system are independent Bernoulli random variables with success probabilities $\xi^{S\downarrow}_1,\ldots,\xi^{S\downarrow}_n$, and $D^{\truncate}_{j,t}=B^{\truncate}_{j,t}$ whenever $Q^{+\truncate}_j(t)\ge1$. If $j'$ is recommended in the actual system, the joining indicator $A^{\truncate}_{j',t}$ is a further independent Bernoulli random variable with success probability $\bar a_{j',t}(Q^{\truncate}_{j'}(t))$. By Lemma~\ref{lemma:QPTAS-uniform-reconstruction}, and since the reconstruction draws and the independent draw of $V_{j',t}$ in the suppressed case are independent of these inputs, the variables $V_{j',t},U_{1,t},\ldots,U_{n,t}$ are, conditional on this history, i.i.d.\ uniform on $[0,1]$. This has two consequences:
\begin{enumerate}[label=(\roman*), leftmargin=2.5em]
	\item the virtual process $\{\bm Q(t)\}$ has the law of the system operated under $\pi$; and
	\item the outcomes of the actual system are generated by the same uniform random variables as their virtual counterparts:
	\[
	A^{\truncate}_{j',t}=
	\begin{cases}
		\mathbb{I}\bigl\{V_{j',t}\le\bar a_{j',t}(Q^{\truncate}_{j'}(t))\bigr\}, & \text{if $j'$ is recommended in the actual system},\\
		0, & \text{if the recommendation of $j'$ is suppressed},
	\end{cases}
	\]
	and
	\[
	B^{\truncate}_{j,t}=\mathbb{I}\bigl\{U_{j,t}\le\xi^{S\downarrow}_j\bigr\}=B_{j,t},
	\]
	where, as in the proof of Lemma~\ref{lemma:QPTAS_rounding-xi-OPT-comparison}, we take the unobserved service opportunity $B^{\truncate}_{j,t}$ to be $\mathbb{I}\{U_{j,t}\le\xi^{S\downarrow}_j\}$ when $Q^{+\truncate}_j(t)=0$.
\end{enumerate}
The completion indicators are $D^{\truncate}_{j,t}=\mathbb{I}\{Q^{+\truncate}_j(t)\ge1\}\cdot B_{j,t}$ and $D_{j,t}=\mathbb{I}\{Q^+_j(t)\ge1\}\cdot B_{j,t}$.

\paragraph{Ordering of the two systems.} We next establish the ordering of the two systems.

\begin{claim}
	\label{claim:QPTAS-truncation-ordering}
	For all $j\in[n]$, we have $Q^{\truncate}_j(t)\le Q_j(t)$ for $t\in[T+1]$, and $Q^{+\truncate}_j(t)\le Q^+_j(t)$ for $t\in[T]$.
\end{claim}
\begin{proof}
	We proceed by induction on $t$. The claim holds at $t=1$ because both systems start empty. Suppose that $Q^{\truncate}_j(t)\le Q_j(t)$ for all $j\in[n]$. We first show that this ordering is preserved by the joining step, that is,
	\begin{equation}
		\label{eq:QPTAS-truncation-ordering-join}
		Q^{+\truncate}_j(t)=Q^{\truncate}_j(t)+A^{\truncate}_{j,t}\le Q_j(t)+A_{j,t}=Q^+_j(t),
		\qquad j\in[n].
	\end{equation}
	Let $j'=\pi(t,\bm Q(t))$. For $j\neq j'$, both joining indicators are zero, so~\eqref{eq:QPTAS-truncation-ordering-join} follows immediately. Consider therefore $j=j'$.
	\begin{itemize}[leftmargin=*]
		\item If the recommendation of $j$ is suppressed, then $A^{\truncate}_{j,t}=0$, and $Q^{+\truncate}_j(t)=Q^{\truncate}_j(t)\le Q_j(t)\le Q^+_j(t)$.
		\item If $j$ is recommended in the actual system and $Q^{\truncate}_j(t)=Q_j(t)$, the two systems face the same joining probability, and both joining indicators are generated by the same uniform random variable $V_{j,t}$, by consequence~(ii). Hence, $A^{\truncate}_{j,t}=A_{j,t}$ and $Q^{\truncate}_j(t) + A^{\truncate}_{j,t} = Q_j(t) + A_{j,t} $.
		\item If $j$ is recommended in the actual system and $Q^{\truncate}_j(t)<Q_j(t)$, then $Q^{\truncate}_j(t)+1\le Q_j(t)$, since queue lengths are integer valued. Because both joining indicators belong to $\{0,1\}$, we have
		\[
		Q^{\truncate}_j(t)+A^{\truncate}_{j,t}\le Q^{\truncate}_j(t)+1\le Q_j(t)\le Q_j(t)+A_{j,t}.
		\]
	\end{itemize}
	This proves~\eqref{eq:QPTAS-truncation-ordering-join}. Since the two systems share the service opportunity $B_{j,t}$, by consequence~(ii),
	\[
	Q^{\truncate}_j(t+1)
	=
	\bigl(Q^{+\truncate}_j(t)-B_{j,t}\bigr)_+
	\le
	\bigl(Q^+_j(t)-B_{j,t}\bigr)_+
	=
	Q_j(t+1).
	\]
	This completes the induction.
\end{proof}

\paragraph{Revenue comparison.} Since both systems start empty, the revenue from service $j$ is $r^\downarrow_j$ times the number of customers it completes by the end of the horizon. Let
\[
\Acal^{\truncate}_j:=\sum_{t\in[T]}A^{\truncate}_{j,t},
\qquad
\Acal_j:=\sum_{t\in[T]}A_{j,t},
\qquad
\Dcal^{\truncate}_j:=\sum_{t\in[T]}D^{\truncate}_{j,t},
\qquad
\Dcal_j:=\sum_{t\in[T]}D_{j,t}
\]
denote the numbers of customers who join and complete service $j$ in the two systems. Then $\Rcal(\pi^{\truncate})=\sum_{j\in[n]}r^\downarrow_j\,\E[\Dcal^{\truncate}_j]$, and, by consequence~(i), $\Rcal(\pi)=\sum_{j\in[n]}r^\downarrow_j\,\E[\Dcal_j]$. Since $Q^{\truncate}_j(T+1)=\Acal^{\truncate}_j-\Dcal^{\truncate}_j$ and $Q_j(T+1)=\Acal_j-\Dcal_j$, Claim~\ref{claim:QPTAS-truncation-ordering} gives
\begin{equation}
	\label{eq:QPTAS-truncation-completions}
	\Dcal_j-\Dcal^{\truncate}_j
	=
	\bigl(\Acal_j-\Acal^{\truncate}_j\bigr)+\bigl(Q^{\truncate}_j(T+1)-Q_j(T+1)\bigr)
	\le
	\Acal_j-\Acal^{\truncate}_j.
\end{equation}
Thus, the actual system loses completions only through customers who join the virtual system but not the actual one.

Let $H_{j,t}:=\{j=\pi(t,\bm Q(t)),\ Q^{\truncate}_j(t)\ge\Lambda_j\}$ denote the event that the recommendation of $j$ is suppressed in period $t$. On $H_{j,t}$, we have $A^{\truncate}_{j,t}=0$. Outside $H_{j,t}$, either $j\neq\pi(t,\bm Q(t))$, so that $A^{\truncate}_{j,t}=A_{j,t}=0$, or $j$ is recommended in the actual system; in the latter case, $Q^{\truncate}_j(t)\le Q_j(t)$ by Claim~\ref{claim:QPTAS-truncation-ordering}, so $\bar a_{j,t}(Q^{\truncate}_j(t))\ge\bar a_{j,t}(Q_j(t))$ by monotone joining~\eqref{eq:monotone-a}, and the shared $V_{j,t}$ gives $A^{\truncate}_{j,t}\ge A_{j,t}$. Hence,
\begin{equation}
\label{eq:QPTAS-truncation-supressed-event-defn}
A_{j,t}-A^{\truncate}_{j,t}\le A_{j,t}\cdot\mathbb{I}\{H_{j,t}\}.
\end{equation}
The event $H_{j,t}$ is determined by the history up to the start of period $t$, and on $H_{j,t}$ we have $j=j'$ and $Q_j(t)\ge Q^{\truncate}_j(t)\ge\Lambda_j$. On $H_{j,t}$, conditional on this history, $A_{j,t}=\mathbb{I}\{V_{j,t}\le\bar a_{j,t}(Q_j(t))\}$ with $V_{j,t}$ uniform on $[0,1]$. Taking expectations over the history gives
\begin{equation}
	\label{eq:QPTAS-truncation-suppressed}
	\E\bigl[A_{j,t}\cdot\mathbb{I}\{H_{j,t}\}\bigr]
	=
	\E\bigl[\bar a_{j,t}(Q_j(t))\cdot\mathbb{I}\{H_{j,t}\}\bigr]
	\le
	\lambda_t p_j\gamma^{-\Lambda_j/\xi^A_j}\cdot\Pp\bigl(j=\pi(t,\bm Q(t))\bigr).
\end{equation}
Since $\xi^A_j=\xi_j\le(1+\epsilon_\xi)\,\xi^{S\downarrow}_j$, the definition of $\Lambda_j$ gives
\[
\frac{\Lambda_j}{\xi^A_j}
\ge
\frac{\Lambda_j}{(1+\epsilon_\xi)\,\xi^{S\downarrow}_j}
=
\log_\gamma\left(\frac{T}{\xi^{S\downarrow}_j\epsilon}\right),
\qquad\text{so}\qquad
\gamma^{-\Lambda_j/\xi^A_j}\le\frac{\xi^{S\downarrow}_j\epsilon}{T}.
\]
Substituting this into~\eqref{eq:QPTAS-truncation-suppressed}, multiplying by $r^\downarrow_j$, and summing over $j$ gives, for every period $t$,
\[
\sum_{j\in[n]}r^\downarrow_j\,\E\bigl[A_{j,t}\cdot\mathbb{I}\{H_{j,t}\}\bigr]
\le
\frac{\epsilon}{T}\sum_{j\in[n]}\lambda_t\,r^\downarrow_j p_j\xi^{S\downarrow}_j\cdot\Pp\bigl(j=\pi(t,\bm Q(t))\bigr)
\le
\frac{\epsilon}{T}\cdot\lambda_{\max}\max_{j\in[n]} \left\lbrace r^\downarrow_j p_j\xi^{S\downarrow}_j \right\rbrace,
\]
where the last inequality holds because $\pi$ selects at most one service in each period. The quantity $\lambda_{\max}\max_{j}r^\downarrow_jp_j\xi^{S\downarrow}_j$ is the lower bound $U_{\mathrm{LB}}$ of Appendix~\ref{subsec:appendix-lower-bound} for the instance obtained after Steps~1a and~1b, so it is at most $\sup_{\pi'\in\Pi}\Rcal(\pi')$. Combining these bounds with~\eqref{eq:QPTAS-truncation-completions} and~\eqref{eq:QPTAS-truncation-supressed-event-defn} and summing over $t\in[T]$ proves Claim~\ref{claim:QPTAS-truncation-policy}:
\[
\Rcal(\pi)-\Rcal(\pi^{\truncate})
=
\sum_{j\in[n]}r^\downarrow_j\,\E\bigl[\Dcal_j-\Dcal^{\truncate}_j\bigr]
\le
\sum_{t\in[T]}\sum_{j\in[n]}r^\downarrow_j\,\E\bigl[A_{j,t}\cdot\mathbb{I}\{H_{j,t}\}\bigr]
\le
\epsilon\cdot\sup_{\pi'\in\Pi}\Rcal(\pi').
\]

\subsection{Proof of Lemma~\ref{lemma:QPTAS_rounding-a-OPT-comparison}}
\label{subsec:QPTAS-proof_rounding-a-OPT-comparison}

Throughout this proof, the unrounded system has joining functions $\{a_j(\cdot)\}_{j\in[n]}$, while the rounded system has $\{a^\downarrow_j(\cdot)\}_{j\in[n]}$. All other model primitives of the two systems are the same. We use the notation of Appendix~\ref{subsec:QPTAS-coupling-device}, in particular $\bar a_{j,t}(q)=\lambda_t\,a_j(q)$ and $\bar a^\downarrow_{j,t}(q)=\lambda_t\,a^\downarrow_j(q)$, and its coupling device. We prove the following statement, from which Lemma~\ref{lemma:QPTAS_rounding-a-OPT-comparison} follows.

\begin{claim}
	\label{claim:QPTAS-rounding-a-policy}
	For every deterministic Markov policy $\pi\in\Pi^{\truncate}$ for the unrounded system, there exists a policy $\pi^\downarrow\in\Pi^{\truncate}$ for the rounded system such that $\Rcal_{(a^\downarrow)}(\pi^\downarrow)\ge\beta\cdot\Rcal_{(a)}(\pi)$.
\end{claim}

Let $\pi^*_{(a)}$ be an optimal deterministic Markov policy for the unrounded system within $\Pi^{\truncate}$, which exists because the horizon and the state space are finite and the truncation rule only removes actions at some states. Applying Claim~\ref{claim:QPTAS-rounding-a-policy} with $\pi=\pi^*_{(a)}$ gives
\[
\OPT^{\truncate}_{(a^\downarrow)}\ge\Rcal_{(a^\downarrow)}(\pi^\downarrow)\ge\beta\cdot\Rcal_{(a)}(\pi^*_{(a)})=\beta\cdot\OPT^{\truncate}_{(a)}.
\]
It remains to prove the claim. Fix a deterministic Markov policy $\pi\in\Pi^{\truncate}$ for the unrounded system.

\paragraph{Idea of the construction.} As in the proofs of Lemmas~\ref{lemma:QPTAS_rounding-xi-OPT-comparison} and~\ref{lem:truncation}, the policy $\pi^\downarrow$ we construct for the rounded system mimics how $\pi$ operates in the unrounded system. We call the rounded system, on which $\pi^\downarrow$ acts and in which it observes joining and service outcomes, the \emph{actual system}. The policy $\pi^\downarrow$ maintains a virtual copy of the unrounded system operated under $\pi$, called the \emph{virtual system}. It applies $\pi$ to the state of the virtual system and makes the same recommendation in the actual system. The joining and service outcomes of the virtual system are generated from the same uniform random variables that drive the actual system (Lemma~\ref{lemma:QPTAS-uniform-reconstruction}).

We show that the actual system stays at most as congested as the virtual system throughout the horizon (Claim~\ref{claim:QPTAS-rounding-a-ordering}). This ordering alone does not bound the revenue loss. We therefore introduce a third system, used only in the analysis, which we call the \emph{$\beta$-virtual system}. It retains each customer who joins the virtual system independently with probability $\beta$. Two comparisons then give the result: the $\beta$-virtual system earns no more than the actual system, while it earns at least a $\beta$ fraction of the revenue of the virtual system.

\paragraph{Policy construction.} We now construct the policy $\pi^\downarrow$, together with the virtual system that it maintains internally. Let $\bm Q^\downarrow(t)=(Q^\downarrow_j(t))_{j\in[n]}$ denote the queue lengths of the actual system at the beginning of period $t$, and let $\bm Q(t)=(Q_j(t))_{j\in[n]}$ denote the state of the virtual system, which $\pi^\downarrow$ updates as described below. Both systems start with empty queues, $\bm Q^\downarrow(1)=\bm Q(1)=\bm 0$. In period $t$, the policy $\pi^\downarrow$ proceeds as follows.

\begin{center}
	\begin{tcolorbox}[breakable, title={The policy $\pi^\downarrow$ in period $t$}]
		\begin{itemize}[leftmargin=*]
			\item \textbf{Input.} The state $\bm Q^\downarrow(t)$ of the actual system, which is observed, and the state $\bm Q(t)$ of the virtual system, which $\pi^\downarrow$ maintains internally.

			\item \textbf{Recommendation.} Let $j':=\pi(t,\bm Q(t))$ be the recommendation of $\pi$ in the virtual system. Recommend $j'$ in the actual system; if $j'=0$, make no recommendation.

			\item \textbf{Observe the actual system.}
			\begin{itemize}
				\item Let $A^\downarrow_{j,t}:=\mathbb{I}\{\text{the arriving customer joins service $j$ in period $t$}\}$, $\forall j\in[n]$.
				\item Let $D^\downarrow_{j,t}:=\mathbb{I}\{\text{service $j$ completes a customer in period $t$}\}$, $\forall j\in[n]$.
				\item Observe the resulting state $\bm Q^\downarrow(t+1)$.
			\end{itemize}

			\item \textbf{Virtual joining.} Set $A_{j,t}:=0$, $\forall j\in[n]\setminus\{j'\}$. If $j'\neq0$, reconstruct $V_{j',t}$ from $A^\downarrow_{j',t}$ as in Lemma~\ref{lemma:QPTAS-uniform-reconstruction} with $p=\bar a^\downarrow_{j',t}(Q^\downarrow_{j'}(t))$, and define the virtual joining indicator $A_{j',t}:=\mathbb{I}\{V_{j',t}\le\bar a_{j',t}(Q_{j'}(t))\}$.

			\item \textbf{Virtual service.} For each $j\in[n]$, let $Q^{+\downarrow}_j(t):=Q^\downarrow_j(t)+A^\downarrow_{j,t}$. We construct a uniform random variable $U_{j,t}$ shared with the actual system:
			\begin{itemize}
				\item If $Q^{+\downarrow}_j(t)\ge1$, reconstruct $U_{j,t}$ from $D^\downarrow_{j,t}$ as in Lemma~\ref{lemma:QPTAS-uniform-reconstruction} with $p=\xi^{S\downarrow}_j$.
				\item If $Q^{+\downarrow}_j(t)=0$, draw $U_{j,t}\sim\operatorname{Uniform}([0,1])$ independently.
			\end{itemize}
			Define the virtual service opportunity $B_{j,t}:=\mathbb{I}\{U_{j,t}\le\xi^{S\downarrow}_j\}$.

			\item \textbf{Virtual update.} Set $Q_j^+(t):=Q_j(t)+A_{j,t}$ and $Q_j(t+1):=\bigl(Q_j^+(t)-B_{j,t}\bigr)_+$, $\forall j\in[n]$.
		\end{itemize}
	\end{tcolorbox}
\end{center}

\paragraph{Validity and coupling.} The policy $\pi^\downarrow$ uses only the queue lengths and joining outcomes of the actual system, which the platform observes, together with its own randomization; the virtual system that it maintains internally is computed from these observations and this randomization. Hence, $\pi^\downarrow\in\Pi$; we show that $\pi^\downarrow\in\Pi^{\truncate}$ after Claim~\ref{claim:QPTAS-rounding-a-ordering} below.

Let the history up to the start of period~$t$ consist of all joining and service outcomes of the actual system and all random variables generated by $\pi^\downarrow$ before period~$t$. This history determines the states $\bm Q^\downarrow(t)$ and $\bm Q(t)$, and, since $\pi$ is deterministic, also the recommendation $j'=\pi(t,\bm Q(t))$. Conditional on this history, the joining indicator $A^\downarrow_{j',t}$ (if $j'\neq0$) and the service opportunities $B^\downarrow_{1,t},\ldots,B^\downarrow_{n,t}$ of the actual system are independent Bernoulli random variables with success probabilities $\bar a^\downarrow_{j',t}(Q^\downarrow_{j'}(t))$ and $\xi^{S\downarrow}_1,\ldots,\xi^{S\downarrow}_n$, and $D^\downarrow_{j,t}=B^\downarrow_{j,t}$ whenever $Q^{+\downarrow}_j(t)\ge1$. By Lemma~\ref{lemma:QPTAS-uniform-reconstruction}, and since these inputs and the reconstruction draws are independent, the variables $V_{j',t},U_{1,t},\ldots,U_{n,t}$ are, conditional on this history, i.i.d.\ uniform on $[0,1]$. This has two consequences:
\begin{enumerate}[label=(\roman*), leftmargin=2.5em]
	\item the virtual process $\{\bm Q(t)\}$ has the law of the unrounded system operated under $\pi$; and
	\item the outcomes of the actual system are generated by the same uniform random variables as their virtual counterparts:
	\[
	A^\downarrow_{j',t}=\mathbb{I}\bigl\{V_{j',t}\le\bar a^\downarrow_{j',t}(Q^\downarrow_{j'}(t))\bigr\},
	\qquad
	B^\downarrow_{j,t}=\mathbb{I}\bigl\{U_{j,t}\le\xi^{S\downarrow}_j\bigr\}=B_{j,t},
	\]
	where, as in the proof of Lemma~\ref{lemma:QPTAS_rounding-xi-OPT-comparison}, we take the unobserved service opportunity $B^\downarrow_{j,t}$ to be $\mathbb{I}\{U_{j,t}\le\xi^{S\downarrow}_j\}$ when $Q^{+\downarrow}_j(t)=0$.
\end{enumerate}

\paragraph{The $\beta$-virtual system.} We define the $\beta$-virtual system from the same uniform random variables, with state $\bm Q^\beta(t)=(Q^\beta_j(t))_{j\in[n]}$ and $\bm Q^\beta(1)=\bm 0$. It is used only in the analysis and is not maintained by $\pi^\downarrow$. In period $t$, with $j'=\pi(t,\bm Q(t))$ as above, its joining indicators and service opportunities are
\[
A^\beta_{j,t}:=
\begin{cases}
	\mathbb{I}\bigl\{V_{j',t}\le\beta\,\bar a_{j',t}(Q_{j'}(t))\bigr\}, & \text{if } j=j',\\
	0, & \text{if } j\neq j',
\end{cases}
\qquad
B^\beta_{j,t}:=B_{j,t},
\]
and its queues evolve as $Q^\beta_j(t+1):=\bigl(Q^{+\beta}_j(t)-B_{j,t}\bigr)_+$ with $Q^{+\beta}_j(t):=Q^\beta_j(t)+A^\beta_{j,t}$. 

Note that $A^\beta_{j',t}$ depends on the queue length $Q_{j'}(t)$ of the virtual system, \emph{not} on the queue length $Q^\beta_{j'}(t)$ of the $\beta$-virtual system. Thus, the $\beta$-virtual system does not follow a joining function of its own. It retains a subset of the customers who join the virtual system. Particularly, since $\beta\le1$, we have $\beta\,\bar a_{j',t}(Q_{j'}(t))\le\bar a_{j',t}(Q_{j'}(t))$ and thus $A^\beta_{j,t}\le A_{j,t}$: every customer who joins the $\beta$-virtual system also joins the virtual system. 

The completion indicators of the three systems are
\[
D^\beta_{j,t}=\mathbb{I}\{Q^{+\beta}_j(t)\ge1\}\cdot B_{j,t},
\qquad
D^\downarrow_{j,t}=\mathbb{I}\{Q^{+\downarrow}_j(t)\ge1\}\cdot B_{j,t},
\qquad
D_{j,t}=\mathbb{I}\{Q^+_j(t)\ge1\}\cdot B_{j,t}.
\]

\paragraph{Ordering of the three systems.} We next establish the ordering of the three systems.

\begin{claim}
	\label{claim:QPTAS-rounding-a-ordering}
	For all $j\in[n]$, we have $Q^\beta_j(t)\le Q^\downarrow_j(t)\le Q_j(t)$ for $t\in[T+1]$, and $Q^{+\beta}_j(t)\le Q^{+\downarrow}_j(t)\le Q^+_j(t)$ for $t\in[T]$.
\end{claim}
\begin{proof}
	We proceed by induction on $t$. The claim holds at $t=1$ because all three systems start empty. Suppose that $Q^\beta_j(t)\le Q^\downarrow_j(t)\le Q_j(t)$ for all $j\in[n]$. We first show that this ordering is preserved by the joining step, that is,
	\begin{equation}
		\label{eq:QPTAS-rounding-a-ordering-join}
		Q^{+\beta}_j(t)\le Q^{+\downarrow}_j(t)\le Q^+_j(t),
		\qquad j\in[n].
	\end{equation}
	Let $j'=\pi(t,\bm Q(t))$. For $j\neq j'$, all three joining indicators are zero, so~\eqref{eq:QPTAS-rounding-a-ordering-join} follows immediately. Consider therefore $j=j'$. Since $\pi\in\Pi^{\truncate}$, we have $Q_j(t)<\Lambda_j$.

	\noindent\underline{Actual and virtual systems.}
	\begin{itemize}[leftmargin=*]
		\item If $Q^\downarrow_j(t)=Q_j(t)$, then $\bar a^\downarrow_{j,t}(Q^\downarrow_j(t))=\bar a^\downarrow_{j,t}(Q_j(t))\le\bar a_{j,t}(Q_j(t))$, because the joining probabilities are rounded down. Since both joining indicators are generated by the same uniform random variable $V_{j,t}$, by consequence~(ii), we have $A^\downarrow_{j,t}\le A_{j,t}$, and hence $Q^{+\downarrow}_j(t)\le Q^+_j(t)$.
		\item If $Q^\downarrow_j(t)<Q_j(t)$, then $Q^\downarrow_j(t)+1\le Q_j(t)$, since queue lengths are integer valued. Because both joining indicators belong to $\{0,1\}$, we have $Q^\downarrow_j(t)+A^\downarrow_{j,t}\le Q^\downarrow_j(t)+1\le Q_j(t)\le Q_j(t)+A_{j,t}.$
	\end{itemize}

	\noindent\underline{$\beta$-virtual and actual systems.} By Condition~\eqref{eq:QPTAS_a-rounding-condition}, which applies because $Q_j(t)<\Lambda_j$, and by monotone joining~\eqref{eq:monotone-a} for the rounded joining function, which applies because $Q^\downarrow_j(t)\le Q_j(t)$, we have
	\[
	\beta\,\bar a_{j,t}(Q_j(t))\le\bar a^\downarrow_{j,t}(Q_j(t))\le\bar a^\downarrow_{j,t}(Q^\downarrow_j(t)).
	\]
	Since both joining indicators are generated by the same uniform random variable $V_{j,t}$, we have $A^\beta_{j,t}\le A^\downarrow_{j,t}$, and hence $Q^{+\beta}_j(t)\le Q^{+\downarrow}_j(t)$.

	This proves~\eqref{eq:QPTAS-rounding-a-ordering-join}. Since the three systems share the service opportunity $B_{j,t}$, by consequence~(ii) and the definition of the $\beta$-virtual system and \eqref{eq:QPTAS-rounding-a-ordering-join}, we have
	\[
	\bigl(Q^{+\beta}_j(t)-B_{j,t}\bigr)_+
	\le
	\bigl(Q^{+\downarrow}_j(t)-B_{j,t}\bigr)_+
	\le
	\bigl(Q^+_j(t)-B_{j,t}\bigr)_+.
	\]
	This completes the induction.
\end{proof}

By Claim~\ref{claim:QPTAS-rounding-a-ordering}, whenever $\pi^\downarrow$ recommends a service $j$ in period $t$, we have $Q^\downarrow_j(t)\le Q_j(t)<\Lambda_j$, because $\pi\in\Pi^{\truncate}$. Hence, $\pi^\downarrow\in\Pi^{\truncate}$.

\paragraph{Revenue comparison.} For each service $j$, let
\[
\Dcal^\beta_j:=\sum_{t\in[T]}D^\beta_{j,t},
\qquad
\Dcal^\downarrow_j:=\sum_{t\in[T]}D^\downarrow_{j,t},
\qquad
\Dcal_j:=\sum_{t\in[T]}D_{j,t}
\]
be the numbers of customers who complete service $j$ in the $\beta$-virtual, actual, and virtual systems, respectively.  Since all systems start empty, $\Rcal_{(a^\downarrow)}(\pi^\downarrow)=\sum_{j\in[n]}r^\downarrow_j\,\E[\Dcal^\downarrow_j]$, and, by consequence~(i), $\Rcal_{(a)}(\pi)=\sum_{j\in[n]}r^\downarrow_j\,\E[\Dcal_j]$. We compare the $\beta$-virtual system with each of the other two.

\noindent\underline{$\beta$-virtual and actual systems.} By Claim~\ref{claim:QPTAS-rounding-a-ordering}, we have \[D^\beta_{j,t}=\mathbb{I}\{Q^{+\beta}_j(t)\ge1\}\cdot B_{j,t}\le\mathbb{I}\{Q^{+\downarrow}_j(t)\ge1\}\cdot B_{j,t}=D^\downarrow_{j,t}.\] Summing over $t\in[T]$ gives
\begin{equation}
	\label{eq:QPTAS-rounding-a-beta-actual}
	\E\bigl[\Dcal^\beta_j\bigr]\le\E\bigl[\Dcal^\downarrow_j\bigr].
\end{equation}

\noindent\underline{$\beta$-virtual and virtual systems.} Condition on the complete trajectory of the virtual system, that is, on the joining indicators and service opportunities $\{(A_{j,t},B_{j,t})\}_{j\in[n],\,t\in[T]}$ of all services. This trajectory determines the queue lengths $\bm Q(t)$ in every period, and hence also the recommendation $j'=\pi(t,\bm Q(t))$, which depends on all queue lengths. Now fix a service $j$, and let $\mathcal S_j$ denote the set of customers who join service $j$ in the virtual system and complete service by the end of the horizon, so that $|\mathcal S_j|=\Dcal_j$. We index each customer in $\mathcal S_j$ by the period in which the customer joined service $j$.

A customer who joins service $j$ in period $t$ in the virtual system does so because $V_{j,t}\le\bar a_{j,t}(Q_j(t))$. By the definition of $A^\beta_{j,t}$, the same customer also joins the $\beta$-virtual system if and only if $V_{j,t}\le\beta\,\bar a_{j,t}(Q_j(t))$. To compute the probability of this event, we describe the conditional law of the variables $V_{j',t}$ given the virtual-system trajectory. Before conditioning, all uniform random variables used in the construction are independent, by the validity paragraph above. The trajectory depends on each $V_{j',t}$ only through the joining indicator $A_{j',t}=\mathbb{I}\{V_{j',t}\le\bar a_{j',t}(Q_{j'}(t))\}$, and the threshold $\bar a_{j',t}(Q_{j'}(t))$ is itself fixed by the trajectory since the trajectory determines the queue lengths. Similarly, the trajectory depends on each $U_{j,t}$ only through the service opportunity $B_{j,t}=\mathbb{I}\{U_{j,t}\le\xi^{S\downarrow}_j\}$. Conditioning on the trajectory therefore amounts to restricting each uniform random variable to a fixed interval. Such conditioning preserves independence, and each variable remains uniform on its interval. For a customer in $\mathcal S_j$ who joined in period $t$, the interval of $V_{j,t}$ is $[0,\bar a_{j,t}(Q_j(t))]$. This customer therefore joins the $\beta$-virtual system with probability exactly $\beta$, independently of the other customers in $\mathcal S_j$.

Now consider the customers in $\mathcal S_j$ who also join the $\beta$-virtual system. Because the customers of the $\beta$-virtual system form a subset of those of the virtual system, and the two systems share the same service opportunities, each such customer completes service no later in the $\beta$-virtual system than in the virtual system. This follows from FIFO by induction over customers in order of arrival. Consequently,
\[
\Dcal^\beta_j\ge\sum_{t\in\mathcal S_j}\mathbb{I}\bigl\{V_{j,t}\le\beta\,\bar a_{j,t}(Q_j(t))\bigr\}.
\]
Given the virtual-system trajectory, the right-hand side is distributed as $\operatorname{Binomial}(\Dcal_j,\beta)$, so $\E\bigl[\Dcal^\beta_j\,\big|\,\text{virtual-system trajectory}\bigr]\ge\beta\,\Dcal_j$. Taking expectations gives
\begin{equation}
	\label{eq:QPTAS-rounding-a-beta-virtual}
	\E\bigl[\Dcal^\beta_j\bigr]\ge\beta\cdot\E\bigl[\Dcal_j\bigr].
\end{equation}
Since $\pi^\downarrow\in\Pi^{\truncate}$, combining~\eqref{eq:QPTAS-rounding-a-beta-actual} and~\eqref{eq:QPTAS-rounding-a-beta-virtual} and summing over $j$ proves Claim~\ref{claim:QPTAS-rounding-a-policy}:
\[
\Rcal_{(a^\downarrow)}(\pi^\downarrow)
=
\sum_{j\in[n]}r^\downarrow_j\,\E\bigl[\Dcal^\downarrow_j\bigr]
\ge
\beta\sum_{j\in[n]}r^\downarrow_j\,\E\bigl[\Dcal_j\bigr]
=
\beta\cdot\Rcal_{(a)}(\pi).
\]

\subsection{Proof of Lemma~\ref{lemma:QPTAS_policy-transfer}}
\label{subsec:QPTAS-proof_policy-transfer}

Throughout this proof, the original system has joining functions $\{a_j(\cdot)\}_{j\in[n]}$ and service probabilities $\{\xi^S_j\}_{j\in[n]}$, while the fully rounded system has $\{a^\downarrow_j(\cdot)\}_{j\in[n]}$ and $\{\xi^{S\downarrow}_j\}_{j\in[n]}$. The fully rounded system is the system of \texttt{DSR-M}. For a policy $\pi$ for the original system, recall from Lemma~\ref{lemma:QPTAS_rounding-rewards} that $\Rcal_{(r^\downarrow)}(\pi)$ denotes its expected net revenue when only the rewards are rounded down, from $\{r_j\}_{j\in[n]}$ to $\{r^\downarrow_j\}_{j\in[n]}$; the joining functions and service probabilities remain those of the original system. We use the notation of Appendix~\ref{subsec:QPTAS-coupling-device}, in particular $\bar a_{j,t}(q)=\lambda_t\,a_j(q)$ and $\bar a^\downarrow_{j,t}(q)=\lambda_t\,a^\downarrow_j(q)$, and its coupling device. We prove the following statement, from which Lemma~\ref{lemma:QPTAS_policy-transfer} follows.

\begin{claim}
	\label{claim:QPTAS-transfer-policy}
	For every deterministic Markov policy $\bar\pi$ for the fully rounded system, there exists a policy $\pi\in\Pi$ for the original system such that $\Rcal_{(r^\downarrow)}(\pi)\ge\overline{\Rcal}(\bar\pi)$.
\end{claim}

Since rounding the rewards down can only reduce net revenue, the second inequality of Lemma~\ref{lemma:QPTAS_rounding-rewards} and Claim~\ref{claim:QPTAS-transfer-policy} give
\[
\Rcal(\pi)=\Rcal_{(r)}(\pi)\ge\Rcal_{(r^\downarrow)}(\pi)\ge\overline{\Rcal}(\bar\pi),
\]
which proves Lemma~\ref{lemma:QPTAS_policy-transfer}. Note that the claim does not require $\bar\pi\in\Pi^{\truncate}$. It remains to prove the claim. Fix a deterministic Markov policy $\bar\pi$ for the fully rounded system.

\paragraph{Idea of the construction.} The construction reverses the direction of the proofs of Lemmas~\mbox{\ref{lemma:QPTAS_rounding-xi-OPT-comparison}--\ref{lemma:QPTAS_rounding-a-OPT-comparison}}. The policy $\pi$ we construct for the original system mimics how $\bar\pi$ operates in the fully rounded system. We call the original system, on which $\pi$ acts and in which it observes joining and service outcomes, the \emph{actual system}. The policy $\pi$ maintains a virtual copy of the fully rounded system operated under $\bar\pi$, called the \emph{virtual system}. It applies $\bar\pi$ to the state of the virtual system and makes the same recommendation in the actual system. The joining and service outcomes of the virtual system are generated from the same uniform random variables that drive the actual system (Lemma~\ref{lemma:QPTAS-uniform-reconstruction}).

Different from previous proofs, the actual and virtual systems are \emph{not} ordered by congestion, because the actual system has both higher joining probabilities and higher service probabilities. We therefore introduce a third system, used only in the analysis, which we call the \emph{partially rounded system}. It has the original joining functions and the rounded service probabilities, and it follows the same recommendations. We show that it is at least as congested as both the actual and the virtual systems throughout the horizon (Claim~\ref{claim:QPTAS-transfer-ordering}). Two comparisons then give the result: the partially rounded system earns no more than the actual system, while it earns no less than the virtual system.

\paragraph{Policy construction.} We now construct the policy $\pi$, together with the virtual system that it maintains internally. Let $\bm Q(t)=(Q_j(t))_{j\in[n]}$ denote the queue lengths of the actual system at the beginning of period $t$, and let $\bm Q^\dd(t)=(Q^\dd_j(t))_{j\in[n]}$ denote the state of the virtual system, which $\pi$ updates as described below. The double arrow $\dd$ indicates that both the joining functions and the service probabilities are rounded in the virtual system. Both systems start with empty queues, $\bm Q(1)=\bm Q^\dd(1)=\bm 0$. In period $t$, the policy $\pi$ proceeds as follows.

\begin{center}
	\begin{tcolorbox}[breakable, title={The policy $\pi$ in period $t$}]
		\begin{itemize}[leftmargin=*]
			\item \textbf{Input.} The state $\bm Q(t)$ of the actual system, which is observed, and the state $\bm Q^\dd(t)$ of the virtual system, which $\pi$ maintains internally.

			\item \textbf{Recommendation.} Let $j':=\bar\pi(t,\bm Q^\dd(t))$ be the recommendation of $\bar\pi$ in the virtual system. Recommend $j'$ in the actual system; if $j'=0$, make no recommendation.

			\item \textbf{Observe the actual system.}
			\begin{itemize}
				\item Let $A_{j,t}:=\mathbb{I}\{\text{the arriving customer joins service $j$ in period $t$}\}$, $\forall j\in[n]$.
				\item Let $D_{j,t}:=\mathbb{I}\{\text{service $j$ completes a customer in period $t$}\}$, $\forall j\in[n]$.
				\item Observe the resulting state $\bm Q(t+1)$.
			\end{itemize}

			\item \textbf{Virtual joining.} Set $A^\dd_{j,t}:=0$, $\forall j\in[n]\setminus\{j'\}$. If $j'\neq0$, reconstruct $V_{j',t}$ from $A_{j',t}$ as in Lemma~\ref{lemma:QPTAS-uniform-reconstruction} with $p=\bar a_{j',t}(Q_{j'}(t))$, and define the virtual joining indicator $A^\dd_{j',t}:=\mathbb{I}\{V_{j',t}\le\bar a^\downarrow_{j',t}(Q^\dd_{j'}(t))\}$.

			\item \textbf{Virtual service.} For each $j\in[n]$, let $Q^+_j(t):=Q_j(t)+A_{j,t}$. We construct a uniform random variable $U_{j,t}$ shared with the actual system:
			\begin{itemize}
				\item If $Q^+_j(t)\ge1$, reconstruct $U_{j,t}$ from $D_{j,t}$ as in Lemma~\ref{lemma:QPTAS-uniform-reconstruction} with $p=\xi^S_j$.
				\item If $Q^+_j(t)=0$, draw $U_{j,t}\sim\operatorname{Uniform}([0,1])$ independently.
			\end{itemize}
			Define the virtual service opportunity $B^\dd_{j,t}:=\mathbb{I}\{U_{j,t}\le\xi^{S\downarrow}_j\}$.

			\item \textbf{Virtual update.} Set $Q^{+\dd}_j(t):=Q^\dd_j(t)+A^\dd_{j,t}$ and $Q^\dd_j(t+1):=\bigl(Q^{+\dd}_j(t)-B^\dd_{j,t}\bigr)_+$, $\forall j\in[n]$.
		\end{itemize}
	\end{tcolorbox}
\end{center}

\paragraph{Validity and coupling.} The policy $\pi$ uses only the queue lengths and joining outcomes of the actual system, which the platform observes, together with its own randomization; the virtual system that it maintains internally is computed from these observations and this randomization. Hence, $\pi\in\Pi$.

Let the history up to the start of period~$t$ consist of all joining and service outcomes of the actual system and all random variables generated by $\pi$ before period~$t$. Since $\bar\pi$ is deterministic, this history determines $j'=\bar\pi(t,\bm Q^\dd(t))$. As in the proofs of Lemmas~\mbox{\ref{lemma:QPTAS_rounding-xi-OPT-comparison}--\ref{lemma:QPTAS_rounding-a-OPT-comparison}}, Lemma~\ref{lemma:QPTAS-uniform-reconstruction} implies that, conditional on this history, the variables $V_{j',t},U_{1,t},\ldots,U_{n,t}$ are i.i.d.\ uniform on $[0,1]$. This has two consequences:
\begin{enumerate}[label=(\roman*), leftmargin=2.5em]
	\item the virtual process $\{\bm Q^\dd(t)\}$ has the law of the fully rounded system operated under $\bar\pi$; and
	\item the outcomes of the actual system are generated by the same uniform random variables as their virtual counterparts:
	\[
	A_{j',t}=\mathbb{I}\bigl\{V_{j',t}\le\bar a_{j',t}(Q_{j'}(t))\bigr\},
	\qquad
	B_{j,t}=\mathbb{I}\bigl\{U_{j,t}\le\xi^S_j\bigr\},
	\]
	where, as in the proof of Lemma~\ref{lemma:QPTAS_rounding-xi-OPT-comparison}, we take the unobserved service opportunity $B_{j,t}$ to be $\mathbb{I}\{U_{j,t}\le\xi^S_j\}$ when $Q^+_j(t)=0$.
\end{enumerate}

\paragraph{The partially rounded system.} We define the partially rounded system from the same uniform random variables, with state $\bm Q^\downarrow(t)=(Q^\downarrow_j(t))_{j\in[n]}$ and $\bm Q^\downarrow(1)=\bm 0$. It is used only in the analysis and is not maintained by $\pi$. In period $t$, with $j'=\bar\pi(t,\bm Q^\dd(t))$ as above, its joining indicators and service opportunities are
\[
A^\downarrow_{j,t}:=
\begin{cases}
	\mathbb{I}\bigl\{V_{j',t}\le\bar a_{j',t}(Q^\downarrow_{j'}(t))\bigr\}, & \text{if } j=j',\\
	0, & \text{if } j\neq j',
\end{cases}
\qquad
B^\downarrow_{j,t}:=\mathbb{I}\bigl\{U_{j,t}\le\xi^{S\downarrow}_j\bigr\}=B^\dd_{j,t},
\]
and its queues evolve as $Q^\downarrow_j(t+1):=\bigl(Q^{+\downarrow}_j(t)-B^\downarrow_{j,t}\bigr)_+$ with $Q^{+\downarrow}_j(t):=Q^\downarrow_j(t)+A^\downarrow_{j,t}$. Since $\xi^{S\downarrow}_j\le\xi^S_j$, we have $B^\downarrow_{j,t}\le B_{j,t}$: every service opportunity of the partially rounded system is also a service opportunity of the actual system. The completion indicators of the three systems are
\[
D_{j,t}=\mathbb{I}\{Q^+_j(t)\ge1\}\cdot B_{j,t},
\qquad
D^\downarrow_{j,t}=\mathbb{I}\{Q^{+\downarrow}_j(t)\ge1\}\cdot B^\downarrow_{j,t},
\qquad
D^\dd_{j,t}=\mathbb{I}\{Q^{+\dd}_j(t)\ge1\}\cdot B^\downarrow_{j,t}.
\]

\paragraph{Ordering of the three systems.} We next establish that the partially rounded system is at least as congested as each of the other two.

\begin{claim}
	\label{claim:QPTAS-transfer-ordering}
	For all $j\in[n]$, we have $Q_j(t)\le Q^\downarrow_j(t)$ and $Q^\dd_j(t)\le Q^\downarrow_j(t)$ for $t\in[T+1]$, and $Q^+_j(t)\le Q^{+\downarrow}_j(t)$ and $Q^{+\dd}_j(t)\le Q^{+\downarrow}_j(t)$ for $t\in[T]$.
\end{claim}
\begin{proof}
	We proceed by induction on $t$. The claim holds at $t=1$ because all three systems start empty. Suppose that $Q_j(t)\le Q^\downarrow_j(t)$ and $Q^\dd_j(t)\le Q^\downarrow_j(t)$ for all $j\in[n]$. We first show that these orderings are preserved by the joining step, that is,
	\begin{equation}
		\label{eq:QPTAS-transfer-ordering-join}
		Q^+_j(t)\le Q^{+\downarrow}_j(t)
		\quad\text{and}\quad
		Q^{+\dd}_j(t)\le Q^{+\downarrow}_j(t),
		\qquad j\in[n].
	\end{equation}
	Let $j'=\bar\pi(t,\bm Q^\dd(t))$. For $j\neq j'$, all three joining indicators are zero, so~\eqref{eq:QPTAS-transfer-ordering-join} follows immediately. Consider therefore $j=j'$.

	\noindent\underline{Actual and partially rounded systems.}
	\begin{itemize}[leftmargin=*]
		\item If $Q_j(t)=Q^\downarrow_j(t)$, then the two systems have the same joining probability $\bar a_{j,t}(Q_j(t))$. Since both joining indicators are generated by the same uniform random variable $V_{j,t}$, by consequence~(ii), we have $A_{j,t}=A^\downarrow_{j,t}$, and hence $Q^+_j(t)=Q^{+\downarrow}_j(t)$.
		\item If $Q_j(t)<Q^\downarrow_j(t)$, then $Q_j(t)+1\le Q^\downarrow_j(t)$, since queue lengths are integer valued. Because both joining indicators belong to $\{0,1\}$, we have $Q_j(t)+A_{j,t}\le Q_j(t)+1\le Q^\downarrow_j(t)\le Q^\downarrow_j(t)+A^\downarrow_{j,t}$.
	\end{itemize}

	\noindent\underline{Virtual and partially rounded systems.}
	\begin{itemize}[leftmargin=*]
		\item If $Q^\dd_j(t)=Q^\downarrow_j(t)$, then $\bar a^\downarrow_{j,t}(Q^\dd_j(t))\le\bar a_{j,t}(Q^\dd_j(t))=\bar a_{j,t}(Q^\downarrow_j(t))$, because the joining probabilities are rounded down. Since both joining indicators are generated by the same uniform random variable $V_{j,t}$, we have $A^\dd_{j,t}\le A^\downarrow_{j,t}$, and hence $Q^{+\dd}_j(t)\le Q^{+\downarrow}_j(t)$.
		\item If $Q^\dd_j(t)<Q^\downarrow_j(t)$, then $Q^\dd_j(t)+1\le Q^\downarrow_j(t)$. The same argument as above gives $Q^{+\dd}_j(t)\le Q^{+\downarrow}_j(t)$.
	\end{itemize}
This proves~\eqref{eq:QPTAS-transfer-ordering-join}. For the service step, using $B^\downarrow_{j,t}\le B_{j,t}$ and $B^\downarrow_{j,t}=B^\dd_{j,t}$ completes the induction:
	\begin{align*}
	Q_j(t+1)
	&=
	\bigl(Q^+_j(t)-B_{j,t}\bigr)_+
	\le
	\bigl(Q^{+\downarrow}_j(t)-B^\downarrow_{j,t}\bigr)_+
	=
	Q^\downarrow_j(t+1),\\
	Q^\dd_j(t+1)
	&=
	\bigl(Q^{+\dd}_j(t)-B^\downarrow_{j,t}\bigr)_+
	\le
	\bigl(Q^{+\downarrow}_j(t)-B^\downarrow_{j,t}\bigr)_+
	=
	Q^\downarrow_j(t+1). \qedhere
	\end{align*}
\end{proof}

\paragraph{Revenue comparison.} For each service $j$, let
\[
\Dcal_j:=\sum_{t\in[T]}D_{j,t},
\qquad
\Dcal^\downarrow_j:=\sum_{t\in[T]}D^\downarrow_{j,t},
\qquad
\Dcal^\dd_j:=\sum_{t\in[T]}D^\dd_{j,t}
\]
be the numbers of customers who complete service $j$ in the actual, partially rounded, and virtual systems, respectively. All systems start empty, thus $\Rcal_{(r^\downarrow)}(\pi)=\sum_{j\in[n]}r^\downarrow_j\,\E[\Dcal_j]$, and, by consequence~(i), $\overline{\Rcal}(\bar\pi)=\sum_{j\in[n]}r^\downarrow_j\,\E[\Dcal^\dd_j]$. We compare the partially rounded system with each of the other two.

\noindent\underline{Virtual and partially rounded systems.} By Claim~\ref{claim:QPTAS-transfer-ordering}, we have \[D^\dd_{j,t}=\mathbb{I}\{Q^{+\dd}_j(t)\ge1\}\cdot B^\downarrow_{j,t}\le\mathbb{I}\{Q^{+\downarrow}_j(t)\ge1\}\cdot B^\downarrow_{j,t}=D^\downarrow_{j,t}.\]
Summing over $t\in[T]$ gives $\Dcal^\dd_j\le\Dcal^\downarrow_j$.

\noindent\underline{Actual and partially rounded systems.} Here the completions are compared over the whole horizon rather than period by period. In each period $t$, if $j=j'$, then $Q_j(t)\le Q^\downarrow_j(t)$ by Claim~\ref{claim:QPTAS-transfer-ordering}, so monotone joining~\eqref{eq:monotone-a} gives $\bar a_{j,t}(Q_j(t))\ge\bar a_{j,t}(Q^\downarrow_j(t))$. Since both joining indicators are generated by the same uniform random variable $V_{j,t}$, by consequence~(ii), we have \[A_{j,t} = \mathbb{I}\bigl\{V_{j,t}\le\bar a_{j,t}(Q_{j}(t))\bigr\} \ge \mathbb{I}\bigl\{V_{j,t}\le\bar a_{j,t}(Q^\downarrow_{j}(t))\bigr\} = A^\downarrow_{j,t}.\] 
If $j\neq j'$, both joining indicators $A_{j,t}$ and $A^\downarrow_{j,t}$ are zero. Therefore, we have $A_{j,t} \geq A^\downarrow_{j,t}$ for all $j \in [n]$ and $t \in [T]$. This also implies $\sum_{t \in [T]} A_{j,t} \geq \sum_{t \in [T]} A^\downarrow_{j,t}$ for all $j \in [n]$.
 
 Since all systems start empty, the number of customers who complete service $j$ equals the number who join it minus the number still waiting at the end of the horizon. Hence, by Claim~\ref{claim:QPTAS-transfer-ordering} at $t=T+1$, we have $Q_j(T+1) \leq Q^\downarrow_j(T+1)$ and 
\[
\Dcal_j
=
\sum_{t\in[T]}A_{j,t}-Q_j(T+1)
\ge
\sum_{t\in[T]}A^\downarrow_{j,t}-Q^\downarrow_j(T+1)
=
\Dcal^\downarrow_j.
\]

We therefore have obtained that $\Dcal^\dd_j\le\Dcal^\downarrow_j\le\Dcal_j$ for every $j\in[n]$. Multiplying by $r^\downarrow_j\ge0$, summing over $j$, and taking expectations proves Claim~\ref{claim:QPTAS-transfer-policy}:
\[
\overline{\Rcal}(\bar\pi)
=
\sum_{j\in[n]}r^\downarrow_j\,\E\bigl[\Dcal^\dd_j\bigr]
\le
\sum_{j\in[n]}r^\downarrow_j\,\E\bigl[\Dcal_j\bigr]
=
\Rcal_{(r^\downarrow)}(\pi).
\]

\subsection{Proof of Claim~\ref{claim:QPTAS-selecting-epsilons}}
\label{subsec:QPTAS-choosing-eps-proof}
Notice that
\begin{align}
	\frac{a_j^\downarrow(q)}{a_j(q)}
	=
	\frac{p_j^\downarrow}{p_j} \cdot
	\gamma^{-q \cdot
		\left(
		\frac{1}{\xi_j^{A\downarrow}}
		-
		\frac{1}{\xi_j^A}
		\right)}
	\ge
	\frac{1}{1+\epsilon} \cdot
	\gamma^{-q \cdot
		\left(
		\frac{1}{\xi_j^{A\downarrow}}
		-
		\frac{1}{\xi_j^A}
		\right)}
	\geq
	\frac{1}{1+\epsilon} \cdot
	\gamma^{-q \cdot
		\frac{\epsilon_\xi}{\xi_j^A}
	},
	\label{eq:QPTAS-a-rounding-ratio}
\end{align}
where in the last inequality we invoke $\xi_j^{A\downarrow}
\ge {\xi_j^A} \slash (1+\epsilon_\xi)$. For every $q$ such that $q < \Lambda_j$, we have
\[
\frac{q}{\xi_j^A}
\le
\frac{\Lambda_j}{\xi_j^A}
=
\frac{ \xi^{S \downarrow}_j \cdot (1 + \epsilon_{\xi}) }{\xi_j^A} \cdot \log_\gamma \left(   \frac{ T}{  \xi^{S \downarrow}_j \cdot \epsilon }  \right)
\leq
2 \cdot \log_\gamma
\left(
\frac{T^2}{\epsilon}
\right)
\leq 4 \cdot \log_{\gamma}  \left(
\frac{T}{\epsilon}
\right) := G
\]
where we invoke that $1 / T \leq \xi^{S \downarrow}_j \leq \xi^{S}_j = \xi^A_j$ and $1 + \epsilon_\xi \leq 2$ in the second inequality. Therefore,
\begin{equation*}
	\frac{a_j^\downarrow(q)}{a_j(q)}
	\ge
	\frac{1}{1+\epsilon} \cdot
	\gamma^{- \eps_\xi \cdot G} \geq (1 - \epsilon) \cdot (1 - \eps_\xi  G \cdot \ln \gamma),
\end{equation*}
where we use ${1}/({1+x})\ge1-x$ and $\gamma^{-x}
=
e^{-x\ln\gamma}
\ge
1-x \cdot \ln\gamma$, for any $x > 0$. Therefore,
\begin{align*}
	\frac{a_j^\downarrow(q)}{a_j(q)}
	\ge
	(1-\epsilon)
	\left(
	1-\epsilon_\xi G \ln\gamma
	\right) \ge
	1-\epsilon-\epsilon_\xi G \cdot \ln\gamma = 1 - 5 \epsilon,
\end{align*}
where the equality holds since $\epsilon_\xi G \ln \gamma = \frac{\epsilon}{ \ln(T / \epsilon)  } \cdot 4 \cdot  \frac{  \ln (T / \epsilon)  }{ \ln \gamma  } \cdot \ln \gamma = 4 \epsilon$.

\subsection{Proof of Proposition~\ref{proposition:compressed-MDP-runtime}}
\label{subsec:appendix-compressed-MDP-runtime}

We start with the runtime of calculating the transition kernel of Section~\ref{subsubsec:compressed-transition}. Then, we calculate the total runtime of solving the Bellman equation~\eqref{eq:bellman} via backward induction.

\paragraph{Computability of the kernel.} Fix a state $\bm N$ and an action $a$. Given realizations $\bm z$ of $\bm Z$ and $(u,v)$ of $(A^*,B^*)$, the next compressed state is obtained by moving $z_{k,q}$ services of type $k$ from queue length $q$ to $q-1$ for every $(k,q)$, and the recommended service from $q^*$ to $(q^*+u-v)_+$. The probability of the realization is the product of the corresponding binomial and Bernoulli probabilities, and summing over all realizations that lead to the same next state gives the kernel. 

The kernel is computed by enumerating all realizations, so its cost is governed by their number. Each type-state pair $(k,q)$ with $q\ge1$ contributes $\eta_{k,q}+1$ possible values of $Z_{k,q}$. Hence, the number of possible values of $\bm Z=(Z_{k,q})_{k\in[K],q\in[H_k]}$ is
\begin{equation*}
	\prod_{k=1}^K\prod_{q=1}^{H_k}(\eta_{k,q}+1)
	\le
	(n+1)^M,
\end{equation*}
where $M:=\sum_{k=1}^K(H_k+1)$ is the total number of type-state pairs represented in the compressed state. In addition, the recommended service contributes at most four possible outcome pairs $(u,v)\in\{0,1\}^2$. Therefore, the kernel can be evaluated by enumerating at most $4(n+1)^M$ realizations for each state-action pair. For each realization, its probability is a product of at most $M$ binomial probabilities, one for each type-state pair, and two Bernoulli probabilities, and the next state is obtained by updating at most $M$ entries of $\bm N$; both take $O(M)$ arithmetic operations. Hence, the complete transition distribution from a fixed state-action pair can be enumerated in running time
\begin{equation}
	O\left(
	M(n+1)^M
	\right).
	\label{eq:kernel-computation}
\end{equation}

\paragraph{Solving the Bellman equation~\eqref{eq:bellman}.}

Each compressed state $\bm N(t) = \left(  N_{k,q}(t) \right)_{k \in [K], q \in [H_k]_0}$ satisfies $\sum_{k \in [K]} \sum_{q \in [H_k]_0} N_{k,q}(t) = n$, as there are totally $n$ services. This implies that each component $N_{k,q}(t)$ satisfies $0 \leq N_{k,q}(t) \leq n$. Let $\Scal$ denote the set of possible compressed states. We thus have
\begin{align}
	| \Scal | \leq (n+1)^{ | \{ (k,q) \, : \, k \in [K], q \in [H_k]_0 \} | } = (n+1)^M,
	\label{eq:state-count}
\end{align}
where $M=\sum_{k=1}^K(H_k+1)$ is the total number of type-state pairs, as defined above.

At every state, there are at most $M$ feasible type-state pairs and therefore at most $|\mathcal A(\bm N)|\le M+1$ feasible actions, including the null action. By~\eqref{eq:kernel-computation}, the transition distribution for each state-action pair can be enumerated in $O\left(M(n+1)^M\right)$ time. Therefore, one Bellman backup at a fixed state can be performed in $O\left(
(M+1) \cdot M(n+1)^M
\right)$ time. Combining this bound with~\eqref{eq:state-count}, the backward-induction implementation over $T$ periods runs in
\begin{equation}
	O\left( T \cdot |\Scal| \cdot 
	M(M+1)(n+1)^{M}
	\right) = 
	O\left(
	T M^2 (n+1)^{2M}
	\right)
	=
	T\cdot
	n^{
		O\left(
		\log^5(T/\epsilon)/\epsilon^3
		\right)},
	\label{eq:runtime-general-M}
\end{equation}
where the second equality uses
\[
M
\le
K \left(1+\max_k H_k\right)
=
\underbrace{O\left(
	\frac{\log^4(T/\epsilon)}{\epsilon^3}
	\right)}_{\text{by~\eqref{eq:number_of_types}}}
\cdot
\underbrace{O\left(
	\log \left( \frac{T}{\epsilon} \right)
	\right)}_{\text{by~\eqref{eq:type-queue-length}}}
=
O\left(
\frac{\log^5(T/\epsilon)}{\epsilon^3}
\right).
\]

\section{Omitted Proofs from Section~\ref{sec:LP}}

\subsection{Proof of Lemma~\ref{lemma:objective-2}}
\label{appendix-subsec:objective-reformulation}

Each customer who joins service~$j$ during the horizon pays $r_j$, and each such customer who
has not completed service by the end of period~$T$ is refunded $r_j$. Hence, the net revenue
equals $r_j$ for each customer who joins service~$j$ during the horizon and completes service
by the end of period~$T$, summed over $j\in[n]$. Condition on the history up to the beginning
of period~$t$ together with the recommendation $J^\pi_t = j \ne 0$. A customer then arrives
and joins with probability $\bar a_{j,t}(Q^\pi_j(t))$, and by FIFO, she completes service with
probability $\Pp(\sum_{\tau=t}^T B_{j,\tau} \ge Q^\pi_j(t) + 1) = \kappa_{j,Q^\pi_j(t),t}$, since
the capacities $B_{j,\tau}$, $\tau\ge t$, are independent of the history and of the joining
decision. The conditional expected revenue from the customer of period~$t$ is therefore
$I_{j,t}(Q^\pi_j(t))$; if $J^\pi_t=0$, it is $0=I_{0,t}(Q_0(t))$. Taking expectations and
summing over $t \in [T]$ proves the lemma.

\subsection{Proof of Theorem~\ref{theorem:greedy}}
\label{appendix-subsec:proof-of-greedy}

We first fix any policy $\pi$. We will construct two stochastic systems of the same \ref{problem:DSR-abstract} instance such that one system evolves under the greedy policy and the other under the policy $\pi$. We couple the two systems so that they have the identical service capacities $B_{j,t}$ for all $j \in [n]$ and $t \in [T]$. In the meantime, the two systems have independent arrivals, joining events, and any randomization under policy $\pi$. Let $\bm{Q}^{\greedy}$ and $\bm{Q}^{\pi}$ denote the queue processes of the two systems. Also, we use  $\Fcal_t$ to denote the filtration generated by both systems' histories before period $t$ together with both recommendations $J^{\greedy}_t$ and $J^{\pi}_t$. The capacities $B_{j,\tau}$ for $j \in [n]$ and $\tau \geq t$ are independent of $\Fcal_t$.

Let $P_t$ be policy $\pi$'s realized revenue from period $t$, i.e.,
\[
P_t := r_j \cdot \mathbb{I} \left\lbrace  \text{$\pi$'s customer joins $j$ in period $t$}  \right\rbrace \cdot  \mathbb{I} \left\lbrace \sum_{\tau = t}^T B_{j,\tau} \geq Q^{\pi}_j(t) + 1  \right\rbrace.
\]
Therefore, $\Rcal(\pi) = \Ebb \left[ \sum_{t=1}^T P_t \right]$. Similarly, let $G_t := \max_{j} I_{j,t}(Q^{\greedy}_j(t))$ and $\Rcal(\greedy) = \Ebb \left[ \sum_{t = 1}^T G_t \right]$.

For each period $t$ with $J^{\pi}_t = j \neq 0$, we define the following $\Fcal_t$-measurable events: $\mathcal{E}^1_t := \{  Q^{\greedy}_j(t) \leq Q^{\pi}_j(t)   \}$ and $ \mathcal{E}^2_t := \{  Q^{\greedy}_j(t) > Q^{\pi}_j(t)   \}$. We now analyze two cases:

\underline{Case (1): greedy's queue of $j$ is no longer than $\pi$'s}. Since both $a_j(q)$ and $\kappa_{j,q,t}$ are nonincreasing in $q$, we know that $I_{j,t}(q)$ is also nonincreasing in $q$. Therefore, on event $\mathcal{E}^1_t$, we have
\[
\Ebb \left[   P_t \cdot \mathbb{I}\left\lbrace \mathcal{E}^1_ t  \right\rbrace  \mid \Fcal_t \right] = \mathbb{I}\left\lbrace \mathcal{E}^1_ t  \right\rbrace \cdot I_{j,t} \left( Q^\pi_j(t) \right) \leq \mathbb{I}\left\lbrace \mathcal{E}^1_ t  \right\rbrace \cdot I_{j,t} \left( Q^{\greedy}_j(t) \right) \leq G_t.
\]
Summing over $t$ and taking expectations, we have $\Ebb \left[  \sum_{t=1}^T P_t \cdot \mathbb{I}\left\lbrace \mathcal{E}^1_ t  \right\rbrace  \right] \leq \Ebb \left[   \sum_{t=1}^T G_t \right] = \mathcal{R}(\greedy)$.

\underline{Case (2): greedy's queue of $j$ is longer than $\pi$'s}. On event $\mathcal{E}^2_t$, there exists a customer in service $j$ of the greedy system at position $Q^{\pi}_j(t) + 1$ at the beginning of time period $t$. We call this customer $c_t$. On event $\mathcal{E}^2_t$, if $\pi$'s customer joins service $j$ in period $t$, this customer is at position $Q^{\pi}_j(t)+1$ and completes service by the end of the horizon if and only if $\sum_{\tau=t}^{T} B_{j,\tau} \ge Q^{\pi}_j(t)+1$. In the greedy system, customer $c_t$ is at the same position, and customers who join later are behind $c_t$ under FIFO. Since both systems have the same service capacities, $c_t$ completes service under exactly the same condition. Thus,
\[
P_t  \cdot \mathbb{I}\left\lbrace \mathcal{E}^2_t  \right\rbrace  \leq r_j \cdot \mathbb{I} \left\lbrace \text{$c_t$ completes service in the greedy system by the end of the horizon} \right\rbrace.
\]
We claim that customer $c_t$ only corresponds to a time period $t$ with $P_t > 0$. In other words, there do not exist two periods $t$ and $t'$ such that $P_t > 0$, $P_{t'} >0$, while $c_t$ and $c_{t'}$ correspond to the same customer in the greedy system. If this claim is true, then we easily have
\[
\Ebb \left[  \sum_{t=1}^T P_t  \cdot \mathbb{I}\left\lbrace \mathcal{E}^2_t  \right\rbrace  \right] \leq \Rcal(\greedy).
\]
This is because the assignment $t \mapsto c_t$ maps the periods in which $\pi$ earns revenue on $\mathcal{E}^2_t$ to distinct customers who complete service in the greedy system, and each such customer pays the same revenue $r_j$. The revenue that $\pi$ earns in these periods is thus at most the total revenue of the greedy system. 

To prove this claim, it suffices to show that in any later period $t' \geq  t+1$ such that customer $c_t$ is still in the greedy system, her location in service $j$ is strictly prior to $Q^{\pi}_j(t') + 1$. We prove this by induction. To this end, we use $\ell(s)$ to denote the location of customer $c_t$ at the beginning of time period $s \geq t$. By definition, we have $\ell(t) = Q^{\pi}_j(t) + 1$. 

\begin{itemize}[leftmargin=*]
	\item Consider time period $t+1$. If customer $c_t$ is still in the system, then $\ell(t+1) = \ell(t) - B_{j,t} \geq 1$. Also, recall that $P_t > 0$, which implies that a customer joins service $j$ in the $\pi$ system in time period $t$. Thus, we know $Q^{\pi}(t+1) = Q^{\pi}(t) + 1 - B_{j,t} =\ell(t+1)$. Hence, $Q^{\pi}(t+1) + 1 > \ell(t+1)$.
	
	\item Assume that the claim holds for a time $s \geq t + 1$. That is, $Q^{\pi}_j(s) + 1> \ell(s)$, which implies $Q^{\pi}_j(s) \geq \ell(s)$.
	
	\item Consider time period $s+1$. If customer $c_t$ is still in the system, then $\ell(s+1) = \ell(s) - B_{j,s} \geq 1$. In the meantime, $Q^{\pi}_j(s+1) \geq Q^{\pi}_j(s) - B_{j,s} \geq \ell(s) - B_{j,s} = \ell(s+1)$, where the second inequality follows the induction hypothesis. Therefore, $Q^{\pi}_j(s+1) + 1 > \ell(s+1)$, which completes the proof of the claim.
	
\end{itemize}
To obtain the $1/2$ approximation guarantee, we simply observe that
\begin{align}
	\Rcal(\pi) = \Ebb \left[ \sum_{t=1}^T P_t \right] = \Ebb \left[  \left(  \sum_{t=1}^T P_t \cdot \mathbb{I} \left\lbrace \mathcal{E}^1_t \right\rbrace \right) +   \left(  \sum_{t=1}^T P_t \cdot \mathbb{I} \left\lbrace \mathcal{E}^2_t \right\rbrace \right) \right] \leq 2 \cdot \Rcal(\greedy).
\end{align}
By taking $\pi$ as the optimal policy, we establish that $\Rcal(\greedy) \geq \frac{1}{2} \cdot \OPT$.

\subsection{Proof of Proposition~\ref{proposition:greedy-tightness}}
\label{appendix-subsec:greedy-tightness}

\paragraph{The instance.} Fix parameters $\xi, \delta, \eta \in (0,1)$ with $h := \xi(1+\delta) \le 1$, and a horizon $T \ge 2$. Let $n = 2$ and $\lambda_t = 1$ for all $t$. We construct the two services as follows.
\begin{itemize}[leftmargin=2em]
	\item \emph{Service 1} (valuable, slow): $r_1 = \frac{1}{h}$,
	$B_{1,t} \sim \mathrm{Bernoulli}(\xi)$, and
	$a_1(0) = 1$, $a_1(q) = h$ for $q \ge 1$.
	\item \emph{Service 2} (cheaper, never congested): $r_2 = 1$, $B_{2,t} \equiv 1$
	(i.e.\ $\xi_2 = 1$), $a_2(0) = 1-\eta$, and $a_2(q) = 0$ for $q \ge 1$.
\end{itemize}
Define $p := 1 - \eta$ and $\Bcal_t := \sum_{\tau=t}^{T} B_{1,\tau} \sim \mathrm{Bin}(T-t+1, \xi)$. Thus, $\kappa_{1,q,t} = \Pp(\Bcal_t \ge q+1)$.

\begin{proposition}\label{prop:half}
	For the instance above,
	\[
	\frac{\Rcal(\greedy)}{\OPT} \;\le\;
	\frac{1 + (1-\eta)\delta + (1 - \eta)(1+\delta)(E_T + 1)/T}{1 + (1-\eta)(1-\xi)(1+\delta)(1 - 1/T)},
	\]
	where $E_T := 1 + {\xi}^{-1} \cdot \left({1 + 2\sqrt{2T\ln(2T)} + \sqrt{2\xi T \ln(2/\eta)}}\right)$. Since $E_T = O(\sqrt{T \log T})$ for fixed parameters, we have
	\[
	\limsup_{T\to\infty} \frac{\Rcal(\greedy)}{\OPT} \;\le\; \rho(\xi,\delta,\eta) :=
	\frac{1 + (1-\eta)\delta}{1 + (1-\eta)(1-\xi)(1+\delta)},
	\]
	and $\rho(\xi,\delta,\eta) \to \frac12$ as $(\xi,\delta,\eta) \to 0$.
\end{proposition}

Proposition~\ref{prop:half} implies that the factor $\frac{1}{2}$ in Theorem~\ref{theorem:greedy} is tight: for every $\epsilon > 0$ there is an instance of \ref{problem:DSR-abstract} of two services on which $\mathcal{R}(\greedy) \le (\tfrac{1}{2} + \epsilon)\,\OPT$. Particularly, we choose $\xi, \delta, \eta$ small enough that $\rho(\xi,\delta,\eta) \le \frac{1}{2} + \frac{\epsilon}{2}$, then choose $T$ large enough that the finite-$T$ bound of Proposition~\ref{prop:half} is at most $\frac{1}{2} + \epsilon$.

\paragraph{Observations.} The following observations are helpful in establishing Proposition~\ref{prop:half}.
\begin{enumerate}[label=(\roman*), leftmargin=2em]
	\item Service 2 is always empty, because
	$B_{2,t} = 1$
	\item Service 1's greedy index equals $\kappa_{1,0,t}/h$ if $Q_1(t) = 0$ and
	$r_1 h\, \kappa_{1,Q_1(t),t} = \kappa_{1,Q_1(t),t}$ if $Q_1(t) \ge 1$. Since $h \le 1$,
	in both cases the index is at least $\kappa_{1,Q_1(t),t}$.
	\item \emph{Primitive randomness.} Since $\lambda_t = 1$, a customer arrives in every period.
	We represent joining by independent $V_t \sim \mathrm{Uniform}[0,1]$. That is, the customer joins the
	recommended service $j$ if and only if $V_t \le a_j(Q_j(t))$. The sequences $(V_t)$ and
	$(B_{1,t})$ are independent.
	\item \emph{Revenue.} A policy's expected net revenue equals
	$\sum_j r_j\, \E[\Ccal_j]$, where $\Ccal_j$ is the number of customers who join service~$j$
	and complete service by the end of period~$T$.
\end{enumerate}

\paragraph{Claims.} We establish Proposition~\ref{prop:half} through a collection of claims. Particularly, in what follows, Claim~\ref{claim:tightness-opt-LB} constructs a lower bound of $\OPT$ and Claim~\ref{claim:tightness-greedy-revenue} constructs an upper bound of the performance of the greedy policy. Claim~\ref{claim:tightness-backlog} establishes that the greedy policy recommends service 2 only when service 1 is badly backlogged. Claim~\ref{claim:tightness-drift} shows that before the greedy policy first recommends service 2, the backlog in service 1 grows at rate $\delta\xi$.

\begin{claim}\label{claim:tightness-opt-LB}
	$\OPT \ge {T}/({1+\delta}) + (1-\eta)(1-\xi)(T-1)$.
\end{claim}

\begin{proof}
	Consider the policy $\bar\pi$ that recommends service~1 if $Q_1(t) = 0$ and service~2 otherwise.
	If $Q_1(t) = 0$, the arriving customer joins service~1 with probability $a_1(0) = 1$, so
	$Q_1^+(t) = 1$. If $Q_1(t) = 1$, service~2 is recommended and $Q_1^+(t) = 1$. Therefore, in either way
	$Q_1(t+1) = 1 - B_{1,t}$, so $Q_1(t) \in \{0,1\}$ and $\mathbb{I}\{Q_1(t) = 0\} = B_{1,t-1}$ for $t \ge 2$. Also, $Q_1(1) = 0$.
	
	Service~1 receives one customer in each period with $Q_1(t) = 0$, and the number of customers who do
	not complete by $T$ is $Q_1(T+1) = 1 - B_{1,T}$. Service~2 is recommended in each period
	with $Q_1(t) = 1$, and each such recommendation yields a completed customer with probability $p$.
	Therefore, $\mathcal{R}(\bar\pi)
	= r_1 + (T-1) \cdot \left(  r_1 \xi + r_2 (1  - \xi) p \right) - r_1 (1 - \xi) = r_1\cdot \xi T + r_2 \cdot p(1-\xi)(T-1)$. Along with $r_1 \xi = \xi/h = 1/(1+\delta)$, $r_2 = 1$, and $\OPT \ge \mathcal{R}(\bar\pi)$, we prove the claim.
\end{proof}

\begin{claim}\label{claim:tightness-greedy-revenue}
	Let $M := | \{t \in [T] : J^{\greedy}_t = 2\}|$ be the number of periods in which greedy recommends
	service~2. Then $\Rcal(\greedy) \le {T}/({1+\delta}) + (1-\eta)\,\E[M]$.
\end{claim}

\begin{proof}
	Every customer of service~1 who completes service uses a service opportunity $B_{1,t} = 1$
	with $t \le T$, so the expected number of such customers is at most $\E[\sum_{t \le T} B_{1,t}] = \xi T$,
	contributing to expected net revenue at most $r_1 \xi T = T/(1+\delta)$. By Observation (i), each recommendation of service~2 yields
	revenue $r_2 = 1$ with probability $p = 1 - \eta$, contributing $(1 - \eta)\,\E[M]$.
\end{proof}

\begin{claim}\label{claim:tightness-backlog}
	If $J^\greedy_t = 2$, then $Q_1(t) \ge \xi(T-t+1) - \sqrt{2\xi T \ln(2/\eta)}$.
\end{claim}

\begin{proof}
	If $J^\greedy_t = 2$, service~1's index is at most that of service~2, which is $p$ by Observation (i). Thus, if $J^\greedy_t = 2$, then $p \geq I_{1,t}(Q_1(t)) \geq \kappa_{1,Q_1(t),t}$ by Observation (ii). Let $\Bcal_t \sim \mathrm{Bin}(m, \xi)$ with $m = T - t + 1 \ge 1$ and mean $\mu = \xi m > 0$. Then the above condition is equivalent to 
	\[\eta \leq \Pp(\Bcal_t \le Q_1(t)).\]
	Let $a := \sqrt{2\mu \ln(2/\eta)}$. For contradiction, suppose $Q_1(t) < \mu - a$. The Chernoff bound for the
	lower tail of a binomial, $\Pp(\Bcal_t \le \mu - a) \le \exp(-a^2/(2\mu))$, gives
	$\Pp(\Bcal_t \le Q_1(t)) \le \eta/2 < \eta$, a contradiction. Hence
	$Q_1(t) \ge \mu - a \ge \xi(T-t+1) - \sqrt{2\xi T \ln(2/\eta)}$, using $\mu \le \xi T$.
\end{proof}

\begin{claim}\label{claim:tightness-drift}
	Let $\tau := \min\{t \in [T] : J^{\greedy}_t = 2\}$, with $\tau := T+1$ if there is no such $t$, and let
	$x_T := \sqrt{2T\ln(2T)}$. There is an event $\mathcal{G}$ with $\Pp(\mathcal{G}) \ge 1 - 1/T$
	on which, if $\tau \le T$,
	\[
	Q_1(\tau) \le 1 + \delta\xi(\tau - 1) + 2x_T .
	\]
\end{claim}

\begin{proof}
	Let $X$ be the queue of service~1 under the policy that always recommends service~1, driven by
	the same $V_t$ and $B_{1,t}$ for $t \in [T]$. Set $X(1) = 0$ and
	$X(t+1) = \big(X(t) + \mathbb{I}\{V_t \le a_1(X(t))\} - B_{1,t}\big)_+$. Notice that for $t < \tau$, the greedy policy recommends service~1, so $Q_1$ follows the same recursion with the same
	inputs, and $Q_1(t) = X(t)$ for all $t \le \tau$.
	
	Let $D_s := \mathbb{I}\{V_s \le h\} - B_{1,s}$. These are i.i.d.\ with values in $\{-1,0,1\}$ and mean
	$h - \xi = \delta\xi$. If $X(s) \ge 1$, then $a_1(X(s)) = h$ and $X(s) + D_s \ge 0$, so
	$X(s+1) = X(s) + D_s$. If $X(s) = 0$, then $a_1(0) = 1$ and $X(s+1) = 1 - B_{1,s} \le 1$.
	Fix $t \le T$ and let $k := \max\{s \le t-1 : X(s) = 0\}$, which exists because $X(1) = 0$.
	Then $X(t) = X(k+1) + \sum_{s=k+1}^{t-1} D_s = 1 + \sum_{s=k+1}^{t-1} D_s$, where the first equality holds since $X(k') \geq 1$ for $k+1 \leq k' \leq t-1$. 
	
	Let $W_j := \sum_{s=1}^{j} (D_s - \delta\xi)$, with $W_0 = 0$. Then
	$\sum_{s=k+1}^{t-1} D_s = \delta\xi(t-1-k) + W_{t-1} - W_k \le \delta\xi(t-1) + 2\max_{j \le T}|W_j|$.
	The increments $D_s - \delta\xi$ are independent, have mean zero, and take values in an interval
	of length~2, so Hoeffding's maximal inequality gives
	$\Pp(\max_{j \le T} |W_j| \ge x) \le 2\exp(-x^2/(2T))$. With $x = x_T$ the right-hand side
	equals $1/T$. On $\mathcal{G} := \{\max_{j\le T}|W_j| < x_T\}$ we obtain
	$X(t) \le 1 + \delta\xi(t-1) + 2x_T$ for all $t \le T$, and in particular for $t = \tau$.
\end{proof}

\paragraph{Completing the proof of Theorem~\ref{theorem:greedy}.} Following the definitions and notation in previous claims, we know that all recommendations of service~2 in the greedy policy occur in periods $t \ge \tau$, so $M \le T - \tau + 1$.
On $\mathcal{G}$ with $\tau \le T$, Claims~\ref{claim:tightness-backlog} and~\ref{claim:tightness-drift} give
\[
\xi(T-\tau+1) - \sqrt{2\xi T\ln(2/\eta)} \;\le\; Q_1(\tau) \;\le\; 1 + \delta\xi(\tau-1) + 2x_T,
\]
hence $\xi(1+\delta)\,\tau \ge \xi T - 1 - 2x_T - \sqrt{2\xi T\ln(2/\eta)}$, which implies
\[
M \;\le\; T - \tau + 1 \;\le\; \frac{\delta}{1+\delta}\,T + 1 + \frac{1 + 2x_T + \sqrt{2\xi T\ln(2/\eta)}}{\xi}
\;=\; \frac{\delta}{1+\delta}\,T + E_T .
\]
This also holds if $\tau = T+1$, since then $M = 0$. Off $\mathcal{G}$, $M \le T$. Therefore,
\[\E[M] \le \frac{\delta}{1+\delta}T + E_T + T\,\Pp(\mathcal{G}^c) \le \frac{\delta}{1+\delta}T + E_T + 1.
\]
By Claim~\ref{claim:tightness-greedy-revenue},
$\Rcal(\greedy) \le \frac{T}{1+\delta} + (1-\eta)\big(\frac{\delta T}{1+\delta} + E_T + 1\big)$.
Dividing this by the lower bound of Claim~\ref{claim:tightness-opt-LB} and multiplying numerator and denominator
by $(1+\delta)/T$, we have
\[
\frac{\Rcal(\greedy)}{ \OPT  } \leq \frac{  \frac{T}{1 + \delta} +(1-\eta)\big(\frac{\delta T}{1+\delta} + E_T + 1\big)  }{   \frac{T}{1 + \delta} + (1 - \eta)(1- \xi)(T-1) } \leq \frac{  1 + (1 - \eta) \delta + \frac{ (E_T + 1)(1 + \delta)(1 - \eta) }{T}  }{   1 + (1-\eta)(1 - \xi)(1 + \delta) \cdot \frac{T-1}{T} } \longrightarrow \rho(\xi,\delta,\eta)
\]
as $T \rightarrow \infty$. Moreover,  $\rho(\xi,\delta,\eta) \to (1+0)/(1+1) = 1/2$ as
$(\xi,\delta,\eta) \to 0$. \qed

\subsection{Proof of Proposition~\ref{prop:LP_upperbound}}
\label{subsec:LP-upperbound-proof}

Fix a policy $\pi$. It may be randomized and history-dependent. Let $\Pp_\pi$ and $\E_\pi$
denote probability and expectation under $\pi$. Let $J^\pi_t$ denote the contingent
recommendation of $\pi$ in period $t$ (Section~\ref{subsec:MOLP}). All queues start empty, and
at most one customer joins in each period. Hence, $Q_j(t)\in[t-1]_0$ for all $j$ and $t$. For
$j\in[n]$, $t\in[T]$, and $q\in[t-1]_0$, define
\[
\bar y_{j,q,t}:=\Pp_\pi\bigl(Q_j(t)=q\bigr)
\qquad\text{and}\qquad
\bar x_{j,q,t}:=\Pp_\pi\bigl(Q_j(t)=q,\ J^\pi_t=j\bigr).
\]
Thus, $\bar y_{j,q,t}$ is the probability that service $j$ has queue length $q$ at the
beginning of period $t$. The probability $\bar x_{j,q,t}$ additionally requires that $\pi$
recommends service $j$. We show that $(\bar x,\bar y)$ is feasible for~\ref{eq:MOLP} and that
its objective value equals $\Rcal(\pi)$.

\paragraph{Feasibility.} Constraint~(1) holds obviously. %
Constraint~(2) holds because
$\sum_{j\in[n]}\sum_{q\in[t-1]_0}\bar x_{j,q,t}=\Pp_\pi(J^\pi_t\ne0)\le1$. Constraints~(4)
and~(5) hold because all queues start empty. For Constraint~(3), we fix $j$ and $t$. Let $A_{j,t}$ be the indicator that a customer arrives in
period $t$ and joins service $j$. Condition on the history up to the beginning of period $t$
and on $J^\pi_t$. If $J^\pi_t=j$, then $A_{j,t}=1$ with probability $\bar a_{j,t}(Q_j(t))$.
Otherwise, $A_{j,t}=0$. Hence, for $q\in[t-1]_0$,
\[
\Pp_\pi\bigl(Q_j(t)=q,\ A_{j,t}=1\bigr)=\bar a_{j,t}(q)\,\bar x_{j,q,t}
\qquad\text{and}\qquad
\Pp_\pi\bigl(Q_j(t)=q,\ A_{j,t}=0\bigr)=\bar y_{j,q,t}-\bar a_{j,t}(q)\,\bar x_{j,q,t}.
\]
Service then gives $Q_j(t+1)=(Q_j(t)+A_{j,t}-B_{j,t})_+$. The capacity $B_{j,t}$ is
independent of $(Q_j(t),J^\pi_t,A_{j,t})$. Thus, $\Pp_\pi\bigl(Q_j(t+1)=q'\bigm|Q_j(t)=q,\ A_{j,t}=u\bigr)=K_j(q+u,q')$ for $u\in\{0,1\}$. By the law of total probability,
\[
\bar y_{j,q',t+1} = \Pp_\pi\bigl(Q_j(t+1)=q'\bigr)
=\sum_{q=0}^{t-1}\sum_{u\in\{0,1\}}
\Pp_\pi\bigl(Q_j(t)=q,\ A_{j,t}=u\bigr)\,K_j(q+u,q').
\]
Substituting the two joint probabilities above gives~\eqref{eq:flow-kernel}. This is Constraint~(3).

\paragraph{Objective.} By Lemma~\ref{lemma:objective-2},
$\Rcal(\pi)=\sum_{t=1}^T\E_\pi\bigl[I_{J^\pi_t,t}\bigl(Q_{J^\pi_t}(t)\bigr)\bigr]$. Fix $t$.
Since $I_{0,t}(\cdot)=0$,
\[
\E_\pi\bigl[I_{J^\pi_t,t}\bigl(Q_{J^\pi_t}(t)\bigr)\bigr]
=\sum_{j=1}^n\sum_{q=0}^{t-1}I_{j,t}(q)\,\Pp_\pi\bigl(J^\pi_t=j,\ Q_j(t)=q\bigr)
=\sum_{j=1}^n\sum_{q=0}^{t-1}I_{j,t}(q)\,\bar x_{j,q,t}.
\]
Summing over $t\in[T]$, the objective value of $(\bar x,\bar y)$ in~\ref{eq:MOLP} equals
$\Rcal(\pi)$. Hence, $Z^{\LP}\ge\Rcal(\pi)$. Since $\pi$ is arbitrary,
$Z^{\LP}\ge\sup_\pi\Rcal(\pi)=\OPT$.

\subsection{Proof of Theorem~\ref{thm:LPGA_approx_ratio}}
\label{subsec:proof_of_LP_performance}

We show that $\Ebb[R^{\ALG}]\ge(1-1/e)\,Z^{\LP}$, where $R^{\ALG}$ denotes the realized net revenue of {\LPGA}, so that $\Rcal(\LPGA)=\Ebb[R^{\ALG}]$; Theorem~\ref{thm:LPGA_approx_ratio} then follows from Proposition~\ref{prop:LP_upperbound}. The proof has three steps, which establish the properties of the virtual system outlined in Section~\ref{subsec:LP-main-theorem}.
\begin{itemize}[leftmargin=*]
	\item \emph{Step 1} shows that each actual queue is never longer than its virtual counterpart (Lemma~\ref{lem:path-dom}). Since $a_j(q)$ and $\kappa_{j,q,t}$ are nonincreasing in $q$, so is the index $I_{j,t}(q)$. Hence, $I_{j,t}(Q_j(t))\ge I_{j,t}(\widehat Q_j(t))$ for every service $j$; that is, the index of each service in the actual system is no smaller than in the virtual system. This is property~(iii) in Section~\ref{subsec:LP-main-theorem}.
	\item \emph{Step 2} shows that the virtual queues have exactly the marginals $\bm y^*$ planned by the LP and are independent across services (Lemma~\ref{lem:virtual-marginals}); these are properties~(i) and~(ii) in Section~\ref{subsec:LP-main-theorem}.
	\item \emph{Step 3} compares the revenue of {\LPGA} in each period with the corresponding term of $Z^{\LP}$. By Step~1, {\LPGA} earns at least the largest virtual value among the activated services. By Step~2, these values are independent across services and the expected number of activated services is at most one, so a correlation-gap inequality shows that this largest value retains a $(1-1/e)$ fraction of their sum, whose expectation is the period's term of $Z^{\LP}$.
\end{itemize}
Throughout the proof, let $\Fcal_t$ denote the information available to {\LPGA} at the beginning
of period $t$: the actual and virtual queue lengths up to period $t$, the arrival and joining
outcomes before period $t$, and all random draws of {\LPGA} before period $t$ (Lines~3, 10, 12,
and~18 of Algorithm~\ref{alg:lpga}).

\subsubsection{Step 1: Compare the Actual and Virtual Queues}

{\LPGA} observes only the joining outcomes $\bm A_t$ and the queue lengths $\bm Q(t+1)$ of the actual system, not its service capacities $B_{j,t}$. Instead, in Lines~15--18, it constructs $\widehat B_{j,t}$ from the observed transition of the actual queue. The following lemma shows that $\widehat B_{j,t}$ can also serve as the service capacity of the actual system. Recall
that $Q_j^+(t)=Q_j(t)+A_{j,t}$ is the queue length of service $j$ after joining in period $t$.

\begin{lemma}
	\label{lem:capacity-reconstruction}
	For every service $j$ and period $t$, the following hold.
	\begin{enumerate}[label=(\alph*)]
		\item $Q_j(t+1)=\bigl(Q_j^+(t)-\widehat B_{j,t}\bigr)_+$.
		\item Conditional on $\Fcal_t$ and on the outcomes of period $t$ before service, namely
		$\bm\zeta_t$, $\bm A_t$, and $\widehat{\bm A}_t$, the variables
		$\widehat B_{1,t},\ldots,\widehat B_{n,t}$ are independent, and
		$\Pp(\widehat B_{j,t}=b)=\xi_{j,b}$ for all $b\in\Zp$.
	\end{enumerate}
\end{lemma}

\begin{proof}
	The actual queues evolve as $Q_j(t+1)=(Q_j^+(t)-B_{j,t})_+$, where $B_{j,t}$ has distribution
	$\{\xi_{j,b}\}_{b\ge0}$. The capacities $B_{1,t},\ldots,B_{n,t}$ are mutually independent, and they are independent of $\Fcal_t$ and of the outcomes of period $t$ before service.

	For part~(a), if $Q_j(t+1)>0$, Line~16 sets $\widehat B_{j,t}=Q_j^+(t)-Q_j(t+1)$, so
	$(Q_j^+(t)-\widehat B_{j,t})_+=Q_j(t+1)$. If $Q_j(t+1)=0$, Line~18 draws
	$\widehat B_{j,t}\ge Q_j^+(t)$, so $(Q_j^+(t)-\widehat B_{j,t})_+=0=Q_j(t+1)$.

	For part~(b), condition on $\Fcal_t$, $\bm\zeta_t$, $\bm A_t$, and $\widehat{\bm A}_t$, which
	fixes $q^+:=Q_j^+(t)$. On the event $\{B_{j,t}<q^+\}$, we have $Q_j(t+1)>0$ and hence
	$\widehat B_{j,t}=B_{j,t}$. On the complementary event $\{B_{j,t}\ge q^+\}$, we have
	$Q_j(t+1)=0$, and $\widehat B_{j,t}$ is drawn independently from~\eqref{eq:bhat-conditional}.
	Thus, for any $b$,
	\begin{align*}
		\Pp(\widehat B_{j,t}=b)
		&=\Pp\bigl(\widehat B_{j,t}=b,\ B_{j,t}<q^+\bigr)+\Pp\bigl(\widehat B_{j,t}=b,\ B_{j,t}\ge q^+\bigr)\\
		&=\xi_{j,b}\,\mathbb I\{b<q^+\}+\Pp(B_{j,t}\ge q^+)\cdot\frac{\xi_{j,b}}{\sum_{b'\ge q^+}\xi_{j,b'}}\cdot\mathbb I\{b\ge q^+\}
		=\xi_{j,b},
	\end{align*}
	where the last equality uses $\Pp(B_{j,t}\ge q^+)=\sum_{b'\ge q^+}\xi_{j,b'}$.
	Finally, $\widehat B_{j,t}$ is a function of $B_{j,t}$ and of the draw in Line~18 for
	service~$j$, and these are independent across services.
\end{proof}

By part~(a), the number of customers served at service $j$ in period $t$ is
$Q_j^+(t)-Q_j(t+1)=\min\{Q_j^+(t),\widehat B_{j,t}\}$. Hence, the actual system is unchanged if
its service capacities $B_{j,t}$ are replaced by $\widehat B_{j,t}$. By part~(b), the
$\widehat B_{j,t}$ have the same law as the $B_{j,t}$. In the rest of the proof, both the actual and virtual systems are therefore served with the capacities $\widehat B_{j,t}$.

\begin{lemma}
	\label{lem:path-dom}
	For every service $j$ and period $t\in[T+1]$, we have $Q_j(t)\le\widehat Q_j(t)$ almost surely.
	Consequently, the retention probability in Line~10 lies in $[0,1]$, so {\LPGA} is well
	defined, and
	\begin{equation}
		I_{j,t}\bigl(Q_j(t)\bigr)\ge I_{j,t}\bigl(\widehat Q_j(t)\bigr),
		\qquad j\in[n],\ t\in[T].
		\label{eq:index-dom}
	\end{equation}
\end{lemma}

\begin{proof}
	We first show by induction on $t$ that, for every $t\in[T+1]$, we have $Q_j(t)\le\widehat Q_j(t)$ for all $j\in[n]$. It holds at $t=1$, since $Q_j(1)=\widehat Q_j(1)=0$ for every $j$. Now, we fix $ t \in [T] $ and assume that $Q_j(t)\le\widehat Q_j(t)$ for all $j\in[n]$ at the beginning of time period $t$. Then, we can see that Line~10 is well defined in period $t$. Particularly, Line~10 uses the retention
	probability $\bar a_{j,t}(\widehat Q_j(t))/\bar a_{j,t}(Q_j(t))$ only when
	$\bar a_{j,t}(Q_j(t))>0$. In that case, the retention probability lies in $[0,1]$,
	because $a_j(\cdot)$ is nonincreasing and $Q_j(t)\le\widehat Q_j(t)$ by the induction
	hypothesis.
	
	We now show that $Q_j(t+1)\le\widehat Q_j(t+1)$ for all $j\in[n]$. We first establish that
	$Q_j(t)+A_{j,t}\le\widehat Q_j(t)+\widehat A_{j,t}$ for every $j \in [n]$. We consider three
	cases for each $j$:
	\begin{itemize}
		\item If $j\ne J_t$, then $A_{j,t}=0$ and thus $Q_j(t)+A_{j,t} = Q_j(t) \leq \widehat Q_j(t) \leq \widehat Q_j(t)+\widehat A_{j,t}$.
		\item If $j=J_t$ and $Q_j(t)=\widehat Q_j(t)$, then $\widehat A_{j,t}=A_{j,t}$ by Line~10 and thus $Q_j(t)+A_{j,t} = \widehat Q_j(t)+\widehat A_{j,t}$.
		\item If $j=J_t$ and $Q_j(t)<\widehat Q_j(t)$, then $Q_j(t) + 1 \leq \widehat Q_j(t)$ and thus
		$Q_j(t)+A_{j,t}\le Q_j(t)+1\le\widehat Q_j(t)\le\widehat Q_j(t)+\widehat A_{j,t}$.
	\end{itemize}
	By Lemma~\ref{lem:capacity-reconstruction}(a) and Line~19, both queues are served with the
	same capacity $\widehat B_{j,t}$, so
	\[
	Q_j(t+1)
	=
	\bigl(Q_j(t)+A_{j,t}-\widehat B_{j,t}\bigr)_+
	\le
	\bigl(\widehat Q_j(t)+\widehat A_{j,t}-\widehat B_{j,t}\bigr)_+
	=
	\widehat Q_j(t+1).
	\]
	This proves $Q_j(t+1)\le\widehat Q_j(t+1)$ for all $j\in[n]$ and completes the induction. Since the inductive step covers all period $t\in[T]$, Line~10 is well defined in every period, so is {\LPGA}. Finally, $a_j(q)$ is nonincreasing in $q$ by its monotonicity, and $\kappa_{j,q,t}$ is
	nonincreasing in $q$ by its definition~\eqref{eq:kappa}. Hence,
	$I_{j,t}(q)=r_j\lambda_t a_j(q)\kappa_{j,q,t}$ is nonincreasing in $q$. Then,
	$Q_j(t)\le\widehat Q_j(t)$ gives~\eqref{eq:index-dom}.
\end{proof}

\subsubsection{Step 2: Connect the Virtual Queues to \ref{eq:MOLP}}
\label{subsec:LP-proof-Step2}

We next show that the virtual queues constructed in {\LPGA} have exactly the marginals $\bm y^*$ planned by
\ref{eq:MOLP} and are independent across services. We first characterize the virtual joining
indicators.

\begin{lemma}
	\label{lemma:virtual-arrival}
	For every period $t\in[T]$, conditional on $\Fcal_t$, the virtual joining indicators
	$\widehat A_{1,t},\ldots,\widehat A_{n,t}$ are independent, and
	\begin{equation}
		\widehat A_{j,t}\,\big|\,\Fcal_t
		\sim
		\operatorname{Bernoulli}\!\left(
		\theta_{j,\widehat Q_j(t),t}\,\bar a_{j,t}\bigl(\widehat Q_j(t)\bigr)
		\right),
		\qquad j\in[n].
		\label{eq:ahat-kernel}
	\end{equation}
	In particular, this conditional law depends on $\Fcal_t$ only through $\widehat{\bm Q}(t)$.
\end{lemma}

\begin{proof}
	To simplify notation, we write $\hat p_j:=\bar a_{j,t}(\widehat Q_j(t))$ and
	$\hat\theta_j:=\theta_{j,\widehat Q_j(t),t}$. We first condition on $\Fcal_t$ and on the
	activation vector $\bm\zeta_t=(\zeta_{j,t})_{j\in[n]}$. Under this conditioning, $S_t$ is
	fixed, and so is $J_t$, because it depends only on $S_t$, $\bm Q(t)$, and the deterministic
	tie-breaking rule. For each $j\in[n]$, we have the following.
	\begin{itemize}
		\item If $j=J_t$ and $\bar a_{j,t}(Q_j(t))>0$, a customer arrives and joins service $j$
		with probability $\bar a_{j,t}(Q_j(t))$, independently of $\Fcal_t$ and $\bm\zeta_t$. Line~10 retains this customer with probability
		$\frac{\hat p_j}{\bar a_{j,t}(Q_j(t))}$, independently. Thus, $\widehat A_{j,t}=1$ with
		probability $\bar a_{j,t}(Q_j(t)) \cdot \frac{\hat p_j}{\bar a_{j,t}(Q_j(t))} =  \hat p_j$.
		\item If $j=J_t$ and $\bar a_{j,t}(Q_j(t))=0$, Line~10 sets $\widehat A_{j,t}=0$. Since
		$Q_j(t)\le\widehat Q_j(t)$ by Lemma~\ref{lem:path-dom} and $a_j(\cdot)$ is nonincreasing,
		$\hat p_j=0$ as well, so again $\widehat A_{j,t}=1$ with probability $\hat p_j$.
		\item If $j\in S_t\setminus\{J_t\}$, Line~12 draws $\widehat A_{j,t}=1$ with probability
		$\hat p_j$.
		\item If $j\notin S_t$, that is, $\zeta_{j,t}=0$, then $\widehat A_{j,t}=0$.
	\end{itemize}
	Hence, for every $b_j\in\{0,1\}$,
	\begin{equation}
		\Pp\!\left(\widehat A_{j,t}=b_j\,\middle|\,\Fcal_t,\bm\zeta_t\right)
		=
		(\zeta_{j,t}\hat p_j)^{b_j}(1-\zeta_{j,t}\hat p_j)^{1-b_j},
		\label{eq:virtual-accepted-arrival-distri-under-conditioning}
	\end{equation}
	with the convention $0^0=1$. Moreover, conditional on $\Fcal_t$ and $\bm\zeta_t$, the indicators
	$\widehat A_{1,t},\ldots,\widehat A_{n,t}$ are independent. Specifically, $\widehat A_{J_t,t}$ is
	determined by the actual arrival and joining outcome of period $t$ and the retention draw in
	Line~10. Each $\widehat A_{j,t}$ with $j\in S_t\setminus\{J_t\}$ is determined by its own draw in Line~12.
	For $j\notin S_t$, we have $\widehat A_{j,t}=0$.

	By Line~3, conditional on $\Fcal_t$, the activations $\zeta_{j,t}$ are independent with
	$\zeta_{j,t}\sim\operatorname{Bernoulli}(\hat\theta_j)$. Hence, for every
	$\bm b\in\{0,1\}^n$,
	\begin{align*}
		\Pp\!\left(\widehat{\bm A}_t=\bm b\,\middle|\,\Fcal_t\right)
		&=
		\sum_{\bm z\in\{0,1\}^n}
		\Pp\!\left(\widehat{\bm A}_t=\bm b\,\middle|\,\Fcal_t,\bm\zeta_t=\bm z\right)
		\cdot
		\Pp\!\left(\bm\zeta_t=\bm z\,\middle|\,\Fcal_t\right)
		\\
		&=
		\sum_{\bm z\in\{0,1\}^n}
		\prod_{j=1}^n
		(z_j\hat p_j)^{b_j}(1-z_j\hat p_j)^{1-b_j}
		\cdot
		\hat\theta_j^{z_j}(1-\hat\theta_j)^{1-z_j}
		&
		=
		\prod_{j=1}^n
		(\hat\theta_j\hat p_j)^{b_j}(1-\hat\theta_j\hat p_j)^{1-b_j}.
	\end{align*}
	The last equality follows by writing the sum over $\bm z$ as a product of sums over each $z_j\in\{0,1\}$, and then evaluating each factor separately for $b_j=1$ and $b_j=0$.
	This product of Bernoulli probability mass functions establishes both the claimed
	conditional distribution~\eqref{eq:ahat-kernel} and the independence of
	$\widehat A_{1,t},\ldots,\widehat A_{n,t}$ conditional on $\Fcal_t$ alone. Since
	$\hat\theta_j$ and $\hat p_j$ depend only on $\widehat Q_j(t)$, so does this conditional law.
\end{proof}

\begin{lemma}
	\label{lem:virtual-marginals}
	For every $j\in[n]$, $t\in[T+1]$, and $q\in[t-1]_0$, we have
	$\Pp\bigl(\widehat Q_j(t)=q\bigr)=y^*_{j,q,t}$. Moreover, for every $t\in[T+1]$, the virtual
	queues $\widehat Q_1(t),\ldots,\widehat Q_n(t)$ are independent.
\end{lemma}

\begin{proof}
	For $j\in[n]$, $t\in[T]$, and $q\in\Zp$, let $p_{j,q,t}:=\theta_{j,q,t}\,\bar a_{j,t}(q)$, and
	define the one-period transition probabilities
	\[
	P_{j,t}(q,q')
	:=
	(1-p_{j,q,t})\,K_j(q,q')+p_{j,q,t}\,K_j(q+1,q'),
	\qquad q,q'\in\Zp.
	\]
	By Line~19, $\widehat Q_j(t+1)=(\widehat Q_j(t)+\widehat A_{j,t}-\widehat B_{j,t})_+$.
	Conditional on $\Fcal_t$, the indicators $\widehat A_{j,t}$ are independent with
	$\widehat A_{j,t}\sim\operatorname{Bernoulli}(p_{j,\widehat Q_j(t),t})$
	(Lemma~\ref{lemma:virtual-arrival}), and conditional on $\Fcal_t$ and $\widehat{\bm A}_t$,
	the capacities $\widehat B_{j,t}$ are independent with distribution $\{\xi_{j,b}\}_{b\ge0}$
	(Lemma~\ref{lem:capacity-reconstruction}(b)). Since $K_j(q,q')=\Pp((q-B_{j,t})_+=q')$, it
	follows that, for every $\bm q'\in\Zp^n$,
	\begin{equation}
		\Pp\!\left(\widehat{\bm Q}(t+1)=\bm q'\,\middle|\,\Fcal_t\right)
		=
		\prod_{j=1}^n P_{j,t}\bigl(\widehat Q_j(t),q'_j\bigr).
		\label{eq:virtual-transition}
	\end{equation}

	We now prove both statements in the lemma by induction on $t$. At $t=1$, all virtual queues are zero, so
	$\Pp(\widehat Q_j(1)=0)=1=y^*_{j,0,1}$, and they are independent. Suppose that both
	statements hold at period $t\in[T]$. Since each virtual queue grows by at most one customer
	per period, $\widehat Q_j(t)\in[t-1]_0$ and $\widehat Q_j(t+1)\in[t]_0$.

	\noindent\underline{Marginals.} Fix $j$ and $q'\in[t]_0$. Summing~\eqref{eq:virtual-transition} over the queue lengths of the other services, each of which contributes a factor $\sum_{q''}P_{k,t}(\widehat Q_k(t),q'')=1$, gives \mbox{$\Pp(\widehat Q_j(t+1)=q'\mid\Fcal_t)=P_{j,t}(\widehat Q_j(t),q')$}. Taking expectations gives
	\begin{align*}
		\Pp\!\left(\widehat Q_j(t+1)=q'\right)
		&=
		\sum_{q=0}^{t-1}\Pp\!\left(\widehat Q_j(t)=q\right)P_{j,t}(q,q')\\
		& =
		\sum_{q=0}^{t-1}y^*_{j,q,t}
		\Bigl[(1-p_{j,q,t})\,K_j(q,q')+p_{j,q,t}\,K_j(q+1,q')\Bigr]
		\\
		&=
		\sum_{q=0}^{t-1}
		\Bigl[
		\bigl(y^*_{j,q,t}-\bar a_{j,t}(q)\,x^*_{j,q,t}\bigr)K_j(q,q')
		+\bar a_{j,t}(q)\,x^*_{j,q,t}\,K_j(q+1,q')
		\Bigr] =
		y^*_{j,q',t+1},
	\end{align*}
	where the second equality uses the induction hypothesis, the third uses
	$y^*_{j,q,t}\,\theta_{j,q,t}=x^*_{j,q,t}$ (by~\eqref{eq:theta} and Constraint~(1), since
	$x^*_{j,q,t}=0$ whenever $y^*_{j,q,t}=0$), and the last is Constraint~(3) of \ref{eq:MOLP}.

	\noindent\underline{Independence.} For every $\bm q'\in\Zp^n$, taking expectations
	in~\eqref{eq:virtual-transition} gives
	\[
	\Pp\!\left(\widehat{\bm Q}(t+1)=\bm q'\right)
	=
	\E\!\left[\prod_{j=1}^n P_{j,t}\bigl(\widehat Q_j(t),q'_j\bigr)\right]
	=
	\prod_{j=1}^n\E\!\left[P_{j,t}\bigl(\widehat Q_j(t),q'_j\bigr)\right]
	=
	\prod_{j=1}^n\Pp\!\left(\widehat Q_j(t+1)=q'_j\right),
	\]
	where the second equality uses the independence of
	$\widehat Q_1(t),\ldots,\widehat Q_n(t)$ and the last equality holds because
	$\E[P_{j,t}(\widehat Q_j(t),q'_j)]=\Pp(\widehat Q_j(t+1)=q'_j)$, as in the equality of
	the marginal computation. Hence,
	$\widehat Q_1(t+1),\ldots,\widehat Q_n(t+1)$ are independent, which completes the induction.
\end{proof}

Lemma~\ref{lem:virtual-marginals} establishes properties~(i) and~(ii) of the virtual system in
Section~\ref{subsec:LP-main-theorem}: its queues have exactly the marginals $\bm y^*$ planned
by \ref{eq:MOLP} and are independent across services.

\subsubsection{Step 3: Compare the Net Revenue of {\LPGA} with the LP Objective}
\label{subsubsec:LP-proof-step3-revenue}

For each service $j$ and period $t$, let $C_{j,t}$ be the indicator that a customer joins
service $j$ in period $t$ and completes service by the end of period $T$. The net revenue of {\LPGA} attributable to period $t$ is
$G_t^{\ALG}:=\sum_{j=1}^n r_jC_{j,t}$, and $R^{\ALG}=\sum_{t=1}^TG_t^{\ALG}$.

\paragraph{Revenue of {\LPGA} in period $t$.} The next lemma states that, given the activations,
{\LPGA} earns in expectation the largest index among the activated services at their
actual queue lengths.

\begin{lemma}
	\label{lem:completion-reward-max}
	For every period $t$, we have $\E\!\left[G_t^{\ALG}\,\middle|\,\Fcal_t,\bm\zeta_t\right]
	=
	\max_{j\in[n]}\ \zeta_{j,t}\,I_{j,t}\bigl(Q_j(t)\bigr)$.
\end{lemma}

\begin{proof}
	Condition on $\Fcal_t$ and $\bm\zeta_t$. Then $S_t$ and $J_t$ are fixed. By Line~5,
	\[
	I_{J_t,t}\bigl(Q_{J_t}(t)\bigr)=\max_{j\in S_t\cup\{0\}}I_{j,t}\bigl(Q_j(t)\bigr),
	\]
	where $I_{0,t}(\cdot)=0$.
	Since every index is nonnegative and $\zeta_{j,t}=0$ for $j\notin S_t$, this maximum equals
	$\max_{j\in[n]}\zeta_{j,t}I_{j,t}(Q_j(t))$. It therefore suffices to show that
	$\E[G_t^{\ALG}\mid\Fcal_t,\bm\zeta_t]=I_{J_t,t}(Q_{J_t}(t))$.

	Given $\Fcal_t$ and $\bm\zeta_t$, if $J_t=0$, no service is recommended, so $G_t^{\ALG}=0=I_{0,t}(\cdot)$. If $J_t=j\in[n]$, a
	customer arrives and joins service $j$ with probability $\bar a_{j,t}(Q_j(t))$,
	independently of $\Fcal_t$ and $\bm\zeta_t$. Under FIFO, this customer has $Q_j(t)$
	customers ahead and all later customers join behind, so the customer completes service by the
	end of period $T$ if and only if $\sum_{\tau=t}^T\widehat B_{j,\tau}\ge Q_j(t)+1$. By
	Lemma~\ref{lem:capacity-reconstruction}(b), applied in periods $t,\ldots,T$, the capacities
	$\widehat B_{j,t},\ldots,\widehat B_{j,T}$ are independent with distribution
	$\{\xi_{j,b}\}_{b\ge0}$, and independent of $\Fcal_t$, $\bm\zeta_t$, and $A_{j,t}$. Hence, the
	customer completes service with probability $\kappa_{j,Q_j(t),t}$, and
	$\E\!\left[G_t^{\ALG}\,\middle|\,\Fcal_t,\bm\zeta_t\right]
	=
	r_j\,\bar a_{j,t}\bigl(Q_j(t)\bigr)\,\kappa_{j,Q_j(t),t}
	=
	I_{j,t}\bigl(Q_j(t)\bigr).$
\end{proof}

\paragraph{Comparison with the virtual system.} For each service $j$, define
\[
Y_j:=\zeta_{j,t}\,I_{j,t}\bigl(\widehat Q_j(t)\bigr),
\]
the index of service $j$ at its virtual queue length if it is activated, and $0$ otherwise.
By~\eqref{eq:index-dom} in Lemma~\ref{lem:path-dom}, we have $\zeta_{j,t}I_{j,t}(Q_j(t))\ge Y_j$ for every $j$. Hence, by
Lemma~\ref{lem:completion-reward-max}, we establish that
\begin{equation}
	\E\!\left[G_t^{\ALG}\,\middle|\,\Fcal_t,\bm\zeta_t\right]
	\ge
	\max_{j\in[n]}Y_j.
	\label{eq:maximum-comparison}
\end{equation}
It remains to bound $\E[\max_{j\in[n]}Y_j]$ from below. We first establish three properties of $Y_1,\ldots,Y_n$.

\begin{lemma}
	\label{lem:Y-properties}
	For every period $t$, the following hold: %
	\begin{enumerate}[label=(\alph*)]
		\item $Y_1,\ldots,Y_n$ are independent.
		\item $\sum_{j=1}^n\Pp(Y_j>0)\le1$.
		\item $\sum_{j=1}^n\E[Y_j]
		=
		\sum_{j=1}^n\sum_{q=0}^{t-1}I_{j,t}(q)\,x^*_{j,q,t}$, the contribution of period $t$ to
		the objective of \ref{eq:MOLP}.
	\end{enumerate}
\end{lemma}

\begin{proof}
	By Line~3, conditional on $\Fcal_t$, the activations $\zeta_{1,t},\ldots,\zeta_{n,t}$ are
	independent with $\zeta_{j,t}\sim\operatorname{Bernoulli}(\theta_{j,\widehat Q_j(t),t})$.
	Hence, for every $\bm q\in\Zp^n$ and $\bm z\in\{0,1\}^n$,
	{\small
	\[
	\Pp\!\left(\widehat{\bm Q}(t)=\bm q,\ \bm\zeta_t=\bm z\right)
	=
	\Pp\!\left(\widehat{\bm Q}(t)=\bm q\right)
	\prod_{j=1}^n\theta_{j,q_j,t}^{z_j}(1-\theta_{j,q_j,t})^{1-z_j} =
	\prod_{j=1}^n
	\Pp\!\left(\widehat Q_j(t)=q_j\right)
	\theta_{j,q_j,t}^{z_j}(1-\theta_{j,q_j,t})^{1-z_j},
	\]
	}where the second equality uses the independence of $\widehat Q_1(t),\ldots,\widehat Q_n(t)$
	(Lemma~\ref{lem:virtual-marginals}). 
	Thus, the pairs $(\widehat Q_1(t),\zeta_{1,t}),\ldots,(\widehat Q_n(t),\zeta_{n,t})$ are mutually independent across services. Part~(a) follows as $Y_j$ is a function of $(\widehat Q_j(t),\zeta_{j,t})$. 
	
	For
	part~(b), we first have $\{Y_j>0\}\subseteq\{\zeta_{j,t}=1\}$. Thus,
	\[
	\sum_{j=1}^n\Pp(Y_j>0) \le \sum_{j=1}^n \Pp\!\left( \zeta_{j,t}=1 \right) = 
	\sum_{j=1}^n\sum_{q=0}^{t-1}\Pp\!\left(\widehat Q_j(t)=q,\ \zeta_{j,t}=1\right)
	=
	\sum_{j=1}^n\sum_{q=0}^{t-1}x^*_{j,q,t}
	\le
	1,
	\]
	where the second equality follows since $\Pp(\widehat Q_j(t)=q,\ \zeta_{j,t}=1)=y^*_{j,q,t}\,\theta_{j,q,t}=x^*_{j,q,t}$ by
	Lemma~\ref{lem:virtual-marginals} and~\eqref{eq:theta}, and the last inequality is Constraint~(2) of \ref{eq:MOLP}. For part~(c),
	\[\E[Y_j]=\sum_{q=0}^{t-1}\Pp(\widehat Q_j(t)=q,\ \zeta_{j,t}=1)\,I_{j,t}(q)
	=\sum_{q=0}^{t-1}x^*_{j,q,t}\,I_{j,t}(q). 	\qedhere
	\]
\end{proof}

\paragraph{Correlation gap.} By Lemma~\ref{lem:Y-properties}(a) and~(b), the values
$Y_1,\ldots,Y_n$ are independent, and the expected number of positive values among them is at most one. Under
these two conditions, their maximum retains a $(1-1/e)$ fraction of their sum in expectation.
This is the rank-one case of the correlation-gap bound for weighted matroid rank functions
\citep[Lemma~4.1]{yan2011mechanism}; we include a short proof in the form needed here.

\begin{lemma}
	\label{lem:appendix-corr-gap}
	Let $W_1,\ldots,W_n$ be independent nonnegative random variables with
	$\sum_{j=1}^n\Pp(W_j>0)\le1$. Then
	$\E\bigl[\max_{j\in[n]}W_j\bigr]\ge(1-1/e)\sum_{j=1}^n\E[W_j]$.
\end{lemma}

\begin{proof}
	For $z>0$, let $p_j(z):=\Pp(W_j\ge z)$ and $s(z):=\sum_{j=1}^np_j(z)$. Since
	$\{W_j\ge z\}\subseteq\{W_j>0\}$, we have $0\le s(z)\le\sum_{j=1}^n\Pp(W_j>0)\le1$. By
	independence and $1-x\le e^{-x}$,
	\[
	\Pp\Bigl(\max_{j\in[n]}W_j\ge z\Bigr)
	=
	1-\prod_{j=1}^n\bigl(1-p_j(z)\bigr)
	\ge
	1-e^{-s(z)}
	\ge
	\Bigl(1-\frac1e\Bigr)s(z),
	\]
	where the last inequality uses the concavity of $s\mapsto1-e^{-s}$ on $[0,1]$. Integrating over $z>0$ and using $\E[W]=\int_0^\infty\Pp(W\ge z)\,dz$ proves the lemma.
\end{proof}

\paragraph{Completing the proof of Theorem~\ref{thm:LPGA_approx_ratio}.} Fix a period $t$. Taking expectations
in~\eqref{eq:maximum-comparison}, and then applying Lemma~\ref{lem:appendix-corr-gap} to
$Y_1,\ldots,Y_n$, which is allowed by Lemma~\ref{lem:Y-properties}(a) and~(b), gives
\begin{equation}
	\E\bigl[G_t^{\ALG}\bigr]
	\ge
	\E\Bigl[\max_{j\in[n]}Y_j\Bigr]
	\ge
	\Bigl(1-\frac1e\Bigr)\sum_{j=1}^n\E[Y_j]
	=
	\Bigl(1-\frac1e\Bigr)\sum_{j=1}^n\sum_{q=0}^{t-1}I_{j,t}(q)\,x^*_{j,q,t},
	\label{eq:Ymean-period-LP}
\end{equation}
where the equality is Lemma~\ref{lem:Y-properties}(c). Summing over $t\in[T]$ and using
$R^{\ALG}=\sum_{t=1}^TG_t^{\ALG}$ gives $\E[R^{\ALG}]\ge(1-\frac{1}{e})\,Z^{\LP}$. Since
$Z^{\LP}\ge\OPT$ by Proposition~\ref{prop:LP_upperbound}, this completes the theorem.

\subsection{Performance Guarantees with Nonempty Initial Queues}
\label{appendix-subsec:initial-queues}

We show that Theorems~\ref{theorem:greedy} and~\ref{thm:LPGA_approx_ratio} continue to hold with the nonempty initial queues of Section~\ref{subsec:impatience}, under the monotonicity condition~\eqref{eq:monotone-a}. Both proofs rely on the per-period form of the objective in Section~\ref{subsec:general-model}, which remains valid: the completion factor $\kappa_{j,q,t}$ depends only on the position $q+1$ of a joining customer, not on whether the customers ahead joined during the horizon, and initial customers contribute no revenue.

\paragraph{The greedy policy.} We follow the notation and arguments in Appendix~\ref{appendix-subsec:proof-of-greedy}. Case~(1) uses only the per-period form and the monotonicity of $I_{j,t}(\cdot)$, so it carries over unchanged. In Case~(2), we must verify that the customer $c_t$ in the greedy system is not an initial customer; otherwise, $c_t$ would contribute no revenue to the greedy policy. Let $\bar Q_j(t)$ denote the number of initial customers remaining at service $j$ at the beginning of period $t$. Under FIFO service, the initial customers are at the front of the queue, so $\bar Q_j(t+1)=(\bar Q_j(t)-B_{j,t})_+$. Because the two systems share the service capacities $\{B_{j,t}\}$, $\bar Q_j(t)$ is the same in both systems. The remaining initial customers are part of the queue in the $\pi$ system, so $Q^{\pi}_j(t)\ge\bar Q_j(t)$. Hence, the position of $c_t$ in the greedy system satisfies $Q^{\pi}_j(t)+1>\bar Q_j(t)$; that is, $c_t$ is behind every remaining initial customer and therefore joined during the horizon. The rest of the argument is unchanged.

\paragraph{The {\LPGA} policy.} We replace the initial-state constraints~(4)--(5) of~\ref{eq:MOLP} by $y_{j,q,1}=\mathbb{I}\{q=Q_j(1)\}$ for all $j$ and $q$, and let $q$ range over $[Q_j(1)+t-1]_0$ in period $t$; because $\max_j Q_j(1)$ is polynomially bounded in $T$, the LP remains of polynomial size. In the proof of Proposition~\ref{prop:LP_upperbound}, the occupancy probabilities of any policy satisfy the new initial-state constraints, and the objective identity holds by the per-period form above. Thus, $Z^{\LP}$ remains an upper bound on $\OPT$. {\LPGA} initializes its virtual queues at $\widehat Q_j(1)=Q_j(1)$. The proofs of Lemmas~\ref{lem:capacity-reconstruction}--\ref{lem:appendix-corr-gap} then apply with this base case and with $q$ ranging over $[Q_j(1)+t-1]_0$ in period $t$.

\subsection{Proof of Theorem~\ref{thm:nonmonotone-hardness}}
\label{subsec:proof-MIS-reduction}

We reduce the Maximum Independent Set problem (Max-IS) to \ref{problem:DSR-abstract} with
nonempty initial queues and nonmonotone joining probabilities. For $\rho\ge1$, we say that a
maximization problem is NP-hard to approximate within a factor of $\rho$ if, unless
$\text{P}=\text{NP}$, no polynomial-time algorithm computes for every instance a value $V$ with
$\OPT/\rho\le V\le\OPT$, where $\OPT$ is the optimal objective value of the instance. For every
constant $\epsilon>0$, Max-IS on graphs with $n$ vertices is NP-hard to approximate within a
factor of $n^{1-\epsilon}$: \citet{haastad1999clique} established this hardness under the
assumption $\text{NP}\ne\text{ZPP}$, and \citet{zuckerman2007linear} showed that it is NP-hard. Both state the result for Max-Clique, which is
equivalent to Max-IS on the complement graph.
We map each graph to an instance of \ref{problem:DSR-abstract} with one service per vertex,
whose optimal net revenue determines the size of a maximum independent set up to a small
additive error.

\paragraph{Maximum independent set problem.}
Let $G=(V,E)$ be an instance of Max-IS, where $|V|=n$ and $|E|=m$. A subset $S\subseteq V$ is
an \emph{independent set} if no two distinct vertices in $S$ are adjacent; that is,
$(u,v)\notin E$ for all distinct $u,v\in S$. The objective of Max-IS is to find an independent
set of maximum cardinality, whose size we denote by $\alpha(G)$. We assume $m\ge1$; otherwise,
$V$ itself is a maximum independent set.

We fix an arbitrary ordering of the edges, denoted by $e_1,e_2,\ldots,e_m$. For each vertex
$v\in V$, let $d_v=\deg(v)$ and enumerate the edges incident to $v$ as
$e_{v,1},e_{v,2},\ldots,e_{v,d_v}$, in the global edge ordering. For each $k\in[d_v]$, let
$\tau_{v,k}\in[m]$ denote the position of $e_{v,k}$ in the global edge ordering; equivalently,
$e_{v,k}=e_{\tau_{v,k}}$.

\paragraph{Constructed instance of \ref{problem:DSR-abstract}.}
We construct an instance of \ref{problem:DSR-abstract} with $n$ services and horizon
$T=m+nL$, where $L:=4mn$. Each service corresponds to a vertex in $V$. We set $\lambda_t=1$ for
all $t$, so that exactly one customer arrives in every period, and set $r_j=1$ for every
$j\in[n]$. All services have the same deterministic service capacity, $B_{j,t}=C$ for all
$j\in[n]$ and $t\in[T]$, where $C=T+2$. Finally, each service is initialized with
$Q_j(1)=C(T-1)+1$ customers, which we call the \emph{initial customers}. The instance has
polynomial size in $n$ and $m$, and $\max_jQ_j(1)$ is polynomially bounded in $T$.

To describe the evolution of the initial customers, define $\rho_t:=C(T-t)+1$ for $t\in[T]$.
Let $\bar Q_j(t)$ denote the number of initial customers remaining in service $j$ at the
beginning of period $t$. For $t\in[T-1]$, at least $\rho_t\ge C+1$ initial customers are
present at the beginning of period $t$, so exactly $C$ initial customers are served in period
$t$. Hence, $\bar Q_j(t)=Q_j(1)-C(t-1)=\rho_t\ge1$ for $t\in[T]$, so at least one initial
customer is ahead of all newly joined customers at the beginning of every period $t\in[T]$.
In particular, $\bar Q_j(T)=\rho_T=1$: exactly one initial customer remains in each service at
the beginning of period $T$.

For each service $j$ and period $t$, let $A_{j,t}$ be the indicator that a customer arrives in
period $t$ and joins service $j$, and let $\Acal_j(t):=\sum_{t'=1}^tA_{j,t'}$ be the number of
customers who have joined service $j$ through period $t$. Since at least one initial customer
remains ahead of them, no customer who joins during the horizon is served before period $T$.
Consequently,
\begin{equation}
	Q_j(t)=\bar Q_j(t)+\Acal_j(t-1)=\rho_t+\Acal_j(t-1),
	\qquad t\in[T].
	\label{eq:MIS-queue}
\end{equation}
The choice $C=T+2$ ensures that all customers are served in period $T$:
$Q_j(T+1)=(1+\Acal_j(T)-C)_+=0$, because $\Acal_j(T)\le T<C-1$. Therefore, no refunds are
incurred, and the net revenue of any policy equals the total number of customers who join
during the horizon, $\sum_{j\in[n]}\Acal_j(T)$.

\paragraph{Joining probability function.}
For each service $j\in[n]$, let $h_j:=m+L(j-1)+1$. The periods $h_j,\ldots,h_j+L-1$ form the
\emph{block} of service $j$; the blocks of the $n$ services partition the periods
$m+1,\ldots,T$. We define the joining probability function $a_j(\cdot)$ by
\begin{align}
	a_j\bigl(\rho_{\tau_{j,k}}+k-1\bigr)&=1,
	\qquad k\in[d_j], \label{eq:inapprox-nonzero-a-1}
	\\
	a_j\bigl(\rho_{h_j+\ell}+d_j+\ell\bigr)&=1,
	\qquad \ell\in[L-1]_0, \label{eq:inapprox-nonzero-a-2}
\end{align}
and $a_j(q)=0$ for all other $q\in\Zp$. The first line~\eqref{eq:inapprox-nonzero-a-1} opens service $j$ in the time periods
$\tau_{j,k}\in[m]$ of its incident edges; the second line~\eqref{eq:inapprox-nonzero-a-2} opens it in every period
$h_j+\ell\in\{m+1,\ldots,T\}$ of its block.

By~\eqref{eq:MIS-queue}, the queue length $Q_j(t)=\rho_t+\Acal_j(t-1)$ records both the
current period $t$ and the number $\Acal_j(t-1)$ of customers who have joined service $j$
before period $t$: since $Q_j(t)-1=C(T-t)+\Acal_j(t-1)$ and $0\le\Acal_j(t-1)\le t-1<C$,
dividing $Q_j(t)-1$ by $C$ gives quotient $T-t$ and remainder $\Acal_j(t-1)$.
The queue
lengths in~\eqref{eq:inapprox-nonzero-a-1} and~\eqref{eq:inapprox-nonzero-a-2} decode in
the same way, to the pairs $(\tau_{j,k},k-1)$ and $(h_j+\ell,d_j+\ell)$, because
$k-1<d_j\le n<C$ and $d_j+\ell<n+L\le T<C$. Since $\lambda_t=1$, this gives the following
claim.

\begin{claim}
	\label{claim:MIS-clock}
	Under any policy, for every service $j$ and period $t\in[T]$, a customer recommended service
	$j$ in period $t$ joins with probability one if
	\begin{enumerate}[label=(\alph*)]
		\item $t=\tau_{j,k}$ for some $k\in[d_j]$ and $\Acal_j(t-1)=k-1$, or
		\item $t=h_j+\ell$ for some $\ell\in[L-1]_0$ and $\Acal_j(t-1)=d_j+\ell$,
	\end{enumerate}
	and joins with probability zero otherwise.
\end{claim}

\paragraph{Approximation-preserving mapping.}
We establish the following two claims.

\begin{claim}
	\label{claim:MIS-DSR-map-1}
	For any independent set $S\subseteq V$, there exists a policy $\pi_S$ for the constructed
	instance such that $\Rcal(\pi_S)\ge L|S|$.
\end{claim}

\begin{claim}
	\label{claim:MIS-DSR-map-2}
	For any deterministic policy $\pi$ for the constructed instance, the set
	$S_\pi:=\{j\in[n]:\Acal_j(m)=d_j\}$ is an independent set, and
	$\Rcal(\pi)\le m+L|S_\pi|$.
\end{claim}

Let $\OPT$ denote the optimal net revenue of
the constructed instance. Since its state and action spaces are finite, $\OPT$ is attained by a
deterministic Markov policy $\pi^*$. Applying Claim~\ref{claim:MIS-DSR-map-1} to a maximum
independent set $S^*$ and Claim~\ref{claim:MIS-DSR-map-2} to $\pi^*$ gives
\begin{equation}
	L\,\alpha(G)=L|S^*|\le\Rcal(\pi_{S^*})\le\OPT=\Rcal(\pi^*)\le m+L|S_{\pi^*}|
	\le m+L\,\alpha(G),
	\label{eq:MIS-OPT-bounds}
\end{equation}
where the last inequality holds because $S_{\pi^*}$ is an independent set.

\paragraph{Completing the proof.} Fix $\epsilon>0$ and a constant $c\ge1$, and suppose that a
polynomial-time algorithm \texttt{ALG} computes, for every instance of
\ref{problem:DSR-abstract} with nonempty initial queues, a value $V$ with
$\OPT/(c\,n^{1-\epsilon})\le V\le\OPT$. Given a graph $G$ with $m\ge1$, we construct the
instance above in polynomial time, run \texttt{ALG}, and define $W:=V/L-1/(4n)$. Then,
by~\eqref{eq:MIS-OPT-bounds} and $m/L=1/(4n)$,
\[
W\le\frac{\OPT}{L}-\frac{m}{L}\le\alpha(G)
\]
and
\[
W\ge\frac{\OPT}{c\,L\,n^{1-\epsilon}}-\frac{1}{4n}
\ge\frac{\alpha(G)}{c\,n^{1-\epsilon}}-\frac14\cdot\frac{\alpha(G)}{c\,n^{1-\epsilon}}
=\frac34\cdot\frac{\alpha(G)}{c\,n^{1-\epsilon}}
\ge\frac{\alpha(G)}{n^{1-\epsilon/2}},
\]
where the second inequality uses $\alpha(G)\ge1$ and $c\,n^{1-\epsilon}\le n$, and the last
uses $n^{\epsilon/2}\ge4c/3$; both hold for all $n\ge(4c/3)^{2/\epsilon}$. Hence, for all
sufficiently large $n$, \texttt{ALG} would yield a polynomial-time algorithm that
approximates $\alpha(G)$ within a factor of $n^{1-\epsilon/2}$, which is NP-hard. Therefore,
unless P${}={}$NP, no such algorithm \texttt{ALG} exists. This proves
Theorem~\ref{thm:nonmonotone-hardness}.

\paragraph{Proof of Claim~\ref{claim:MIS-DSR-map-1}.} Fix an independent set $S\subseteq V$.
The policy $\pi_S$ operates as follows. In period $t\in[m]$, if edge $e_t$ is incident to a
vertex $j\in S$, then $\pi_S$ recommends service $j$; otherwise, it makes no recommendation.
This is well defined because $S$ is an independent set, so no edge has both endpoints in $S$.
In every period $h_j+\ell$, $\ell\in[L-1]_0$, of the block of a service $j\in S$, $\pi_S$
recommends service $j$; in every period of the block of a service $j\notin S$, it makes no
recommendation.

Fix $j\in S$. During the first $m$ periods, $\pi_S$ recommends service $j$ exactly in periods
$\tau_{j,1}<\cdots<\tau_{j,d_j}$. We show by induction on $k$ that each of these
recommendations results in a join. If the first $k-1$ recommendations resulted in joins, then
$\Acal_j(\tau_{j,k}-1)=k-1$, so by Claim~\ref{claim:MIS-clock}(a) the customer in period
$\tau_{j,k}$ joins. Hence, $\Acal_j(m)=d_j$. Service $j$ is not recommended in periods
$m+1,\ldots,h_j-1$, so $\Acal_j(h_j-1)=d_j$. By the same induction on $\ell\in[L-1]_0$, using
Claim~\ref{claim:MIS-clock}(b), every recommendation in the block of service $j$ results in a
join. Thus, each service $j\in S$ receives $d_j+L$ joins, and
$\Rcal(\pi_S)=\sum_{j\in S}(d_j+L)\ge L|S|$.

\paragraph{Proof of Claim~\ref{claim:MIS-DSR-map-2}.} Fix a deterministic policy $\pi$. Since
$\lambda_t=1$, $a_j(\cdot)\in\{0,1\}$ and the capacities are deterministic, the constructed
instance evolves deterministically under $\pi$; hence $S_\pi$ is a fixed set and $\Rcal(\pi)$
is the number of customers who join during the horizon. Write $\Acal^m_j:=\Acal_j(m)$. By Claim~\ref{claim:MIS-clock}, during the first $m$ periods, a
customer can join service $j$ only in periods $\tau_{j,1},\ldots,\tau_{j,d_j}$. Hence,
$\Acal^m_j\le d_j$, with equality only if a customer joins service $j$ in each of these
periods.

We first show that $S_\pi$ is an independent set. Suppose that two adjacent vertices $i$ and
$j$ both belong to $S_\pi$, and let $e_t=(i,j)$ be their common edge. Then a customer joins
service $i$ and a customer joins service $j$ in period $t$. This is impossible, because exactly
one customer arrives in period $t$ and receives at most one recommendation. Hence, $S_\pi$ is
an independent set.

We next bound the number of joins in periods $m+1,\ldots,T$. By Claim~\ref{claim:MIS-clock},
in these periods a customer can join service $j$ only in the block of service $j$. If
$j\in S_\pi$, service $j$ receives at most $L$ joins in its block. If $j\notin S_\pi$, then
$\Acal^m_j<d_j$, and service $j$ receives no join in its block: by induction on
$\ell\in[L-1]_0$, $\Acal_j(h_j+\ell-1)=\Acal^m_j<d_j+\ell$, so condition~(b) of
Claim~\ref{claim:MIS-clock} fails in period $h_j+\ell$. Hence, periods $m+1,\ldots,T$
generate at most $L|S_\pi|$ joins. Since at most one customer joins in each of the first $m$
periods, $\Rcal(\pi)\le m+L|S_\pi|$.

\end{appendices}

\end{document}